\documentclass[12pt,reqno,notitlepage]{amsart}
\usepackage{amsmath,amsfonts,amsthm,amssymb,stmaryrd}
\usepackage{enumerate}
\usepackage{indentfirst}
\usepackage[dvips]{graphicx}
\usepackage[all]{xy}
\usepackage{stackrel}
\usepackage{marginnote}
\usepackage{todonotes}
\usepackage{enumitem}

\usepackage[centering, includeheadfoot, hmargin=1.0in, tmargin=1.0in, 
bmargin=1in, headheight=6pt]{geometry}
\usepackage[latin1]{inputenc}
\usepackage[T1]{fontenc}
\usepackage{verbatim}
\usepackage{color}
\usepackage[normalem]{ulem}
\usepackage{tikz-cd}
\usepackage{hyperref}  
\hypersetup{
pdfborder={0 0 0}, 
colorlinks=true, 
citecolor=blue,
linktoc=page,
pdfauthor={Thiago Fassarella, Wodson Mendson, Fred Touzet, J. Pedro Santos}, 
pdftitle={foliations}
}
\renewcommand{\ref}{\hyperref}

\newtheorem{thm}{Theorem}[section]

\newtheorem{cor}[thm]{Corollary}

\newtheorem{lemma}[thm]{Lemma}
\newtheorem{prop}[thm]{Proposition}

\newtheorem{OBS}[thm]{Remark}
\newtheorem{remark}[thm]{Remark}
\newtheorem{ex}[thm]{Example}
\newtheorem{dfn}[thm]{Definition}

\newtheorem{question}[thm]{Question}

\theoremstyle{definition}
\newtheorem{definition}[thm]{Definition}

\renewcommand{\P}{\mathbb{P}}

\numberwithin{equation}{section}

\DeclareMathOperator{\field}{k}
\DeclareMathOperator{\projn}{\mathbb{P}_{k}^{n}}
\DeclareMathOperator{\proj2}{\mathbb{P}_{k}^{2}}

\DeclareMathOperator{\codim}{codim}
\DeclareMathOperator{\sing}{\textrm{Sing}}

\DeclareMathOperator{\ord}{ord}

\DeclareMathOperator{\h}{H}

\DeclareMathOperator{\Ext}{Ext}
\DeclareMathOperator{\at}{\textbf{at}}
\DeclareMathOperator{\Hom}{Hom}

\DeclareMathOperator{\End}{{End}_{\field}}
\DeclareMathOperator{\hh}{H}

\newcommand{\NN}{\mathbb{N}}

\newcommand{\fr}[2]{\frac{#1}{#2}}
\newcommand{\GL}{\mathrm{GL}}
\newcommand{\id}{\mathrm{id}}

\newcommand {\D}{\mathcal D}
\newcommand{\F}{\mathcal F}

\newcommand{\terminou}{\hfill$\lrcorner$}

\begin{document}
\title{Foliations with small singular set in arbitrary characteristic }
\author[T. Fassarella]{Thiago Fassarella}
\address{Universidade Federal Fluminense, Instituto de Matem\'atica e Estat\'istica.
Rua Alexandre Moura 8, S\~ao Domingos, 24210-200 Niter\'oi RJ,
Brazil.}
\email{\color{black}tfassarella@id.uff.br}
\email{\color{black}wmendson@id.uff.br}

\author[W. Mendson]{Wodson Mendson}

\author[J.P. Santos]{Jo\~ao Pedro dos Santos}
\address{Institut Montpelli\'erain Alexander Grothendieck, Universit\'e de Montpellier, Place Eug\`ene Bataillon, 34090 Montpellier, France.}
\email{\color{black}joao\_pedro.dos\_santos@yahoo.com}

\author[F. Touzet]{Fr\'ed\'eric Touzet}
\address{Univ Rennes, CNRS, IRMAR, UMR 6625, F-35000 Rennes, France.}
\email{\color{black}frederic.touzet@univ-rennes1.fr}

 \date{\today}

 \maketitle

\begin{abstract}
This paper investigates the geometry of foliations on smooth algebraic varieties over an algebraically closed field of arbitrary characteristic $p \ge 0$. We address several specific features of foliations in positive characteristic, aiming to highlight both similarities and differences with the characteristic zero case. First, we provide a complete classification of regular foliations on minimal rational and del Pezzo surfaces, establishing that in positive characteristic, most of them are $p$-closed. However, there are regular foliations on Hirzebruch surfaces over arbitrary characteristic $p>0$, and on weak del Pezzo surfaces, in characteristic $p=2$, which are not $p$-closed.  We extend the Camacho-Sad index and its associated sum formula to arbitrary characteristic, using it to study the behavior of foliations and distributions with small singular set. For foliations on projective spaces, we prove that the existence of an invariant hypersurface with a sufficiently small singular set  forces the foliation to be $p$-closed and imposes strict divisibility conditions on the degree of its normal bundle. Finally, we establish versions of the Bott vanishing theorem in both Hodge cohomology and the Chow ring for $p$-closed foliations, providing a positive characteristic analogue to classical vanishing results in complex geometry.
\end{abstract}


\tableofcontents

\section{Introduction}

\subsection{Main results and structure of the paper}
This paper studies foliations on smooth algebraic varieties over an algebraically closed field $\field$ of arbitrary characteristic $p$. When $p=0$, it has been conjectured in \cite{MR3626006} that any regular foliation $\F$ on a rationally connected projective  variety $X$, is algebraically integrable, meaning that all its leaves are algebraic subvarieties of $X$. This conjecture holds at least for surfaces, as it follows from the classification of regular foliations on surfaces by \cite{MR1474805}. The result has also been extended for threefolds under certain hypotheses, see \cite[Theorem 1.4]{MR4502826}. 

If the aforementioned conjecture holds in characteristic zero,  it is natural to expect that a regular foliation on a rationally connected variety, over a field of positive characteristic $p$ is $p$-closed (see Section~\ref{sec:basicmat} for the precise definition of $p$-closed). However, we produce regular foliations on Hirzebruch surfaces in arbitrary characteristic $p>0$ (Remark~\ref{rmk:HirzNonp}), and on weak del Pezzo surfaces of degree $4$, in characteristic $p=2$,  which are not $p$-closed  (Theorem~\ref{thm:5parameter} and Example~\ref{ex:regularnon2closed}). Despite these counterexamples, $p$-closedness remains the expected behavior for most regular foliations on rationally connected varieties.  This expectation can be verified at least in the following cases: for projective spaces $\projn$ and for del Pezzo surfaces.  In the first case, it follows from \cite[Theorem 1.5]{MR2148539} that a regular foliation of codimension $q$ on   $\projn$ occurs only if $q=1$, $p=2$, and $\F$ is given by the kernel of a differential $1$-form $dF$, where $F$ is a homogeneous polynomial of degree two. In particular $\F$ is $2$-closed, see Theorem~\ref{thm:regfolproj} below. The second case,  namely del Pezzo surfaces, is addressed  in Theorem~\ref{thm:delpezzo}:
 
 \begin{thm}\label{thmintro:delpezzo}
Let $X$ be a smooth projective del Pezzo surface and  let $\F$ be a regular foliation on it. Then $X$ is either $\mathbb P^1_{\field} \times \mathbb P^1_{\field}$  or $\mathbb F_1$, and $\F$ leaves one of the standard $\mathbb P^1_{\field}$-fibrations invariant. More precisely, if $X=\mathbb P^1_{\field} \times \mathbb P^1_{\field}$ then one of the two rulings is invariant, while for $\mathbb F_1$ the conclusion holds for the standard fibration $\mathbb F_1 \to \mathbb P^1_{\field}$.
\end{thm}

We also study regular foliations on minimal rational surfaces and obtain the  following characterization:

\begin{thm}\label{thmintro:regularminimal} 
Let $\mathcal{F}$ be a regular foliation on a minimal rational surface $X$.  Then $X$ is a Hirzebruch surface $\mathbb F_e$, $e\neq 1$,  and either
\begin{enumerate}
\item $p\ge 0$, and $\mathcal{F}$ is the $\mathbb P^1_{\field}$-fibration of $\mathbb F_e$ over $\mathbb P^1_{\field}$; or
\item\label{item2} $p > 0$, $p$ divides $e$ and $\F$ is a Riccati foliation on $\mathbb F_e$.
\end{enumerate} 
\end{thm}
Theorem~\ref{thmintro:regularminimal} is a consequence of Theorem~\ref{thm:regularminimal} in the main text. There, the reader will find a precise description of the regular Riccati foliation on $\mathbb F_e$.  It follows that, any regular foliation on a Hirzebruch surfaces $\mathbb F_p$ is $p$-closed, see Corollary~\ref{cor:minimalpclosed}.  
 Item~\eqref{item2} of Theorem~\ref{thmintro:regularminimal} stands in sharp contrast with the characteristic zero case, since any regular foliation on Hirzebruch surfaces over $\mathbb C$  is a  $\mathbb P^1_{\field}$-fibration. 
Beyond these three classes (projective spaces, Hirzebruch surfaces, and del Pezzo surfaces) we investigate regular foliations on weak del Pezzo surfaces, which seems to us the next natural case to consider. We show that any regular foliation on a weak del Pezzo surface of degree $6$ or $8$ is $p$-closed, see Corollary~\ref{cor:6and8pclosed}.  By Theorem~\ref{thm:evendegree},  weak del Pezzo surfaces of odd degree admit no regular foliations. Furthermore, by Theorem~\ref{thm:deg6implieschar2}, if $p>2$ then there are no regular foliations on weak del Pezzo surfaces of degree different from $8$.  All the above cases involve rational varieties;  in Proposition~\ref{prop:examplechar3} we present a regular foliation on a non-rational surface, in characteristic $p>2$,  which is not in general $p$-closed.

We also investigate the conditions under which a foliation is $p$-closed, even in the presence of singularities. This leads us to consider invariant subvarieties. A distinctive phenomenon in positive characteristic is that  a singular curve may be invariant under a regular foliation. A typical example is given by the curve given by $f = y^2 - x^p$ and the foliation defined by $dy$. A preliminary analysis might suggest that this occurs only if the degree of the curve is multiple of $p$ or at least greater than $p$. However, this is not the case. We construct examples where a curve $C$, singular at the origin of the affine plane $\mathbb A^2_{\field}$, is invariant under a regular foliation at the origin,  but its degree $\deg (C)$ is relatively small compared  to $p$. See Propositions \ref{P:counterexample} and \ref{P:counterexamplebis}.  On the other hand, $\deg(C)$ cannot be too small, this is the content of the following result, which is Theorem~\ref{thm:invsmalldegree} in the main text.

\begin{thm}\label{thmintro:invsmalldegree}
Let $C\subset \mathbb A^2_{\field}$ be a reduced curve which is invariant by a foliation $\F$ on $\mathbb A^2_{\field}$ and  assume that 
\[
\deg(C)(\deg(C) -1)<p. 
\]
Then $\sing(C)\subset \sing(\F)$.
\end{thm}

For foliations on projective spaces,  the existence of a smooth hypersurface that does not intersect the singular set of $\F$ implies that $p>0$ and that $\F$ is $p$-closed, see Corollary~\ref{cor:HnotSing}.   We extend this result for distributions, where the integrability condition on $\F$ is no longer required. In this more general setting, we assume only the weaker hypothesis that the intersection of the singular set of $\F$ with the invariant hypersurface is small (see Theorem~\ref{thm:HcapSing=empty}): 

\begin{thm}\label{thmintro:HcapSing=empty}

Let $\D$ be a codimension one distribution on $\projn$. Let $Y=\mathcal Z(F)$ be a reduced  hypersurface invariant by $\D$. If $Y\cap \textrm{Sing}(\D)$ has codimension at least $3$,  then $p>0$, $\deg F$  is a multiple of $p$ and $\D$ is given by the kernel of $dF$. In particular, $\D$ is integrable and $p$-closed.
\end{thm}

An ingredient in the proof of Theorem~\ref{thmintro:HcapSing=empty} is the Camacho-Sad formula. The formula says that the sum of Camacho-Sad indices along an invariant curve $C$, inside a surface, coincides with the self-intersection $C^2$. The  Camacho-Sad formula, in arbitrary characteristic, is proved  in Theorem~\ref{thm:CSformula}. It should be noted that the first assertion of this theorem ($\deg F$  is a multiple of $p$) was established in a more general form in \cite[Corollary 4.5]{EstevesKleiman}.

Theorem~\ref{thmintro:HcapSing=empty} yields interesting consequences for distributions on projective spaces having small singular set. The first one is the following (see Corollary~\ref{cor:nonintegrabledis}). 

\begin{thm}\label{corintro:nonintegrabledis}
	Let $\D$ be a codimension one  non integrable distribution on  $\mathbb{P}_{\field}^{n}$ such that $ \sing (\D) $ has codimension at least $3$, then $\D$ has no invariant hypersurface. 
\end{thm}


In the integrable case, we get the following result (see  Corollary~\ref{cor:singcod3}). 

\begin{thm}\label{corintro:singcod3}
Let $\F$ be a codimension one foliation on $\mathbb P^n_{\field}$ of degree $d$ and assume that $\sing (\F)$ has codimension at least $3$.  Then $p>0$,  $d+2$ is a multiple of $p$ and $\F$ is  given by the kernel of a closed projective $1$-form $\Omega$ on $\mathbb A^{n+1}_{\field}$. Moreover, the following conditions are equivalent: 
\begin{enumerate}
\item   $\F$ admits a reduced invariant hypersurface; 
\item   $\Omega$ is exact, i.e there exists a homogeneous polynomial $G$ of degree $d+2$ such that $\Omega=dG$. 
\end{enumerate}
\end{thm}

In Example~\ref{ex:noinvhy}, Theorem~\ref{corintro:singcod3} is applied to show that, for all prime $p>0$ and $n\ge 3$,   there exist codimension one foliations on $\mathbb{P}_{\field}^{n}$ with finite singular set and no invariant hypersurface, in contrast with the regular case Theorem~\ref{thm:regfolproj}. 

The degree $d$ of a codimension one foliation $\F$ on $\projn$ is determined by the degree of the normal bundle $N_{\F}$ of $\F$, that is,  $N_{\F} = \mathcal O_{\mathbb P^n_{\field}}(d+2)$. Then Theorem~\ref{corintro:singcod3} implies that $p$ divides the degree of $N_{\F}$. This phenomenon can be viewed as a reflection of Bott vanishing theorem, which assures the vanishing of suitable powers of Chern classes  of the normal bundle in Hodge cohomology; see Theorem~\ref{thm:bottvanishing} and  Theorem~\ref{prop:SingBB} in the main text.

We proceed with a version of the Bott vanishing theorem,  in the Chow ring of $X$, assuming that $\F$ is $p$-closed (see Corollary~\ref{thm:Bootva}):

\begin{thm}\label{thmintro:Bootva}
Let $\F$ be a regular $p$-closed foliation on a smooth  projective variety $X$. Let $\mathrm{Chern}^*(N_{\mathcal F})$ be the graded subring of the Chow ring $\mathrm{CH}^*(X)$, generated by the Chern classes of $N_{\mathcal F}$. Then
\[
\mathrm{Chern}^\ell(N_{\mathcal F})\subset p^{\ell-q}\cdot\mathrm{CH}^\ell(X)
\]
for each $\ell>q$. 
\end{thm}

Bott vanishing theorem is an important tool in the theory of foliations. Indeed, it has been applied in deep results such as the classification of regular foliations on surfaces \cite{MR1474805} over the complex numbers. Notice that, in positive characteristic there exist regular foliations with $N_{\F}^2\neq 0$ in $\mathbb Z$ (see Remark~\ref{rmk:existencechr2}), also  there exist examples of codimension one regular foliations  on $\mathbb P^n_{\field}$, $n$ odd,   with $N_\F = \mathcal O_{\mathbb P^n_{\field}}(2)$ (see Theorem~\ref{thm:regfolproj}).

\subsection{Basic definitions and notations}\label{sec:basicmat}

Let $X$ be a smooth algebraic variety over an algebraically closed field $\field$ and let $\mathcal O_X$ be its sheaf of regular functions. Let $\Omega^1_X$ and $T_X=\Hom_{\mathcal O_X}(\Omega^1_X,\mathcal O_X)$ be the cotangent and tangent sheaves of $X$, respectively.  

We can see $T_X$ as a sheaf of derivations. First, recall that a {\bf derivation} is a $\field$-linear transformation $w$ of $\mathcal O_X$ satisfying $w(ab) = aw(b)+bw(a)$ for any $a,b\in \mathcal O_X$.   Let $d: \mathcal O_X \to \Omega^1_X$ be the exterior derivative, defined by $d(a) = 1\otimes a - a\otimes 1 \in I_{\Delta}/I^2_{\Delta} = \Omega^1_X$, where $I_{\Delta}$ is the ideal sheaf of the diagonal $\Delta\subset X\times X$.  Then, a section $v$ of $T_X$ can be seen as a derivation $v: \mathcal O_X\to \mathcal O_X$ by $v(a) := v( d(a) )$, for all $a\in \mathcal O_X$.  

Let us endow $T_X$ with a {\bf Lie bracket}, given by $[v_1,v_2] = v_1v_2 - v_2v_1$, for any derivations $v_1, v_2$ of $T_X$. A subsheaf $E\subset T_X$ is said to be saturated if $T_X/E$ is torsion free. A {\bf foliation} $\mathcal F$ on $X$ is determined by a saturated coherent subsheaf $T_{\mathcal F}\subset T_X$ which is involutive, i.e. closed under the Lie bracket: $[T_{\mathcal F}, T_{\mathcal F}]\subset T_{\mathcal F}$.   We call $T_{\mathcal F}$ the {\bf tangent sheaf} of $\mathcal F$. Dropping the involutivity condition, we obtain the definition of a {\bf distribution}.  The {\bf dimension} of $\mathcal F$ is the generic rank of $T_{\mathcal F}$, which is constant on a nonempty open subset of $X$. The difference $\dim X-\dim T_\F$ is the codimension of $\F$.
The double dual $N_{\mathcal F}:=(T_X/T_{\mathcal F})^{**}$  is called the {\bf normal sheaf} of $\mathcal F$.   By a {\bf regular foliation} we mean a foliation whose tangent sheaf  is a subbundle, that is, $T_{\mathcal F}$ and $T_X/T_{\mathcal F}$ are locally free. For regular foliations,  we have $N_{\mathcal F}\simeq T_X/T_{\mathcal F}$.  The dual sheaf $\Omega^1_{\mathcal F}:=T_{\mathcal F}^*$  is called the {\bf cotangent sheaf} of $\mathcal F$. The locus of points where $T_X/T_{\mathcal F}$ is not locally free is called the {\bf singular locus} of $\mathcal F$, denoted here by $\sing(\mathcal F)$. The complement in $X$ of the singular locus of $\F$ is called the {\bf smooth} or {\bf  regular} locus of $\F$.  Finally, let $f:Y\to X$ be a birational morphism from another smooth variety $Y$. Let $V\subset Y$ and $U\subset X$ be open subsets such that $f:V\to U$ is an isomorphism. We then define $f^*\F$ as the foliation on $Y$ defined by the unique saturated submodule $T_{f^*\F}$ of $T_Y$ which agrees with $(f|_V)^*(T_\F|_U)$ on $V$.

If the characteristic of $\field$ is $p>0$, then $w^p$ is still a derivation for any derivation $w\in T_X$. A foliation $\mathcal F$ is said to be {\bf $p$-closed} if $v^p\in T_{\mathcal F}$ for each $v\in  T_{\mathcal F}$. 


Given a codimension one distribution $\F$ on $X$, the  inclusion $ N_\F^* \to \Omega^1_X$ gives a $1$-form
\[
\omega \in \h^0(X, \Omega^1_X\otimes  N_\F)
\]
whose singular set coincides with $\sing (\F)$ (recall that the singular set of a form is its  zero locus). The section $\omega$ is unique up to  multiplication of elements of $\h^0(X, \mathcal O_X^*)$. If in addition $\F$ is a foliation, then $\omega$ is integrable: $\omega\wedge d\omega = 0$. A reduced hypersurface $Y$ is {\bf invariant} by $\F$ if the wedge product $\omega\wedge \frac{df}{f}$ is a regular $2$-form for any local equation $f$ that defines $Y$. 

In projective spaces, codimension one distributions and  foliations admit an explicit description. We let $\projn$ and $\mathbb A^n_{\field}$ denote the $n$ dimensional projective space and affine space over  $\field$, respectively.  Let $n, d \in \mathbb{Z}$ with $n>1$, $d\geq 0$. A codimension one distribution  $\mathcal{F}$ of degree $d$ on the projective space $\mathbb{P}_{\field}^{n}$ is given  by a non-zero element $\omega \in \h^0(\mathbb{P}_{\field}^{n},\Omega_{\mathbb{P}_{\field}^{n}}^{1}(d+2))$, unique up to multiplication by an  element of $\field^*$,  such that $\codim \sing (\omega)\geq 2$.  
It follows from Euler's sequence that $\h^0(\mathbb{P}_{\field}^{n} ,\Omega_{\mathbb{P}_{\field}^{n}}^{1}(d+2))$ can be identified with the vector space of homogeneous polynomial $1$-forms in $\mathbb{A}_{\field}^{n+1}$
\[
\omega=\sum_{i=0}^nA_i\,dx_i,\quad\text{$A_i\in\field[x_0,\ldots,x_n]$ homogeneous of degree $d+1$,}
\]
which are annihilated by the radial vector field $R = \sum_{i=0 }^n x_{i}\partial_{x_i}$.

\subsection*{Acknowledgments}   
Fassarella acknowledges the support from CNPq Universal 10/2023. Mendson acknowledges the support from Capes (Grant number 88887.145997/2025-00) and CNPq - Programa Conhecimento Brasil (Grant number 315040/2025-4). Touzet would like to thank the Henri Lebesgue Center for constant support. The authors thank  Brazilian-French Network in Mathematics, CAPES-COFECUB  932/19 and 1017/24. We would also like to thank Jorge Vit\'orio Pereira for inspiring discussions, and Maur\'icio Corr\^ea for bringing a useful reference to our attention.

\section{Foliations on surfaces}

\subsection{Preliminary}

Before presenting results concerning indices of singularities of foliations on surfaces, we first establish the following result, which parallels \cite[V, Lemma 2.4]{suwa}.

\begin{prop}\label{decomposition} 
Let $X$ be a smooth variety over the algebraically closed field $\field$. Let $\mathcal F$ be a distribution of codimension one on $X$ admitting a \textbf{reduced} invariant hypersurface $Y$. Let $x \in X$ be a point, $U \subset X$ an affine and open neighborhood of $x$ where $N^*_{\mathcal F}|_U = \mathcal O_U\omega$ and $\mathcal O_X(-Y)|_U = \mathcal O_U f $. Then there exist an open and affine neighborhood $V$ of $x$, functions $g,h \in \mathcal O_X(V)$ and $\sigma \in \Omega^1_X(V)$
such that
    \begin{eqnarray}\label{eq:ghf}
    g\omega = hdf + f\sigma 
    \end{eqnarray}
where $f$ is relatively prime to $g$ and $h$. Furthermore: 
\begin{itemize}
    \item If $Y$ is smooth at $x$, we can take $g=1$; 
    \item If $\F$ is smooth at $x$, we can take $h=1$.
\end{itemize}
\end{prop}

\begin{proof}
We can find an open neighborhood $V$ of $x$ in $U$ and functions $s_1,\ldots,s_n \in \mathcal O_X(V)$
such that
    \[
        \Omega^1_X|_V = \mathcal O_V ds_1 \oplus \cdots \oplus \mathcal O_V ds_n
        \quad \text{and} \quad
        T_X|_V = \mathcal O_V \partial_{s_1} \oplus \cdots \oplus \mathcal O_V \partial_{s_n}.
    \]
Also, shrinking again $V$ if needed and using that $\mathcal O_{V,x}$ is a unique factorization domain, we
can assume that $f = f_1\cdots f_m$, with $f_1,\ldots,f_m \in \mathcal O(V)$ such that
$\mathcal O_{X,x}f_1,\ldots,\mathcal O_{X,x}f_m$ are pairwise distinct prime ideals. In particular, for each affine and open neighbourhood $W \subset V$ of $x$, the ideals $\mathcal O(W)f_\mu$ remain prime. 

Claim: We can choose the derivation $\partial_{s_1}$ such that $\partial_{s_1}(f)$  and  $f$ are relatively prime in  $\mathcal O_{V,x}$. For each $\textbf{b}=(b_1,\ldots,b_n)\in \field^n$,
let $\partial_\textbf{b}:\mathcal O(V)\to \mathcal O(V)$ be the derivation
$\sum_{i=1}^n b_i\partial_{s_i}$. For each $\mu\in \{1,\ldots,m\}$, let $B_\mu\subset \field^n$ be the subset of all
$\textbf{b}$ such that $\partial_\textbf{b}(f_\mu)\equiv 0 \mod f_\mu$; clearly $B_\mu$ is a subspace. Since the field $\field$ is perfect and the scheme
$\{f_\mu=0\}$ is integral, its smooth locus is dense, see Proposition 16 in
Chapter 2 of \cite{BLR}. By the Jacobi criterion, \cite[Ch. 2, Proposition 7]{BLR}, we have $B_\mu\neq \field^n$.
Let $\textbf{b}\in \field^n\setminus B_1\cup \cdots \cup B_m$. This means that for each $\mu\in \{1,\ldots,m\}$, we have
$\partial_\textbf{b}(f_\mu)\not\equiv 0 \mod f_\mu$ and hence $\partial_\textbf{b}(f)\not\equiv 0 \mod f_\mu$ because
$\partial_\textbf{b}(f)\equiv (f/f_\mu)\partial_\textbf{b}(f_\mu) \mod f_\mu$. Completing
$\{\textbf{b}\}$ to a basis and renaming the functions $s_1,\ldots,s_n$, we obtain that
$\partial_{s_1}(f)\not\equiv 0 \mod f_\mu$ for $\mu\in \{1,\ldots,m\}$. Said differently,
    \begin{eqnarray}\label{eq:relprimes}
    \partial_{s_1}(f) \text{ and } f \text{ are relatively prime in } \mathcal O_{V,x}. 
    \end{eqnarray}
This proves the claim above. 
    
Let us write $\omega=\sum_{j=1}^n A_jds_j$ so that
    \[
    df\wedge \omega
    =
    \sum_{i<j}
    \bigl(\partial_{s_i}(f)A_j-\partial_{s_j}(f)A_i\bigr)
    ds_i\wedge ds_j.
    \]
Since $Y$ is invariant by $\mathcal F$, there exist, for $1\leq i<j\leq n$, elements $\phi_{ij}\in \mathcal O(V)$ such that
    \begin{equation}\label{2.1}
    f\phi_{ij}
    =
    \partial_{s_i}(f)A_j-\partial_{s_j}(f)A_i
    \quad \text{for } i<j. 
    \end{equation}
Consequently, if $\psi\in \mathcal O(V)$ divides $f$, then
    \begin{equation}\label{2.2}
        \partial_{s_i}(f)A_j \equiv \partial_{s_j}(f)A_i \mod \psi
    \quad \text{for any couple } i,j. 
    \end{equation}
A straightforward computation shows that, for each $i\in \{1,\ldots,n\}$, we have
    \begin{equation}\label{2.3}
    \partial_{s_i}(f)\cdot \omega
    =
    A_i\cdot df
    +
    f
    \underbrace{
    \left(
    -\sum_{j<i}\phi_{ji}ds_j+\sum_{j>i}\phi_{ij}ds_j
    \right)
    }_{\sigma_i}.
    \end{equation}
By \eqref{eq:relprimes}, 
$f$ and $\partial_{s_1}(f)$ are relatively prime. If $\psi\in \mathcal O(V)$ is a prime element dividing $f$ and
$A_1$ then, by (\ref{2.2}), we have $\psi$ divides $A_j$ for $j>1$, so
 $\psi$ is a common divisor of $A_1,\ldots,A_n$,
which is excluded. Hence $f$ and $A_1$ are relatively prime. We have achieved decomposition \eqref{eq:ghf} with $g=\partial_{s_1}(f)$ and $h=A_1$.

If $Y$ is smooth at $x$ then, for a certain $i$, we know that $\partial_{s_i}(f)\in \mathcal O(V)^\times$, again shrinking
$V$ if needed. We then arrive at
    \[
    \omega
    =
    \partial_{s_i}(f)^{-1}A_i\cdot df
    +
    f\cdot \partial_{s_i}(f)^{-1}\sigma_i.
    \]
Let $\psi$ be an irreducible common divisor of $f$ and $\partial_{s_i}(f)^{-1}A_i$. A fortiori $\psi$ divides  $A_i$. From (\ref{2.2}), we have $\partial_{s_i}(f)A_j\equiv 0 \mod \psi$; again an impossibility.

Assume now that $\mathcal F$ is regular at $x$. Shrinking $V$ if needed, we can assume that $A_i\in \mathcal O(V)^\times$ for a certain $i$. We then arrive at
    \[
    \partial_{s_i}(f)A_i^{-1}\cdot \omega
    =
    df+f\cdot A_i^{-1}\sigma_i.
    \]
Let $\psi$ be an irreducible common divisor of $f$ and $\partial_{s_i}(f)A_i^{-1}$. Then $\psi$ divides both $f$ and $\partial_{s_i}(f)$, which is excluded if $i=1$. If $i>1$, then
$ \psi$ divides the sum 
\[
f\phi_{1i}A_i^{-1}+\partial_{s_i}(f)A_1A_i^{-1}
= \partial_{s_1}(f)
\]
which is again excluded by \eqref{eq:relprimes}. 
We achieve the decomposition \eqref{eq:ghf} with $g=\partial_{s_i}(f)A_i^{-1}$ and $h=1$.
\end{proof}

It is worth mentioning that in positive characteristic, it can happen that a singular point $x$ of an $\F$-invariant hypersurface $Y$ might   be contained in the regular locus of $\F$. This scenario will be considered   in Section~\ref{SS:singnonsing}.  On the other hand, this phenomenon does not occur as soon as $Y$ is  normal crossing at $x$:
\begin{lemma}\label{L:normalcrossing}
 Notations and hypotheses as in Proposition \ref{decomposition}. Assume there exists a choice of local coordinates $s_1,...,s_n$  at $x$ such that $Y$ is locally defined by $f=s_1 s_2...s_t$, for some $1<t\leq n$. Then $x$ belongs to  $\sing(\mathcal F)$.
\end{lemma}
\begin{proof}
  Write $g\omega =hdf +f\sigma $ as in the conclusion of Proposition \ref{decomposition} and let  $A_{i}\in \mathcal{O}_{X,x}$ such that  $\omega = \sum_{i=1}^{n}A_{i}ds_{i}$. As $g$ and $s_j$ are relatively prime for $1\leq j\leq t$, one can easily infer each $A_i$ is divisible by $s_j$, for $j\not=i$. In particular, $\F$ is singular at $x$.
\end{proof}

\subsection{Intersection formulae} 

Let $X$ be a smooth projective surface over an algebraically closed field $\field$ of characteristic $p\ge 0$. Let $\F$ be a codimension one foliation on $X$ and let $C\subset X$ be a reduced invariant curve. Given a point $x\in C$, consider a local decomposition $g\omega = hdf + f\sigma$, given by Proposition~\ref{decomposition}. Following \cite{MR3328860}, we define the index
\[
Z(\F, C, x) := \; \text{vanishing order of $\frac{h}{g}|_C$ at $x$}.
\]
Note that if $C$ is  smooth, then $Z(\F, C, x)\ge 0$. 
We set

\[
Z(\F, C) := \sum_{x\in C} Z(\F, C, x).
\]

\begin{prop}\label{prop:GM}
Let $X$ be a smooth projective surface over $\field$. Let $\F$ be a  foliation on $X$ and let $C\subset X$ be a reduced  invariant curve of arithmetic genus $p_a$. Then 
\begin{eqnarray*}
N_{\F}\cdot C = C^2 + Z(\F, C)\quad \text{and} \quad
K_{\F} \cdot C = 2p_a-2 + Z(\F, C).
\end{eqnarray*}
\end{prop}
\proof
The proof follows the same lines of \cite[Chap. 2, Prop. 3]{MR3328860}. We do the details for sake of completeness. From Proposition~\ref{decomposition}, there is an open covering $\{U_i\}$ of $X$ and a decomposition 
\[
g_i\omega_i = h_i df_i + f_i\sigma_i
\] 
in each $U_i$. Where the collection $\{f_i\}$ defines $C$, and the $\{\omega_i\}$ are regular $1$-forms on $U_i$ defining $\F$. The cocycle $\{f_{ij}\}$, $f_{ij}=f_if_j^{-1}$,  represents $\mathcal O_X(C)$ and $\{g_{ij}\}$, with $\omega_i = g_{ij}\omega_j$, represents $N_{\F}$.

 It follows from the relations above that
 \[
 \frac{\omega_i}{f_i} = g_{ij}f_{ij}^{-1} \frac{\omega_j}{f_j} \quad \text{in} \quad U_i\cap U_j. 
 \]
 and
 \[
\omega_i = \frac{h_i}{g_i} df_i \quad \text{in} \quad U_i\cap U_j\cap C. 
 \]
 Since $df_i = f_{ij} df_j$ in $U_i\cap U_j \cap C$, we conclude that 
 \[
 \frac{h_i}{g_i}|_C = (g_{ij}f_{ij}^{-1})|_C \cdot  \frac{h_j}{g_j}|_C
 \]
 which means that $\{ \frac{h_i}{g_i}|_C\}$ defines a rational section of $[N_{\F}\otimes \mathcal O_X(-C)]|_C$. On the one hand, the total sum of zeros and poles of this section  coincides with $Z(\F, C)$. On the other hand it is given by $N_{\F}\cdot C - C^2$, proving the first formula. 
 
 The second formula follows from the first and the adjunction formula (cf. \cite[Chap. V, Ex. 1.3]{AGHarthorne})
 \[
 2p_a -2= C\cdot (C+K_X).
 \]   
 This concludes the proof of the proposition. 
\endproof

Let $C\subset X$ be a reduced curve, which is not invariant by $\F$. We define $tang(\F, C)$, the tangency order between $\F$ and $C$ at  $x\in C$ as follows. Let $v\in T_{\F}$ be a local derivation, with isolated zeros, defining $\F$ at $x$, then
\[
tang(\F, C) = \dim_{\field} \frac{\mathcal O_{X,x}}{\big(f, v(f)\big)}.
\]  

We present below well-known intersection formulas for non-invariant curves. As in the previous case, the proof in characteristic zero carries over to arbitrary characteristic.  We omit the details here and refer the reader to  \cite[Chap. 2 Prop. 2]{MR3328860}.

\begin{prop}\label{prop:tangformula}
Let $X$ be a smooth projective surface over $\field$. Let $\F$ be a  foliation on $X$ and let $C\subset X$ be a reduced  curve of arithmetic genus $p_a$, which is not invariant by $\F$. Then 
\begin{eqnarray*}
N_{\F}\cdot C = 2-2p_a +  tang(\F, C)\quad \text{and} \quad
K_{\F} \cdot C = - C^2 + tang(\F, C).
\end{eqnarray*}
\end{prop}
\subsection{Blowing-up}
Let $X$ be a smooth projective surface over an algebraically closed field $\field$ of characteristic $p\ge 0$. Let $\F$ be a codimension one foliation on $X$ and let $C\subset X$ be a reduced invariant curve. Given a point $x\in X$,  let $\pi$ be the blow-up at $x$. Consider a local generator $\omega=Adu+Bdv$. 
  Following \cite{MR3328860}, one consider the following nonnegative integers:

\begin{itemize}
    \item The vanishing order $a(x)$ of $\omega$ at $x$:
    \[
    a(x)= \mathrm{Max}\{n\in\NN: (A,B)\subset \frak{m}_{X,x}^n\}\,;
    \]
    \item The vanishing order $l(x)$ of $\pi^* \omega$ along the exceptional divisor $E=\pi^{-1}(x)$.
\end{itemize}

Denote by  $\tilde\F=\pi^* \F$ the strict transform of $\F$ by $\pi$. The interplay between these objects, as described in \cite[Chapter 2]{MR3328860},  remains unchanged in our setting. This is explicitly stated for further use:
\begin{eqnarray}\label{eq:normalbu}
N_{\tilde\F} = \pi^* N_\F -l(x) E.
\end{eqnarray}

\begin{enumerate}
\item If $E$ is invariant, then
\begin{eqnarray}\label{eq:Zbuinv}
l(x)=a(x)\quad \text{and}\quad Z(\tilde\F, E)= 1 + a(x).        \end{eqnarray}
In particular, $x$ is a regular point of $\F$ if and only if 
\begin{eqnarray}\label{Zburegularinv}
Z(\tilde\F, E)=1.
\end{eqnarray}
\item If $E$ is not invariant, then
\begin{eqnarray}
  l(x) &=& a(x) + 1, \label{eq:noninvla}\\ 
 tang( \tilde\F, E) &=& a(x)-1 = l(x)-2, \label{eq:tangbu}\\  
\ N_{\tilde{\F}}\cdot E &=& 2 + tang( \tilde \F, E). \label{eq:normexep}
\end{eqnarray}
 \end{enumerate}

Now, let $C$ be a smooth rational $\F$-invariant curve on $X$. Recall that
$$N_\F\cdot C= Z(\F, C) + C^2.$$
Assume also that $x\in C$ and let $\tilde C$ be its strict transform (hence $\tilde\F$-invariant). 
We propose to compute $Z(\F, C)$ assuming that $Z(\tilde\F, \tilde C)$ is known. To this end, we observe that
\begin{eqnarray*}
\pi^*(N_\F)\cdot\pi^*(C) &=& ( N_{\tilde\F}+l(x) E)(\tilde C + E)\\
                        &=& Z(\tilde \F, \tilde C) + C^2-1 +2+ tang(\tilde \F, E).
\end{eqnarray*}
Finally, using $N_\F\cdot C = \pi^*(N_\F)\cdot\pi^*(C)$ and $N_\F\cdot C= Z(\F, C) + C^2$,  we get
\begin{eqnarray}\label{eq:ZbytildeZ}
Z(\F, C)=Z(\tilde \F, \tilde C) + 1 + tang(\tilde \F, E).
\end{eqnarray}
\subsection{The Camacho-Sad formula}

We now define the Camacho-Sad index. Let $C = \cup_{i=1}^r C_i$ be the decomposition of $C$ in irreducible components into the formal power series ring of $X$ at $x$. Let $\varphi_i: (\tilde{C_i}, q_i) \to (C_i, x)$ be a normalization map.  The Camacho-Sad index at $x$ is defined as
\[
CS(\F, C, x) := - \sum_{i=1}^r \text{Res}_{q_i} \left\{ \varphi_i^*\left(\frac{\sigma}{h}\right)\right\} \quad \in \field
\]
where $\sigma$ and $h$ are given by \eqref{eq:ghf}. 
If $x$ is smooth point of $C$,  we note that  
\[
CS(\F, C, x) = - \text{Res}_x \left\{ \frac{\sigma}{h}|_C\right\}.
\]
We refer to \cite[Chap. III Section 7]{AGHarthorne} and \cite[Chap. II]{serreclass} for an introduction on residues of differential forms on curves.
In view of Proposition~\ref{decomposition}, if $x$ is a smooth point of $\F$, we can take $h=1$,  then 
\[
CS(\F, C, x) = 0.
\]

The Camacho-Sad formula relates the self-intersection $C^2$ of $C$ with the total sum of indices at $C$. For  a smooth curve, this follows from the  Residue Theorem, which is the basis of the next result.   
 We define:
\[
CS(\F, C)  = \sum_{x\in C} CS(\F, C, x). 
\]

\begin{prop}\label{prop:CSsmooth}
Let $X$ be a smooth projective surface over $\field$. Let $\F$ be a  foliation on $X$ and let $C\subset X$ be a \textbf{smooth} invariant curve. Then 
\[
 C^2 =  CS(\F, C) \quad \text{in} \quad \field. 
\]
\end{prop}

\proof
As before, using Proposition~\ref{decomposition}, we obtain an open cover $\{U_i\}$ of $X$ and a decomposition 
\[
g_i\omega_i = h_i df_i + f_i\sigma_i
\] 
in each $U_i$.  The cocycle $\{f_{ij}\}$, $f_{ij}=f_if_j^{-1}$,  represents $\mathcal O_X(C)$ and $\{g_{ij}\}$, with $\omega_i = g_{ij}\omega_j$, represents $N_{\F}$.

From these relations, we get
\[
d\omega_i = \left(d\text{log}\left(\frac{h_i}{g_i}\right) - \frac{\sigma_i}{h_i}\right)\wedge \omega_i \quad \text{in} \quad U_i\cap C.
\]
If we set 
\[
\beta_i := \left[ d\text{log}\left(\frac{h_i}{g_i}\right) - \frac{\sigma_i}{h_i}\right]|_C
\]
we see that 
\[
\left( d\text{log} g_{ij} + \beta_j - \beta_i \right)\wedge \omega_i = 0 \quad \text{in} \quad U_i\cap U_j \cap C.
\]

Therefore, we can write 
\[
\beta_i = \beta_j + d\text{log} g_{ij} + \alpha_{ij} \omega_i\quad \text{in} \quad U_i\cap U_j \cap C.
\]
for suitable rational functions $\alpha_{ij}$. Let $\{t_i\}$, $t_i = (g_{ij}|_C)\cdot t_j$, be a rational resolution of $N_{\F}|_C$.   Since $C$ is invariant by $\F$,  $\omega_i $ vanishes identically  as $1$-form on $C$, then from the previous identity we obtain a global rational $1$-form $\eta$ on $C$
\[
\eta:= \beta_i - d\text{log} t_i = \beta_j - d\text{log} t_j . 
\]
The total sum of residues of $\{\beta_i\}$ coincides with $Z(\F, C) + CS(\F, C)$, while the residues of $\{  d\log t_i \}$ add to  $N_{\F}\cdot C$. From the Residue  Theorem 
we obtain
\[
Z(\F, C) + CS(\F, C) - N_{\F}\cdot C = 0. 
\](See \cite[Corollary, p.155]{TateResidues} or \cite[Chap. II]{serreclass} for a proof of the Residue Theorem in this context.)
Then the result follows from Proposition~\ref{prop:GM}. 
\endproof

%
%
%
%

We refer to \cite[Section 3]{Per23} for an extension of the above result to codimension one foliations over $\mathbb C$ in higher dimensions. 

Now, we compare the index at a possible singular point $x$ of $C$ with the indices arising during  a resolution process. Let $\pi: \tilde{X} \to X$ be a resolution of the point $x\in C$, and let 
\[
E=\pi^{-1}(x)
\]
denote the exceptional divisor. Let $\tilde{C}$ be the strict transform of $C$  and let $E = \cup_{j=1}^s E_j $
be the decomposition of the exceptional divisor in irreducible components. Then we can write, as a divisor, the total transform as
\[
\pi^*(C) = \tilde{C} + D_x
\]
where $D_x =  \sum_{j=1}^s m_j E_j$. Let $C = \cup_{i=1}^r C_i$ be the decomposition of $C$ in irreducible components into the formal power series ring of $X$ at $x$. For each $C_i$, let $\tilde{C}_i$ its normalization, and let $q_i$ be the point of $\tilde{C}_i$ infinitely near $x$, i.e., $\pi(q_i) = x$. The next result will be useful for an inductive argument in the sequel. 

\begin{lemma}\label{lemma:firstblowup}
Let $\pi_1: X_1 \to X$ be the blow-up at $x$ and let $\F_1 = \pi_1^*(\F)$ be the foliation induced by $\F$ on $X_1$. Let $y_1, \dots, y_l$ be the singular points of $\F_1$ along the exceptional divisor $E^{(1)}$. Then 
\[
CS(\F, C, x) = \sum_{j=1}^l CS(\F_1, C^{(1)}, y_j) + C^{(1)}\cdot (mE^{(1)})
\] 
where $C^{(1)}$ is the strict transform of $C$ and $m$ is the algebraic multiplicity of $C$ at $x$. 
\end{lemma}

\proof
For each $j \in \{1, \dots, l\}$, let $\mathcal I_j$ denote the set of indices $i\in \{1, \dots, r\}$ such that $q_i$ is infinitely near to $y_j$. 

Let $f$ be a local equation for $C$ and consider the decomposition $g\omega = hdf + f\sigma$, given by Proposition~\ref{decomposition}. Let $f_1$ be a local equation for $C^{(1)}$, hence $\pi_1^*(f) = s^mf_1$ where $s$ is an equation for $E^{(1)}$. A direct computation shows that, near each point $y_j$,  there is a local function $g_1$ such that
\begin{eqnarray*}
g_1 \omega_1 = \pi_1^*(h)s^mdf_1 + f_1\cdot\left( ms^{m-1}\pi_1^*(h)ds+s^m\pi_1^*(\sigma)\right)
\end{eqnarray*}
where $\omega_1$ is a $1$-form that defines the foliation $\F_1$. Hence, we see that
\[
CS(\F_1, C^{(1)}, y_j) = - \sum_{i\in \mathcal I_j}\textrm{Res}_{q_i}\left\{\varphi_{1,i}^* \left( m\frac{ds}{s} + \pi_1^*\left(\frac{\sigma}{h}\right)\right)\right\}
\]
where $\varphi_{1,i}: \tilde{C}_i\to C_i^{(1)}$ denotes a normalization map for a branch of $C^{(1)}$ passing through $y_j$.  Consequently, the Camacho-Sad index at each $y_j$ writes
\[
CS(\F_1, C^{(1)}, y_j) = - \sum_{i\in \mathcal I_j} \left[ \textrm{Res}_{q_i}\left\{\varphi_{1,i}^*\left( m\frac{ds}{s} \right)\right\} + \textrm{Res}_{q_i}\left\{ \varphi_i^*\left(\frac{\sigma}{h}\right)\right\} \right]
\]
because $\varphi_i = \pi_1\circ\varphi_{1,i}: \tilde{C_i}\to C_i$. From this, we get
\begin{eqnarray*}
\sum_{j=1}^r CS(\F_1, C^{(1)}, y_j) &=&  - C^{(1)}\cdot (mE^{(1)}) + \sum_{j=1}^r \textrm{Res}_{q_i}\left\{ \varphi_i^*\left(\frac{\sigma}{h}\right)\right\} \\
	 &=& - C^{(1)}\cdot (mE^{(1)}) + CS(\F, C, x)
\end{eqnarray*}
concluding the proof of the lemma.

\endproof

\begin{prop}\label{prop:stricttrans}
With above notation, we have
\[
CS(\F, C, x) = \sum_{y\in \tilde{C}\cap E}  CS(\tilde{\F}, \tilde{C}, y) + \tilde{C}\cdot D_x.
\]
\end{prop}

\proof
We proceed by induction on the minimal number of blow-ups necessary to resolve $C$ at $x$. Let $\pi_1: X_1 \to X$ be  the first blow-up, let $E^{(1)}=\pi_1^{-1}(x)$ and let $C^{(1)}$ be the strict transform.   We denote by $\F_1$ the foliation $\pi_1^*(\F)$ induced by $\F$ on $X_1$.

As before,  let $\pi: \tilde{X} \to X$ be a resolution of $C$ at $x$ and $\tilde{C}$ be the strict total transform. In particular, $\pi$ can be written as $\pi = \pi_1\circ \psi_1$, where $\psi_1: \tilde{X} \to X_1$ resolves the curve $C^{(1)}$. 

%

It follows from Lemma~\ref{lemma:firstblowup} that
\begin{eqnarray}\label{eq:CSfirststep}
CS(\F, C, x) = \sum_{j=1}^l CS(\F_1, C^{(1)}, y_j) + C^{(1)}\cdot (mE^{(1)})
\end{eqnarray}
where $y_1, \dots, y_l$ are singular points of $\F_1$ along $E^{(1)}$. 

Now, by applying the induction hypothesis on $C^{(1)}$, for each $y_j$ one obtains
\begin{eqnarray}\label{eq:CSinduction}
CS(\F_1, C^{(1)}, y_j) =  \sum_{i\in\mathcal I_j} CS(\tilde{\F}, \tilde{C}, q_i) + \tilde{C}\cdot D_{y_j}
\end{eqnarray}
where $\mathcal I_j$ denote the set of indices $i\in \{1, \dots, r\}$ such that $q_i$ is infinitely near to $y_j$. 

 We note that 
\[
\pi^*(C) = \tilde{C} + \sum_{j=1}^l D_{y_j} + \psi_1^*(m E^{(1)})
\] 
and then
\[
D_x = \sum_{j=1}^l D_{y_j} + \psi_1^*(mE^{(1)}). 
\]
From \eqref{eq:CSfirststep} and \eqref{eq:CSinduction} we get 
\begin{eqnarray}\label{eq:CSqq1}
CS(\F, C, x) =  \sum_{i=1}^r CS(\tilde{\F}, \tilde{C}, q_i) + \tilde{C}\cdot  \big(\sum_{j=1}^l D_{y_j} \big) +  C^{(1)}\cdot (m E^{(1)}). 
\end{eqnarray}
By projection formula, we have  $C^{(1)}\cdot E^{(1)} = \tilde{C}\cdot \psi_1^*(E^{(1)})$, and consequently
\[
\tilde{C}\cdot  \big(\sum_{j=1}^l D_{y_j} \big) +  C^{(1)}\cdot (m E^{(1)}) = \tilde{C} \cdot D_x.
\]
This formula, together with \eqref{eq:CSqq1} give the desired result. 
\endproof

The next result extends the well known Camacho-Sad formula in characteristic zero, for algebraically closed fields of  arbitrary characteristic. 

\begin{thm}\label{thm:CSformula}(Camacho-Sad formula)
Let $X$ be a smooth projective surface over an algebraically closed field $\field$ of characteristic $p\ge 0$. Let $\F$ be a  foliation on $X$ and let $C\subset X$ be a reduced invariant curve. Then 
\[
 C^2 =  CS(\F, C) \quad \text{in} \quad \field. 
\]
\end{thm}

\proof
Let $\pi: \tilde{X} \to X$ be an embedded  resolution of   the singularities \cite[Theorem 3.15]{cutkosky}  of $ C$ and let $\tilde{C}$ denote its strict transform. Then, as a divisor, we can write
\[
\pi^*(C) = \tilde{C} + D. 
\] 
It follows from Proposition~\ref{prop:stricttrans} that 
\[
CS(\F, C) = CS(\tilde{\F}, \tilde{C}) + \tilde{C}\cdot D
\]
and then Proposition~\ref{prop:CSsmooth} gives
\begin{eqnarray}\label{eq:CS}
CS(\F, C) = \tilde{C}^2 + \tilde{C}\cdot D. 
\end{eqnarray}
Since $D$ is supported on the exceptional divisor, then 
\[
 0 = \pi^*(C)\cdot D = \tilde{C}\cdot D + D^2. 
\]
Hence, we have 
\[
C^2 = \pi^*(C)^2 = \tilde{C}^2 + \tilde{C}\cdot D
\]
and the conclusion of the proof follows from \eqref{eq:CS}. 

\endproof

\begin{ex}\rm
Let $\mathrm{char} (\field) = 3$. Consider the affine cubic curve $C = \mathcal Z(f)$ on $\mathbb A^2_{\field}$ given by $f = y^2-x^3$. It is invariant by the foliation $\F$ defined by $\omega = dy$. Since $\F$ is regular, then 
\[
CS(\F,C,{\bf 0})=0
\] where ${\bf 0}=(0,0)$. 

Let $\pi: \tilde{\mathbb A}^2_{\field} \to \mathbb A^2_{\field}$ be the blow-up at ${\bf 0}$, $E =\pi^{-1}({\bf 0})$  and let $\tilde{{\bf 0}} = \tilde{C}\cap E$.  The strict transform $\tilde{C}$ has equation $\tilde{f} = t^2-x$,  in local coordinates $\pi(x,t) = (x, tx)$.  The pullback foliation $\tilde{\F}$ is given the $\tilde{\omega} = tdx + xdt$ and we have the local decomposition 
\[
(2t)\cdot\tilde{\omega} = x\cdot d\tilde{f} + \tilde{f}(-dx). 
\]
This gives the Camacho-Sad index 
$CS(\tilde{\F},\tilde{C},\tilde{{\bf 0}})=2$, 
because 
\[
\frac{\eta}{h}|_{\tilde{C}} = \frac{-dx}{x}|_{\tilde{C}} = -2\frac{dt}{t}
\]
and $t$ is a local parameter for $\tilde{C}$ at $\tilde{{\bf 0}}$. In particular, we see that 
\[
CS(\F,C,{\bf 0}) = CS(\tilde{\F},\tilde{C},\tilde{{\bf 0}}) + 4 \quad \text{in} \quad \field 
\]
which agrees with Proposition~\ref{prop:stricttrans}. 

We note also that the corresponding  projective cubic curve on $\proj2$ has self-intersection $C^2 = 9$, while $\tilde{C}^2 = 8$. We can extend $\F$ to a projective foliation that has a single singularity at $(1:0:0)$. Since the cubic does not pass through the singular point of the foliation, then $CS(\F, C) = 0$. For $\tilde{C}$, we get $CS(\tilde{\F}, \tilde{C}) = 2$.  Both computations agree with Theorem~\ref{thm:CSformula}.  

\terminou
\end{ex}

\subsection{Riccati foliations on Hirzebruch surfaces} 

In this section, we  characterize regular Riccati foliations on Hirzebruch surfaces 
\[
\mathbb F_e=\mathbb P(\mathcal O_{\mathbb P^1_{\field}}\oplus \mathcal O_{\mathbb P^1_{\field}}(e)), \quad e\ge 0. 
\]
We say that $\F$ is a Riccati foliation with respect to a $\mathbb P^1_{\field}$-fibration $\pi: X\to C$ over a smooth curve $C$, if the generic fiber of $\pi$ is transverse to $\F$. This means that $tang (\F, F) = 0$, where $F$ represents the generic fiber.  Using Proposition~\ref{prop:tangformula}, we see that $\F$ is a Riccati foliation if and only if $K_{\F}\cdot F = 0$.  

In characteristic zero, there exists only one regular Riccati foliation $\F$ with respect to the projection $\mathbb F_e\to \mathbb P^1_{\field}$, which occurs when $e=0$.  In this case, $\mathbb F_0$ is the product $\mathbb P^1_{\field}\times\mathbb P^1_{\field}$,  and $\F$ corresponds to the second $\mathbb P^1_{\field}$-fibration. See, for instance, \cite[Theorem 3.7]{MR3328860}. We will see below that, in positive characteristic,   new examples arise.  

\begin{prop}\label{prop:regularHirz}
Let $\mathcal{F}$ be a regular Riccati foliation on a  Hirzebruch surface $\mathbb F_e$, with respect to the $\mathbb P^1_{\field}$-fibration $\pi:\mathbb F_e\to \mathbb P^1_{\field}$. Then:
\begin{enumerate}
\item $e=0$, $\mathrm{char} (\field)\ge 0$, and $\mathcal{F}$ is  the second $\mathbb P^1_{\field}$-fibration over $\mathbb P^1_{\field}$; or
\item $e>0$, $\mathrm{char} (\field) =p > 0$, $p$ divides $e$ and $\F$ is a Riccati foliation given by
\[
\omega  = dz + \gamma
\]
where $z$ intends as coordinate of the fiber and $\gamma\in \h^0(\mathbb P^1_{\field}, \Omega^1_{\mathbb P^1_{\field}}(e))$.
\end{enumerate}
\end{prop}

\proof 
Let $\mathcal{F}$ be a regular Riccati foliation on a  Hirzebruch surface $\mathbb F_e = \mathbb PE$, where $E = \mathcal O_{\mathbb P^1_{\field}}\oplus \mathcal O_{\mathbb P^1_{\field}}(e)$.   Let $\{U_i\}$ be a covering of $E$ where $E|_{U_i}$ is trivial with  transition functions 
\begin{eqnarray*}
G_{ij}=\left(
\begin{array}{ccc} 
1 & 0  \\
0 & g_{ij}  \\
\end{array}
\right)
\end{eqnarray*}
where $\{g_{ij}\}$ represents the line bundle $\mathcal O_{\mathbb P^1_{\field}}(e)$. Let us denote by 
\[
\omega_i = dz_i - \beta_iz_i^2+ \theta_i z_i+\gamma_i
\]
the family of $1$-forms defining $\F$ on each local trivialization, where $\beta_i,\theta_i, \gamma_i\in \Omega^1_{U_i}$ and $z_i$ denotes a local coordinate for the fiber over $U_i$.   We can write  $z_i = g_{ij}z_j$, then 
\[
dz_i = g_{ij}'z_j + g_{ij}dz_j
\] 
and using that the kernel of $\omega_i$ and $\omega_j$  coincide, we obtain
$\omega_i = g_{ij}\omega_j$.
From this, we conclude that the following conditions hold on $U_i\cap U_j$: 
\[
\beta_i=g_{ij}^{-1}\beta_j, \,\,\gamma_i=g_{ij}\gamma_j\quad \text{and} \quad \theta_j=\theta_i+dg_{ij}g_{ij}^{-1}. 
\]
Hence, we have
\begin{eqnarray*}
\left\{
\begin{array}{l}
\theta=\{\theta_i\} \quad \text{is a regular connection on} \quad \mathcal O_{\mathbb P^1_{\field}}(e)\\
\beta = \{\beta_i\} \quad \text{defines a global section of} \quad \Omega^1_{\mathbb P^1_{\field}}(-e)\\
\gamma=\{\gamma_i\}\quad \text{defines a global section of} \quad \Omega^1_{\mathbb P^1_{\field}}(e).
\end{array}
\right.
\end{eqnarray*}
 If $e=0$, then $\beta = \gamma = \theta = 0$ and $\F$ is  the second $\mathbb P^1_{\field}$-fibration of $\mathbb P^1_{\field}\times\mathbb P^1_{\field}$.  If $e>0$, then $\beta=0$, and  $p$ divides $e$,  otherwise $\mathcal O_{\mathbb P^1_{\field}}(e)$ has no regular connection. Since  $g_{ij}$ is a $p$-power, we conclude that $\theta$ is a global section of $\Omega^1_{\mathbb P^1_{\field}}$, so $\theta=0$. This gives the desired $1$-form $\omega = dz+\gamma$ defining $\F$  and concludes the proof of the proposition.  
\endproof

\subsection{Regular foliations on  minimal rational surfaces}
In this section, we study regular foliations on projective  minimal rational surfaces, over an algebraically closed field. Before this, we need the following result, whose proof is directly adapted from \cite[Chapter 2]{MR3328860}

\begin{prop} \label{rrat} Let $\mathcal{F}$ be a foliation on a  smooth projective  surface $X$. If $\mathcal{F}$ is regular, then 

\begin{eqnarray}\label{eq:relc_2withnandK}
N_{\mathcal{F}}\cdot K_{\mathcal{F}} = -c_{2}(T_{X})
\end{eqnarray}
\end{prop}
\begin{proof} The regular foliation $\F$ is given by a nonzero 
section $ v \in \h^0(X,T_{X}\otimes K_{\mathcal{F}})$, without zeros.   In particular,

	$$c_2(T_{X}\otimes K_{\mathcal{F}}) = 0.$$

On the other hand, we have
$$
    c_2(T_{X}\otimes K_{\mathcal{F}}) = c_2(T_{X})+c_1(T_{X})K_{\mathcal{F}} + K_{\mathcal{F}}^2 
$$
and using the formula $c_1(T_X) = N_{\mathcal{F}} - K_{\F}$,  we get the desired result
$$
      c_2(T_{X})+ N_{\mathcal{F}} \cdot K_{\mathcal{F}}= 0.
$$
This finishes the proof of the Proposition.
\end{proof}

The following observation has already been made in \cite[p.575]{MR1474805}.
\begin{cor}\label{cor:C2even}

Let $X$ be a smooth projective surface over an algebraically closed field of arbitrary characteristic. Suppose that $X$ admits a regular foliation $\mathcal{F}$. Then $c_2(T_X)$ is even.
\end{cor}
\proof 
 For any line bundle $L$ on $X$, the Riemann--Roch theorem gives
$$
\chi(L) = \chi(\mathcal{O}_X) + \tfrac{1}{2} L \cdot (L - K_X),
$$
from which the result follows by considering $L=N_{\F}^*$, $N_\F^* - K_X=-K_\F$ and \eqref{eq:relc_2withnandK}.
\endproof

\begin{remark}\rm\label{R:smoothnessns}
Using the same argument of Proposition~\ref{rrat}, a general formula including singular foliations is in fact established. See also \cite[Chapter 2]{MR3328860}. It takes the form
$$N_{\mathcal{F}}\cdot K_{\mathcal{F}} + c_{2}(T_{X})  = m(\mathcal{F}) $$
 where the right-hand side is nonzero (and then $>0$) if and only if the singular locus of the foliation is nonempty. It follows that the equality \eqref{eq:relc_2withnandK} provides a necessary and sufficient criterion for the regularity of $\F$.
 \terminou
\end{remark}

We now obtain a characterization of regular foliations on Hirzebruch surfaces.

\begin{lemma}\label{lemma:reghirz}
Let $\mathcal{F}$ be a regular foliation on  a Hirzebruch surface $\mathbb F_e$, $e\ge 0$, then either


\begin{enumerate}
\item $\mathrm{char} (\field)\ge 0$, and $\mathcal F$ is the $\mathbb{P}_{\field}^1$-fibration of $\mathbb F_e$ over $\mathbb{P}_{\field}^1$, i.e. $T_{\mathcal F}=T_{\mathbb F_e/\mathbb{P}_{\field}^1}$;  or
\item $\mathrm{char} (\field) =p > 0$, $p$ divides $e$ and $\F$ is a Riccati foliation on $\mathbb F_e$ given by
\[
\omega  = dz + \gamma
\]
where $z$ intends as coordinate of the fiber and $\gamma\in \h^0(\mathbb P^1_{\field}, \Omega^1_{\mathbb P^1_{\field}}(e))$.
\end{enumerate} 
\end{lemma}

\proof
Let $X=\mathbb F_e$. We let $F$ denote a fiber of the projection $\pi\colon X \longrightarrow \mathbb{P}_{\field}^{1}$ and let $M$ be a section such that $M^2 = -e$, $M\cdot F = 1$. Then, any line bundle is numerically equivalent to $aF+bM$ for some $a, b \in \mathbb{Z}$. The canonical bundle of $X$   can be written as 
\[
K_{X} \equiv (-e-2)F-2M.
\] 
Let us write the canonical bundle of $\F$ as 
\[
K_{\F} \equiv d_1 F + d_2 M.
\]

We assume that $\F$ is not the $\mathbb P^1_{\field}$-fibration over $\mathbb P^1_{\field}$. Since $F^2=0$, it follows from Proposition~\ref{prop:tangformula} that
\[
d_2  = tang (\F, F) \ge 0.
\]
We claim that $M$ is invariant by $\F$. Indeed, if it is not invariant, then 
\[
d_1 - ed_2 = K_{\F}\cdot M = tang(\F, M) +e.
\] 
Hence, we get
\begin{eqnarray}\label{eq:d1tang}
  d_1 = \tau_M + e(1+d_2)\ge 0
\end{eqnarray}
where $\tau_M:=tang(\F, M)$.   By adjunction formula $N_{\F} = K_{\F}-K_X$, we have
\[
N_{\mathcal{F}} \equiv (d_1+e+2)F+(d_2+2)M
\] 
and since $c_{2}(T_{\mathbb F_e}) = 4$ we get
\begin{eqnarray}\label{eq:KNd1d2}
-4 = K_{\mathcal{F}}\cdot N_{\mathcal{F}} = (d_2+2)(d_1-ed_2)+d_1d_2+d_2(e+2).
\end{eqnarray}
Now, using \eqref{eq:d1tang} we obtain 
\[
-4 = (d_2+2)(\tau_M+e)+d_2(d_1+e+2)
\]
which gives a contradiction, because the right hand side is non-negative. This shows that $M$ is invariant by $\F$.

When $M$ is invariant, from Proposition~\ref{prop:GM}, one obtains 
\[
d_1-ed_2 = K_{\F}\cdot M = -2.
\]
and from \eqref{eq:KNd1d2} we get 
\[
-4 = (d_2+2)(-2)+d_2(d_1+e+2). 
\]
Consequently, $d_2(e+d_1) = 0$, and this gives $d_2=0$ and $d_1=-2$. Therefore, $\F$ is a regular Riccati foliation on $\mathbb F_e$. The conclusion of the lemma follows from Proposition~\ref{prop:regularHirz}.  
\endproof

The next result provides a characterization of regular foliations on minimal rational surfaces, namely  the projective plane $\mathbb P^2_{\field}$ or a Hirzebruch surface $\mathbb F_e$, $e\ge 0$, $e\neq 1$. See \cite{Sha65}. In particular, we will see that $\mathbb P^2_{\field}$ does not admit any regular foliation. Regular foliations on arbitrary projective spaces will be addressed in  Section~\ref{section:regprojspaces}. 

\begin{thm}\label{thm:regularminimal} 
Let $\mathcal{F}$ be a regular foliation on a minimal rational smooth surface $X$.  Then $X$ is a Hirzebruch surface $\mathbb F_e$, $e\neq 1$,   and either
\begin{enumerate}
\item \label{I:fibrationcase}$\mathrm{char} (\field)\ge 0$, and $\mathcal{F}$ is the $\mathbb P^1_{\field}$-fibration of $\mathbb F_e$ over $\mathbb P^1_{\field}$; or
\item \label{I:exoticcase} $\mathrm{char} (\field) =p > 0$, $p$ divides $e$ and $\F$ is a Riccati foliation on $\mathbb F_e$ given by
\[
\omega  = dz + \gamma
\]
where $z$ intends as coordinate of the fiber and $\gamma\in \h^0(\mathbb P^1_{\field}, \Omega^1_{\mathbb P^1_{\field}}(e))$.
\end{enumerate} 
\end{thm}

\begin{proof} 

We know that $X$ is the projective plane or a Hirzebruch $\mathbb F_{e}$ for some $e\in \mathbb{N}$, $e\neq 1$, see \cite{Sha65}. It follows from Corollary~\ref{cor:C2even} that $\mathbb{P}_{\field}^{2}$ does not admit a regular foliation.


Now, if $X$ is a Hirzebruch surface $X = \mathbb F_{e}$ for some $e\ge 0$, $e\neq 1$, then the result follows from Lemma~\ref{lemma:reghirz}. 
\end{proof}

\begin{cor}\label{cor:minimalpclosed}
Let $\F$ be a regular foliation on a minimal rational smooth surface $X$ over an algebraically closed field of positive characteristic $p$. Then $X\simeq\mathbb F_e$, $e\neq 1$, and either
\begin{enumerate}
    \item\label{eqcor:item1} $\F$ is the standard ruling of $\mathbb F_e$, in which case $\F$ is $p$-closed; or
    \item\label{eqcor:item2} $p$ divides $e$, $\F$ is a regular Riccati foliation given by $\omega = dz + \gamma(x)dx$, with $\gamma(x)dx\in\h^0(\mathbb P^1_{\field}, \Omega^1_{\mathbb P^1_{\field}}(e))$. Furthermore, $\F$ is $p$-closed if and only if the derivative of $\gamma(x)$ of order $p-1$ with respect to $x$ vanishes. 
\end{enumerate}
\end{cor}

\proof
Let $\F$ be a regular foliation on a rational minimal surface $X$. By Theorem~\ref{thm:regularminimal}, we may assume  $X = \mathbb F_e$. Since the $\mathbb P^1_{\field}$-fibration is clearly $p$-closed, it is sufficient to consider the case where $p$ divides $e$ and $\F$ is a Riccati foliation on $\mathbb F_e$, locally given by $\omega  = dz + \gamma$ with $\gamma\in \h^0(\mathbb P^1_{\field}, \Omega^1_{\mathbb P^1_{\field}}(e))$.

Since the vector space $\h^0(\mathbb P^1_{\field}, \Omega^1_{\mathbb P^1_{\field}}(e))$ has dimension $e-1$, we can write 
\[
\gamma = \gamma(x)dx, \quad \gamma(x) = a_0+a_1x+\cdots +a_{e-2}x^{e-2}
\]
where $x$ denotes a coordinate for $\mathbb P^1_{\field}$ and $a_0,\dots, a_{e-2}\in \field$. Then $\F$ is locally given by the vector field
\[
v = \partial_x - \gamma(x)\partial_z. 
\]
This implies that $v^{(p)} = -\gamma^{(p-1)}(x)\partial_z$, where $\gamma^{(p-1)}$ means the derivative of $\gamma(x)$ order $p-1$ with respect to $x$. Therefore, $v\wedge v^{(p)}=-\gamma^{(p-1)}\partial_x\wedge \partial_z$. This concludes the proof of the corollary.
\endproof

One consequence of Corollary~\ref{cor:minimalpclosed} is that any regular foliation on $\mathbb F_p$ is $p$-closed.  For each $p>0$, we can produce regular foliations of Hirzebruch surfaces which are not $p$-closed: 

\begin{remark}\rm\label{rmk:HirzNonp}
 For each $p>0$, we construct a regular foliation on $\mathbb F_{2p}$ that is not $p$-closed. Any element $\gamma\in \h^0(\mathbb P^1_{\field}, \Omega^1_{\mathbb P^1_{\field}}(2p))$ induces a regular foliation by taking $\omega= dz+\gamma$.  To this end, let $\gamma=x^{p-1}dx$. The corresponding regular Riccati foliation $\F$ on $\mathbb F_{2p}$, locally given by the vector field $v = \partial_x - x^{p-1}\partial_z$ is not $p$-closed, because $v\wedge v^{(p)}=-\gamma^{(p-1)}\partial_x\wedge \partial_z$ is nonzero. 

 To explicitly show that $\F$ extends to a global regular foliation,  consider the $1$-form $\omega = dz + x^{p-1}dx$ in the local chart $(x,z)$. Applying the standard change of coordinates for $\mathbb F_{2p}$, namely $y=x^{-1}$ and $w=x^{-2p}z$, we find that $\F $ is given by $dw - y^{p-1}dy$ in the new chart. An analogous calculation confirms regularity in the remaining affine charts. 
 \terminou
\end{remark}

\begin{remark}\rm
Assume that $e=mp$ in Corollary~\ref{cor:minimalpclosed} \eqref{eqcor:item2}, with $m\geq1$, and write
\[
\gamma=\gamma(x)\,dx,
\qquad
\gamma(x)=\sum_{i=0}^{mp-2}a_i x^i.
\]
Only the monomials whose exponents are congruent to $p-1$ modulo $p$
contribute to the derivative of order $p-1$. Since $(p-1)!=-1$ in
$\field$, we obtain
\[
\gamma^{(p-1)}(x)
=
-\sum_{j=1}^{m-1}a_{jp-1}x^{(j-1)p}.
\]
Hence Corollary~\ref{cor:minimalpclosed} \eqref{eqcor:item2} gives
\[
\mathcal F \text{ is $p$-closed}
\quad\Longleftrightarrow\quad
a_{jp-1}=0
\quad\text{for every }1\leq j\leq m-1.
\]

This condition also characterizes the existence of a polynomial
primitive for $\gamma$. Indeed, the derivative of a polynomial has
zero coefficient at every monomial $x^{jp-1}$, while
$x^i\,dx=d(x^{i+1}/(i+1))$ whenever $p$ does not divide $i+1$.

Finally, let $C$ denote the Cartier operator on rational differentials
of $\field(x)/\field$. Its kernel consists precisely of exact
differentials; see \cite[Section~1, Proposition~2]{seshadri60}. Its action on $\gamma$ is
\[
C(\gamma)
=
\sum_{j=1}^{m-1}a_{jp-1}^{1/p}x^{j-1}\,dx.
\]
Consequently,
\[
\mathcal F \text{ is $p$-closed}
\quad\Longleftrightarrow\quad
C(\gamma)=0\Longleftrightarrow\quad
\gamma=dH \text{ for some }H\in\field[x]
.
\]
When $m=1$, the coefficient condition is empty, so $\mathcal F$ is
necessarily $p$-closed.
\terminou
\end{remark}



In \cite{ESTCompact}, the authors construct notable examples of regular foliations on non-rational surfaces defined over a number field,  whose reductions {\it mod} $p$ are not $p$-closed for infinitely many primes $p$, yet are $p$-closed for a set of primes of positive density. Those foliations have no compact leaves, despite being $p$-closed  {\it mod} $p$ for infinitely many $p$.

In what follows, we will  provide, for every prime $p>2$, an example of a regular Riccati foliation on $C\times \mathbb{P}_{\field}^{1}$,
where $C\subset\mathbb P^2_{\field}$ is an elliptic curve 
cut out by  
\begin{equation}\label{E:ellipticcurve}
f(w,x,y) = y^2w - x(x-w)(x-\lambda w), \qquad \lambda \not \in \{0,1\},
\end{equation}
which is  not, in general,  $p$-closed.

To this  goal, we recall the definition of supersingular elliptic curves, referring to \cite[\S 15 and \S22]{MumfordAbelian74} for further details.

\begin{dfn}
	Let $C$ be an elliptic curve over $\field$. Then $C$ is said to be {\bf supersingular} if its $p$-rank vanishes. An elliptic curve is {\bf ordinary} if it is not supersingular.
%
\end{dfn}

It is useful to know that an elliptic curve is supersingular if and only if the natural map on vector fields, \[\h^0(C, T_C)\longrightarrow\h^0(C,T_C) ,\quad v\longmapsto v^p,\] vanishes. 
Another important and well-known feature concerning the $p$-rank follows \cite[\S22]{MumfordAbelian74}.

\begin{prop}
	Let $C$ be an elliptic curve given by Equation~\eqref{E:ellipticcurve}. Then, $C$ is supersingular if and only if  $\lambda$ is a root of $$\Phi (t)= \sum_{\nu=0}^{\frac{p-1}{2}}{\binom{\frac{p-1}{2}}{ \nu}}^2 t^\nu .$$
\end{prop}

\begin{prop}\label{prop:examplechar3} 
	

Let $\field$ be an algebraically closed field of characteristic $p\geq 3$. Let $C\subset \mathbb{P}_{\field}^{2}$ be the smooth elliptic curve given by Equation~\eqref{E:ellipticcurve}.
Let $\mu\in \field^*$. Take $\omega_\mu = \mu \dfrac{ dx}{ 2y} \in \h^0(C,\Omega_{C}^{1})$ and consider the regular foliation $\F_\mu$  on $C\times \mathbb{P}_{\field}^{1}$ given by the $1$-form
$$
  \quad  \sigma = dz-z\omega_\mu
$$
where $z$ represents a coordinate of $\mathbb{P}_{\field}^{1}$. Then,

\begin{enumerate}
	\item If $C$ is supersingular then $\F_\mu$ is never $p$-closed.
	\item If $C$ is ordinary then $\F_\mu$ is $p$-closed if and only if  $\mu^{p-1}=\pm \Phi(\lambda)$, $\pm$ depending on whether $p\equiv  1\ mod\ 4$ or $p\equiv3\ mod\ 4$.
	\item In the case where $\F_\mu$ is not $p$-closed, the sections $\{z=0\}$ and $\{z=\infty\}$ are the only invariant curves.
\end{enumerate} 
\end{prop}

\begin{proof}

Write $v_1\in \h^0(C,T_{C})$, $\omega_1(v_1)=1$ for the vector field dual to $\omega_1$. The foliation $\F_\mu$, $\mu\not=0$  is thus defined by the regular and nowhere vanishing vector field $Y_\mu= z\partial_z +  \mu^{-1}v_1$ and is then $p$-closed iff 

$$v_1^p=\mu^{p-1}v_1,$$ thus proving the first item. Moreover, if we assume that $\F_\mu$ is not $p$-closed, one observes that ${Y_\mu}^p\wedge Y_\mu$ does not vanish outside  $\{z=0\}\cup \{z=\infty\}$, which establishes the third item. 

It remains to justify the second point of the statement. For this, set $P_1=(0,0)\in C$, $g=\dfrac{y}{x}\in \field(C)$  and remark that, $\mathrm{Res}_{P_1} g \omega_1 =1$ and also,  \cite[\textit{loc.cit}]{MumfordAbelian74} that $\mathrm{Res}_{P_1} g^p \omega_1 ={(-1)}^{ \frac{p-1}{2}}\Phi(\lambda)$. On the other hand,  in some formal coordinate $u$ at $P_1$, one can write $$\omega_1=\big( \sum_{i=0}^{p-2} a_i u^i+ hot\big)du,$$ $a_i\in \field,\ a_0\not=0$. After performing the change of variable  $t=\sum_{i=0}^{p-2}\frac{a_i}{{i+1}} u^{i+1},$ one gets $\omega_1= (1+\alpha t^{p-1}+ hot)dt$,   $\alpha\in \field$ so that $\mathrm{Res}_{P_1} g^p \omega_1=\alpha$. As $v_1= (1-\alpha t^{p-1} + hot)\partial_t$, an easy calculation yields 
$v_1^p= ( \alpha + hot) \partial_t=\alpha v_1$. The proof of the Proposition is now complete. 
\end{proof}


%
%
%

\subsection{Regular foliations on weak  del Pezzo surfaces} 

\subsubsection{The Mori cone of weak del Pezzo surfaces}
The following material is classical. See for instance Dolgachev's book \cite[Chapter 8]{Dolclas}.
Let $X$ be a {\bf weak del Pezzo surface}, i.e. a smooth projective surface such that
\[
- K_X \ \text{is nef and big}.
\]
One defines the \textbf{degree} of $X$ as the positive integer $K_X^2$. A weak del Pezzo surface is isomorphic to $\mathbb F_0$, $\mathbb F_2$ or to the blow-up of a bubble cycle on $\mathbb{P}_{\field}^2$ formed by  $r \le 8$ points, and $K_X^2 = 9-r$.


Let us consider the finite dimensional $\mathbb R$-vector space  ${\rm N_1}(X)_{\mathbb{R}}$ formed by $1$-cycles on $X$, modulo numerical equivalence. Let $\overline{{\rm NE}}(X)$ be the Mori cone of $X$, i.e. the cone spanned by classes of effective curves.

Recall that for weak del Pezzo surface of degree $\leq 7$,  $\overline{{\rm NE}}(X)$
is generated by the classes of (smooth rational) curves of self-intersection $-1$ or $-2$:

\[
\overline{{\rm NE}}(X)
=
\sum_{\text{$(-1)$-curves } C} \mathbb{R}_{\ge 0}[C]
\;+\;
\sum_{\text{$(-2)$-curves } C} \mathbb{R}_{\ge 0}[C].
\]

The sum is finite,  the curves intersect each other at most transversely, there are at most $r$ curves of self-intersection $-2$, moreover:
\begin{itemize}
\item $(-1)$-curves generate the $K_X$-negative extremal rays,
\item $(-2)$-curves lie in the hyperplane $K_X^\perp$.
\end{itemize}


The lattice
\[
K_X^\perp \subset \mathrm{Pic}(X)
\]
is negative definite and carries a root system of type:
\[
A_n,\ D_n,\ E_n.
\]
Moreover, the $(-2)$-curves correspond to effective roots.


The nef cone is given by
\[
\mathrm{Nef}(X)
=
\{ D \in N^1(X)_{\mathbb{R}} \mid D \cdot C \ge 0 \text{ for all $(-1)$ and $(-2)$-curves} \}.
\]


The following cases are of particular interest: 

\begin{itemize}
\item If $-K_X$ is ample, $X$ is a \textbf{del Pezzo surface}, there are no $(-2)$-curves, and $\overline{{\rm NE}}(X)$ is generated by $(-1)$-curves only.
\item A weak del Pezzo surface is said to be \textbf{strict} and we will abbreviate by \textbf{SWDP} if $-K_X$ is not ample. In this situation $(-2)$-curves appear and correspond to ADE configurations.
\end{itemize}

\subsubsection{Regular foliations on weak del Pezzo.}\label{section:weakdelpezzos} We start by noting that a weak del Pezzo surface carrying a regular foliation, has even degree.

\begin{thm}\label{thm:evendegree}
    Let $X$ be a smooth weak del Pezzo surface over an algebraically closed  field $\field$ of arbitrary characteristic.  Let $\F$ be a regular foliation on $X$. Then $X$ has even degree.
\end{thm}
\begin{proof}
    It suffices to notice that the degree of $X$ has the same parity as $c_2(T_X)$ and then conclude by Corollary \ref{cor:C2even}.
\end{proof}

The next result provides a characterization of regular foliations on del Pezzo surfaces. 

\begin{thm}\label{thm:delpezzo}
Let $X$ be a smooth projective del Pezzo surface, over an algebraically closed  field $\field$ of arbitrary characteristic.  Let $\F$ be a regular foliation on $X$. Then $X$ is either $\mathbb P^1_{\field} \times \mathbb P^1_{\field}$  or $\mathbb F_1$, in which case one recovers the description of $\F$ as given by Lemma \ref{lemma:reghirz} and Theorem \ref{thm:regularminimal}: $\F$ leaves one of the standard $\mathbb P^1_{\field}$-fibrations invariant. More precisely, if $X=\mathbb P^1_{\field} \times \mathbb P^1_{\field}$ then one of the two rulings is invariant, while for $\mathbb F_1$ the conclusion holds for the standard fibration $\mathbb F_1 \to \mathbb P^1_{\field}$.

\end{thm}

\begin{proof}
We first claim that for any smooth rational curve $C\subset X$, we have $N_\F\cdot C\geq 0$. Indeed, it follows from Propositions~\ref{prop:GM} and \ref{prop:tangformula}, and Camacho-Sad formula (cf. Theorem~\ref{thm:CSformula}), that  $N_\F \cdot C \ge 2 $ when $C$ is not invariant by $\F$, and $p=\mathrm{char}(k)$ divides $N_\F\cdot C=C^2$, otherwise.  In the last identity we used that $Z(\F,C_i)=0$, which follows from the fact that  $C$ is a smooth curve and $\F$ is regular. By adjunction formula, we obtain $C^2>-2$, therefore $N_\F\cdot C\ge 0$. 

We show first that the case where $X$ is obtained by blowing-up $\mathbb P^2_{\field}$  at $n$ points with $2\le n \le 8$, i.e $X$ has degree $\leq 7$, cannot happen. In this case, each extremal ray $C$ is a $(-1)$-curve, then not invariant under $\F$. This gives $K_\F\cdot C\ge 1$. In particular, $N_\F$ and $K_\F$ are both nef and then $N_\F\cdot K_\F\ge 0$. On the other hand, it follows from   Proposition~\ref{rrat} that 
\[
N_\F\cdot K_\F = -c_2(T_X) < 0
\]
giving a contradiction.

Then the conclusion of the theorem follows from Lemma~\ref{lemma:reghirz} and from the fact that $\mathbb P^2_{\field}$ does not admit  any regular foliation. 
\end{proof}

\begin{cor}\label{cor:delpezzopclosed}
Any  regular foliation on a smooth del Pezzo surface over an algebraically closed field of positive characteristic $p$ is $p$-closed.   
\end{cor}

It follows from \cite{Druel15} that any regular foliation on a weak Fano complex variety $X$ is given by the fibers of a smooth morphism $X\to Y$ onto a projective manifold. It has been conjectured that the same conclusion holds  more generally for rationally connected smooth complex varieties. The following question points in this direction, although it remains open even for surfaces in positive characteristic.

\begin{question}
Let $X$ be a  smooth projective rationally connected variety over a field $\field$ of positive characteristic $p$, and let $\F$ be a regular foliation on $X$. Under what conditions is the foliation $p$-closed?
\end{question}

Regarding the above question, in view of Remark~\ref{rmk:HirzNonp}, additional hypotheses are required to ensure $p$-closedness. Concerning weak del Pezzo surfaces, we will see that regular and not $p$-closed foliations occur for $p=2$ on a surface of degree $K_X^2=4$. 
This justifies our subsequent investigation into the structure of regular foliations on weak del Pezzo surfaces. 

\begin{thm}\label{thm:deg6implieschar2}
  Let  $X$ a weak del Pezzo surface of degree different from $8$ over an algebraically closed field $\field$. Let $\F$ be a regular foliation on $X$. Then $X$ is a SWDP and $\mathrm{char}(\field)=2$
\end{thm}
\begin{proof}

The case of del Pezzo surface is ruled out by Theorem \ref{thm:delpezzo}, only the degree $8$ case can occur. We argue as in the proof of Theorem \ref{thm:delpezzo}, using now that if $\mathrm{char}(\field)\not=2$, none of the $(-1)$ or $(-2)$-curves whose classes generate the Mori cone can be $\F$-invariant. As before, this gives $N_\F \cdot C\geq 2$, $K_\F \cdot C \geq 1$, which implies in turn the nefness of $N_\F$ and $K_\F$, contradicting again 
\[
N_\F\cdot K_\F = -c_2(T_X) < 0.
\]
\end{proof}
\begin{remark}\rm\label{rmk:weakdeg8}
  The description of regular foliations on $\mathbb F_2$ (the unique $SWDP$ surface of degree $8$) is covered by  Theorem~\ref{thmintro:regularminimal}. There is only one foliation given by the standard ruling, except in characteristic $2$  where we also have the exceptional family described in Item~\eqref{item2} of the Theorem~\ref{thmintro:regularminimal}. Every foliation in this family is $2$-closed.

  The remaining del Pezzo surfaces of degree $8$, namely $\P^1_{\field}\times \P^1_{\field}$ and $\mathbb F_1$ are covered by Theorem~\ref{thm:delpezzo} and all the regular foliations are $p$-closed.
   \terminou
  \end{remark}

\begin{remark}\rm\label{rmk:existencechr2}
It is easy to exhibit, in  characteristic $2$, regular foliations on weak del Pezzo surface  in every even degree $\leq 6$.
Indeed, start with a smooth projective surface $X$ over a field of characteristic $2$ carrying a $2$-closed foliation $\F$.  Consider the blow-up at  a  point $p$ not contained in $\rm{Sing}(\F)$. As $\F$ is $2$-closed, it is defined at $p$ by $dx=0$ where $x$ is a suitable formal coordinate. The strict  transform of the foliation  carries exactly one singularity along the exceptional divisor given formally by $xdy+ydx=0$ and that can be resolved by an additional blow-up, giving rise to an exceptional divisor $E$ transverse to the strict transform $\tilde{\F}$ of $\F$ by the composition $\pi$ of these two blow-ups. Note that $\tilde \F$  is still a regular foliation in the neighborhood of $\pi^{-1} (p)$ in the resulting surface $\tilde{X}$.
In particular, if one takes $\mathbb F_1$ equipped with its canonical $\P^1$-fibration, one obtains a regular foliation $\F_r$ on a weak del Pezzo degree $8-2r$ by performing $2r$ successive blow-ups in a suitable way. 

Moreover, a simple calculation yields $N_{\tilde{\F}}= \pi^* N_\F -2E$. Consequently,   $N_{\tilde{\F}}^2=N_{{\F}}^2-4$. As shown in the proof of Lemma \ref{lemma:reghirz},  $N_{\mathcal{F}_0}\equiv 2M +eF$  where $M$ denotes the section of self intersection $-e$, thus implying the vanishing of  $N_{\mathcal{F}_0}^2$ and then proving that $N_{\tilde{\F}_r}^2= -4r\not=0$ for $r>0$. Compare with Proposition~\ref{thm:bottvanishingC1}.\terminou
\end{remark}

\begin{thm}\label{TH:degree6}
Let $X$ be a weak del Pezzo surface of degree $6$ over an algebraically closed field of characteristic $2$, and let $\mathcal{F}$ be a regular foliation on $X$.  Then $\mathcal{F}$ is obtained from the standard ruling on $\mathbb F_1$ by a succession of two blow-ups according to the process described in Remark~\ref{rmk:existencechr2}.
\end{thm}

We prove Theorem~\ref{TH:degree6}  below, before this, we offer some consequences.
\begin{cor} 
Assume the characteristic is $2$. Every regular foliation $\F$ on a weak del Pezzo surface of degree $6$ or $8$ is birationally equivalent to the foliation defined by the standard ruling on $\mathbb F_1$. In particular, $\F$ is $2$-closed.   
\end{cor}

\proof
Any regular foliation on a weak del Pezzo surface of degree 8 is $2$-closed, see Remark~\ref{rmk:weakdeg8}. Then performing two blow-ups as in Remark~\ref{rmk:existencechr2} we obtain a regular foliation on a weak del Pezzo surface of degree 6. Therefore, the conclusion follows from Theorem~\ref{TH:degree6}. 
\endproof

\begin{cor}\label{cor:6and8pclosed}
Assume the characteristic is $p>0$. Any regular foliation on a weak del Pezzo surface of degree $6$ or $8$ is $p$-closed.     
\end{cor}

\proof
The case of degree $8$ is discussed in Remark~\ref{rmk:weakdeg8}, while the case of degree $6$ is addressed by  Theorem~\ref{thm:deg6implieschar2} and Theorem~\ref{TH:degree6}. 
\endproof

Before proving Theorem~\ref{TH:degree6}, let us introduce some notation. Let $X$ be a weak del Pezzo surface of degree $\leq 7$ over an algebraically closed field of characteristic  $p=2$. Assume that $X$ carries a regular foliation $\F$. Denote the finite collection of $(-2)$-curves by $(C_i)_{i\in I}$ and the finite collection of $(-1)$-curves by $(L_j)_{j\in J}$.  According to the terminology in \cite{Dolclas}, the $L_j$'s will be called {\bf lines}. In the subsequent analysis, we shall repeatedly employ the following tangency formulae: 
\begin{itemize}
    \item $N_\F\cdot C_i\geq -2$ , $K_\F\cdot C_i\geq -2$ whenever $C_i$ is invariant; 
    \item $N_\F\cdot C_i\geq 2$ , $K_\F\cdot C_i\geq 2$ if $C_i$ is not invariant;
    \item $N_\F\cdot L_j\geq 2$ , $K_\F\cdot L_j\geq 1$ for every $j$ ($L_j$ is never $\F$-invariant).
\end{itemize}
Note also that,  by virtue of Lemma \ref{L:normalcrossing}, one has necessarily $C_i\cdot C_k=0$ whenever $C_i$ and $C_k$ are distinct invariant $(-2)$-curves. Recall that there exists at least one invariant $(-2)$-curve, otherwise we would have $N_{\F}$ and $K_{\F}$ nef, contradicting the inequality
\[
N_\F\cdot K_\F = -c_2(T_X) < 0.
\]


The next result will be useful in the proof of Theorem~\ref{TH:degree6}.  

\begin{lemma} \label{L:nonnefness} 
Assume the characteristic is $2$. Let $X$ be a weak del Pezzo surface of degree $9-r$, $r\geq 3$ and $\F$ a regular foliation on $X$ (in particular $r$ is an odd number, see Theorem~\ref{thm:evendegree}). Let $C_1, \dots, C_n$ be pairwise distinct $(-2)$ invariant curves and assume $n\leq \frac{r+1}{2}$.  Then either
\[
  N_\F- \sum_i C_i \quad \text{or} \quad K_\F-\sum_i C_i
\]
is not nef.
\end{lemma}
\begin{proof}
    By invariance of the $C_i$'s and regularity of the foliation, one has 
    \[
    K_\F\cdot N_\F=-c_2(X)=-3-r
    \]
    and 
    \[
    N_\F\cdot C_i= K_\F\cdot C_i=-2 \quad \text{for every $i$}. 
    \] 
    
    On the other hand, it follows from Lemma \ref{L:normalcrossing} that $C_i\cdot C_j=0$ whenever $i\not=j$. 
    This implies
    \[
    \big(N_\F- \sum_i C_i\big)\cdot \big(K_\F-\sum_i C_i\big)= -3-r+2n\leq -2<0.
    \]
\end{proof}

\subsubsection{Proof of Theorem \ref{TH:degree6}}   
Following the table given in \cite{Dolclas}, we review all possible cases for $X$ assuming a priori that $X$ admits a regular foliation $\F$. Recall that $C_i\cdot C_j=0$ whenever $C_i$ and $C_j$ are distinct invariant $(-2)$-curves. Also, there exists at least one invariant $(-2)$-curve. 

\medskip

The configuration of the $(-1)$ and $(-2)$ curves will be represented by their incidence graph, see Figures~\ref{graph:A11} to \ref{graph:A2+A1}.

\begin{figure}
	\begin{tikzpicture}
		\node[shape=circle,fill=gray!25, draw=black] (L1) at (-3,0) {$L_1$};
		\node[shape=circle,draw=black] (C) at (0,0) {$C$};
		\node[shape=circle,fill=gray!25, draw=black] (L2) at (3,0) {$L_2$};
		\node[shape=circle,fill=gray!25, draw=black] (L3) at (0,-3) {$L_3$};
		\path [-](L1) edge node[left] {$ $} (C);
		\path [-](L2) edge node[left] {$ $} (C);
		\path [-](L3) edge node[left] {$ $} (C);
	\end{tikzpicture}
	\caption{$A_1$ with three lines }
	\label{graph:A11}
\end{figure}
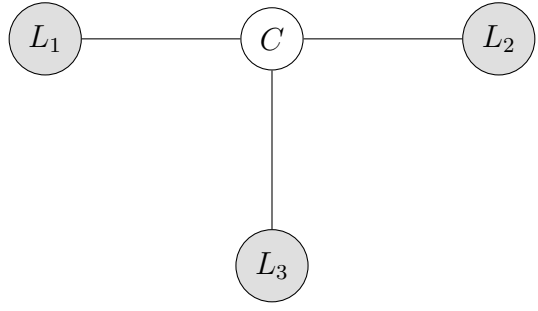

\begin{figure}
\begin{tikzpicture}
	
	\node[shape=circle,fill=gray!25, draw=black] (L1) at (0,3) {$L_1$};
	\node[shape=circle,draw=black] (C1) at (0,0) {$C_1$};
	\node[shape=circle,fill=gray!25, draw=black] (L2) at (3,0) {$L_2$};
	\node[shape=circle,draw=black] (C2) at (3,-3) {$C_2$};
	\path [-](L1) edge node[left] {$ $} (C1);
	\path [-](L2) edge node[left] {$ $} (C1);
	\path [-](C2) edge node[left] {$ $} (L2);
	
\end{tikzpicture}
\caption{$2A_1$ with two lines }
\label{graph:2A1}
\end{figure}
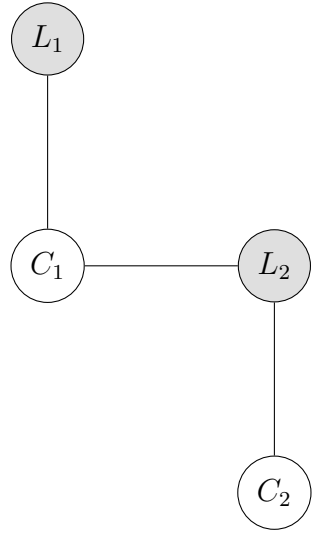

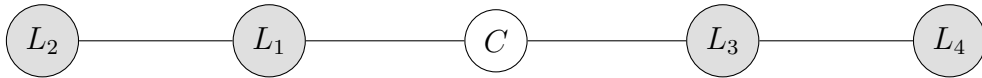
\begin{figure}
	\begin{tikzpicture}
	\node[shape=circle,fill=gray!25, draw=black] (L1) at (-3,0) {$L_1$};
	\node[shape=circle,draw=black] (C) at (0,0) {$C$};
	\node[shape=circle,fill=gray!25, draw=black] (L2) at (-6,0) {$L_2$};
	\node[shape=circle,fill=gray!25, draw=black] (L3) at (3,0) {$L_3$};
	\node[shape=circle,fill=gray!25, draw=black] (L4) at (6,0) {$L_4$};
	\path [-](L1) edge node[left] {$ $} (C);
	\path [-](L2) edge node[left] {$ $} (L1);
	\path [-](L3) edge node[left] {$ $} (C);
	\path [-](L3) edge node[left] {$ $} (L4);
	\end{tikzpicture}
	\caption{$A_1$ with four lines }
	\label{graph:A12}
\end{figure}

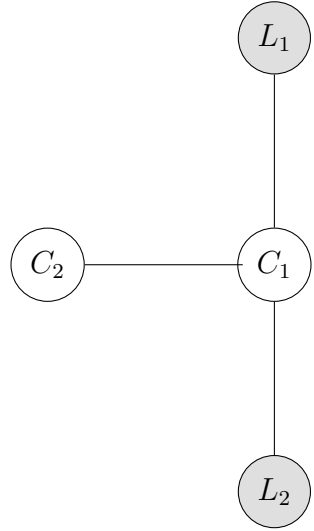
\begin{figure}
	\begin{tikzpicture}
	\node[shape=circle,fill=gray!25, draw=black] (L1) at (0,3) {$L_1$};
	\node[shape=circle,draw=black] (C1) at (0,0) {$C_1$};
	\node[shape=circle,fill=gray!25, draw=black] (L2) at (0,-3) {$L_2$};
	\node[shape=circle, draw=black] (C2) at (-3,0) {$C_2$};
	\path [-](L1) edge node[left] {$ $} (C1);
	\path [-](L2) edge node[left] {$ $} (C1);
	\path [-](C2) edge node[left] {$ $} (C);
	
	\end{tikzpicture}
	\caption{$A_2$ with two lines }
	\label{graph:A2}
\end{figure}

\begin{figure}
	\begin{tikzpicture}
	\node[shape=circle,fill=gray!25, draw=black] (L1) at (6,0) {$L_1$};
	\node[shape=circle,draw=black] (C1) at (0,0) {$C_1$};
	\node[shape=circle, draw=black] (C3) at (9,0) {$C_3$};
	\node[shape=circle, draw=black] (C2) at (3,0) {$C_2$};
	\path [-](L1) edge node[left] {$ $} (C2);
	\path [-](C2) edge node[left] {$ $} (C1);
	\path [-](C3) edge node[left] {$ $} (L1);

	\end{tikzpicture}
	\caption{$A_1+A_2$ with one line }
	\label{graph:A2+A1}
\end{figure}
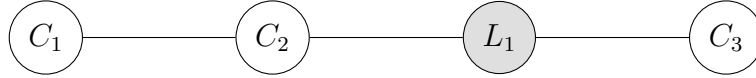

\subsubsection*{{\bf First case: singularity of type $A_1$ with 3 lines}} In this case,  $C$ is the only $(-2)$-curve, hence invariant. Set $L_N=N_\F -C$. Then $L_N \cdot L_i\geq 1$, $L_N\cdot C= 0$.

 Set $L_T= K_\F -C$. Then $L_T\cdot L_i\geq 0$, $L_T\cdot C=0$. Hence $L_T$ and $L_N$ are nef, but this contradicts Lemma \ref{L:nonnefness}.

\subsubsection*{{\bf Second case: singularity of type $2A_1$}}
 
If $C_i$ is invariant but not $C_j$ ($\{i,j\}=\{1,2\}$), set $L_N=N_\F -C_i$, $L_T= K_\F -C_i$. As before, this implies that $L_N$ and $L_T$ are nef, and we obtain the same contradiction as before. 

Thus we are led to consider the only possible situation where both $C_1$ and $C_2$ are invariant. First assume that $tang(\F, L_2)\geq 1$. Then $L_N= N_\F-(C_1+C_2)$ and $L_T= K_\F -(C_1+C_2)$ are nef, again a contradiction with Lemma \ref{L:nonnefness}.

Hence $\F$ is completely transverse to $L_2$. It follows that $\F$ is the pull-back of a singular foliation $\overline{\F}$ defined on the surface $\overline{X}$ obtained by  contracting $L_2$ via $\pi: X\to \overline{X}$.  Let $D_i:={\pi}_*(C_i)$ for $i=1,2$.  By Formula~\eqref{eq:ZbytildeZ}, $Z(\overline{\F}, D_i)=1$. Let $\pi_i: \overline{X} \to X_i$ be the contraction of the $(-1)$-curve $D_i$. From Formula~\eqref{Zburegularinv}, one can derive that $\overline{\F}$ is the pull-back of a \textit{regular} foliation ${\F_i}$ on $X_i$.
Indeed,  when $D_1$ is contracted, ${\F}_1$ is the horizontal (or vertical) foliation on ${X}_1=\P^1\times \P^1$, whereas contracting $D_2$ yields the vertical foliation on ${X}_2=\mathbb F_1$, with the $(-1)$-section corresponding to $L_1$. 

A posteriori, we conclude that $L_1$ was also transverse to $\F$.  By contracting successively $L_1$ and $C_1$, we obtain ${X}_3=\mathbb F_2$ and for the same reasons, the induced foliation on $X_3$ is one of the foliations transverse to the vertical fibration on $\mathbb F_2$. Conversely, any regular foliation on $\mathbb F_2$ transverse to the vertical fibration gives rise, via two suitable blow-ups, to a  regular foliation on a weak Fano surface having a configuration of type $2A_1$. 

\subsubsection*{{\bf Third case: singularity of type $A_1$ with 4 lines.}} In this case there is a single (-2)-curve $C$, which is invariant. Then $L_N=N_\F -C$ and $L_T= K_\F -C$ are nef: contradiction.

\subsubsection*{{\bf Fourth case: singularity of type $A_2$ with 2 lines}}

Either $C_1$ or $C_2$ is invariant, but not both simultaneously, because their intersection is nonempty. If $C_i$ is invariant, then $L_N=N_\F -C_i$, $L_T= K_\F -C_i$ are nef: contradiction.

\subsubsection*{{\bf Fifth case: singularity of type $A_1 + A_2$ with one line}}
At least one of the curves $C_i$ is invariant, but $C_1$ and $C_2$ are not simultaneously invariant.
If $C_1$ is invariant and $C_3$ is not, then $L_N=N_\F -C_1$ and $L_T= K_\F -C_1$ are nef. If both $C_1$ and $C_3$ are invariant, then $L_N=N_\F -C_1 -C_3$ and $L_T= K_\F -C_1-C_3$ are nef. If $C_3$ is the only invariant curve, then $L_N=N_\F -C_3$ and $L_T= K_\F -C_3$ are nef. If $C_3$ and $C_2$ are invariant and $tang(\F, L_1)\geq 1$, then $L_N=N_\F -C_2-C_3$ and $L_T= K_\F -C_2-C_3$ are nef. In all these cases, we reach a contradiction.

The only remaining case is when $C_2$ and $C_3$ are invariant and $tang(\F, L_1)=0$. As in the second case, $\F$ is then the pull-back of a regular foliation $\overline{\F}$ on the degree $6$ del Pezzo surface $\overline{X}$ obtained by contracting successively $L_1$ and then $C_2$ or $C_3$ (more precisely, their image under the first contraction). Contracting $C_2$ yields the vertical foliation on $\mathbb F_1$, whereas contracting $C_3$ yields the vertical foliation on $\mathbb F_2$.


\subsubsection*{{\bf Summary}} Combining the second and fifth cases, we see that all regular foliations on del Pezzo surfaces of degree $6$ are obtained from the vertical foliation on $\mathbb F_1$ by two blow-ups, as wanted. This concludes the proof of Theorem~\ref{TH:degree6}.

\subsubsection{Regular foliations on weak del Pezzo surfaces of degree $4$}

We study the  configuration defined by the incidence graph of Figure~\ref{graph:4A1}. Let $X$ be the corresponding surface (the characteristic of the ground field is $2$). We will describe all regular foliations on $X$. 

\begin{thm}\label{thm:4invarcurevs}
Assume the characteristic is $2$. Let $X$ be a weak del Pezzo surface of degree $4$ having the incidence graph of Figure~\ref{graph:4A1} and let $\F$ be a regular foliation on $X$. Then $\F$ leaves all the $C_i$ invariant. 
\end{thm}

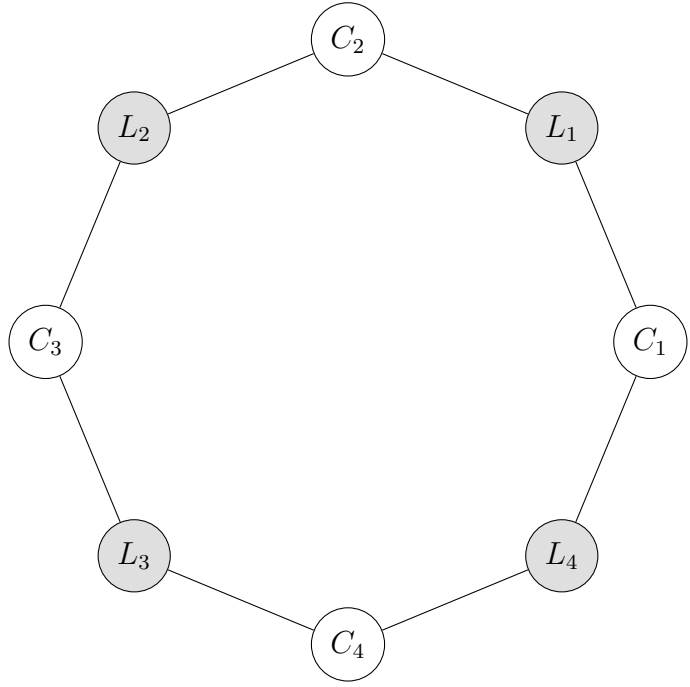
\begin{figure}

\begin{tikzpicture}
	\foreach \ii in {1, ..., 4}{
		\pgfmathsetmacro{\angle}{(2*\ii-1)*45}
		\node[circle, draw,fill=gray!25] (L\ii) at (\angle:4cm) {$L_{ \ii }$};
		
	}
	\foreach \ii in {1, ...,4 }{
		\pgfmathsetmacro{\angle}{(2*(  \ii-1)*45)}
		\node[circle, draw] (C\ii) at (\angle:4cm) {$C_\ii$};
		\draw (L\ii) --  (C\ii);
	}
	\foreach \ii in {1,...,3}{
		\pgfmathtruncatemacro{\next}{\ii+1}
		\draw (L\ii) -- (C\next);
	}
	\draw (L4) --  (C1);

	\end{tikzpicture}
\caption{$4 A_1$ with four lines }
\label{graph:4A1}
\end{figure}
\begin{proof}
We discuss according to the number $n$ of possible invariant $(-2)$-curves.

\subsubsection*{{\bf First case: $n=1$}} We may assume that $C_1$ is invariant. Then the linebundles
\[ 
L_N=N_\F -C_1 \quad \text{and} \quad L_T= K_\F -C_1
\]
are nef, contradicting Lemma \ref{L:nonnefness}.

\subsubsection*{{\bf Second case: $n=2$}} By symmetry, we may assume either the pair $(C_1, C_2)$ or the pair $(C_1, C_3)$ is invariant. Suppose $C_2$ is invariant. If $tang(\F, L_1)\geq 1$, then the following linebundles are nef
\[
L_N=N_\F -C_1-C_2 \quad \text{and} \quad L_T= K_\F -C_1-C_2
\]
again contradiction with Lemma \ref{L:nonnefness}. Hence $\F$ is transverse to $L_1$, and contracting $L_1$ and $C_1$ leads to the degree $6$ surface of type $2A_1$, where the corresponding foliation is regular (for the reasons given in the study of previous cases, using Formula~\eqref{Zburegularinv} to conclude that the transformed foliation is regular). This is the second case in the proof of Theorem \ref{TH:degree6}, then the foliation leaves two $(-2)$-curves invariant,  namely those corresponding to $C_3$ and $C_4$, yielding a contradiction.

If $C_1$ and $C_3$ are invariant, we directly obtain that 
\[
L_N=N_\F -C_1-C_3 \quad \text{and} \quad L_T= K_\F -C_1-C_3
\]
are nef,  again we get a contradiction.

\subsubsection*{{\bf Third case: $n=3$}} By symmetry, assume $C_1, C_2, C_3$ are invariant. If $tang(\F, L_1)\geq 1$ and $tang(\F, L_2)\geq 1$, then 
\[
L_N=N_\F -C_1-C_2-C_3 \quad \text{and} \quad L_T= K_\F -C_1-C_2-C_3
\]
are nef, giving a contradiction. Hence, we may assume $tang(\F, L_1)=0$, which leads to the same contradiction as in the case $n=2$.

Based on the preceding analysis, we conclude that the only viable case is one where $\F$ leaves all the $C_i$ invariant.
\end{proof}

In the next result, we construct a 5-parameter family of regular foliations on the rational surface $X$ that are not $2$-closed.

\begin{thm}\label{thm:5parameter}
Assume the characteristic is $2$. Let $X$ be a weak del Pezzo surface of degree $4$ having the incidence graph of Figure~\ref{graph:4A1}. Then $X$ carries a $5$-parameter family of regular foliations such that each member of the family is not $2$-closed.
\end{thm}

\begin{proof}
We are looking for regular foliations on $X$ which are not $2$-closed. By Theorem~\ref{thm:4invarcurevs}, we can assume that all the $C_i$'s are invariant. Moreover, we claim that we may assume $tang(\F,L_i) = 1$. First, suppose one of the $L_i$, say $L_1$, is transverse to $\F$. Contracting $L_1$ and then $C_1$ or $C_2$ yields the degree $6$ model of type $2A_1$, whose corresponding foliation is in particular $2$-closed.
Now suppose that for all $i$, $tang(\F,L_i)\geq 1$. One checks that 
\[
L_N=N_\F -C_1-C_2-C_3-C_4 \quad \text{and} \quad L_T= K_\F -C_1-C_2-C_3-C_4
\]
are nef and $L_N \cdot L_T =0$. Moreover, for all $i$, we have 
\[
L_N\cdot L_i= tang(\F,L_i)>0 \quad \text{and} \quad L_T\cdot L_i= tang(\F,L_i) -1.
\]
Also, $L_N =  L_T -K_X$ has nonzero self-intersection. By  Hodge index theorem, we conclude that $L_T\equiv 0$, hence $tang(\F,L_i)=1$.  This proves the claim above. 

By contracting simultaneously the four $(-1)$-curves $L_i$, we obtain $\bar X=\P^1 \times \P^1$. The curves $C_i$ are mapped to four $\bar\F$-invariant lines $\mathcal L_i$, namely, in suitable projective coordinates $(x,y)$:
$$\mathcal L_1= \P^1\times \{0\}, \quad \mathcal L_2=\{\infty\}\times \P^1, \quad \mathcal L_3=\P^1\times\{ \infty\}, \quad \mathcal L_4 =\{0\}\times \P^1,$$ where $\bar\F$ is the (singular) foliation induced by $\F$ on $\bar X$.

The cycle $\mathcal C=\cup_i \mathcal L_i$ contains exactly four singular points of $\bar \F$, located at the intersections of adjacent lines. In cyclic order: $p_1=\mathcal{L}_1\cap \mathcal{L}_2$, \dots, $p_4= \mathcal{L}_4\cap \mathcal{L}_1$, and for all $i$, 
$$
Z(\bar\F, \mathcal L_i, p_i)=Z(\bar\F, \mathcal L_{ i+1}, p_i)=2 
$$ 
as a byproduct of \eqref{eq:ZbytildeZ}. 
We then  obtain 
\[
Z(\bar{\F},\mathcal L_i) = 4 = N_{\bar \F}\cdot \mathcal L_i.
\]
This implies that  $N_{\bar\F} \equiv 4 (\mathcal L_2 +\mathcal L_3)$. Also, by \eqref{eq:tangbu}, each singularity has vanishing order: 
\begin{eqnarray}\label{eq:a(p)}
a(p_i)=2 \quad \forall \, i= 1, \dots, 4.
\end{eqnarray}

In the coordinates $(x,y)$, the foliation $\bar \F$ is defined by a polynomial $1$-form
$$\omega= U(x,y)dx +V(x,y) dy,$$
where $U$ and $V$ are coprime and the polar divisor is
\begin{eqnarray}\label{eq:omegainfty}
{ (\omega)}_\infty=4 (\mathcal L_2 +\mathcal L_3).
\end{eqnarray}
Combining \eqref{eq:a(p)} and \eqref{eq:omegainfty} with the invariance of each $\mathcal L_i$,  one can verify (by a tedious but easy calculation) that $\omega$ has the form
$$\omega= (xyP(y)+\lambda y^2 +\mu x^2y^2)dx +(xy Q(x) +\alpha x^2+\beta x^2 y^2)dy,$$
where $P,Q$ are polynomials of degree $\leq 2$ and $\alpha, \beta ,\lambda , \mu \in \field^\times$.

Moreover, the strict transform of $\bar \F$ obtained by the four blow-ups at $p_i$, $i=1,...,4$, is regular 
if and only if 
$$P(y)=\alpha +ay +\beta y^2,\quad Q(x)=\lambda + bx +\mu x^2,$$
with $a\neq b\in \field$.  Indeed,  by Remark~\ref{R:smoothnessns},  it suffices to observe that this strict transform has no singularities on the octagonal cycle defined by the $(-1)$ and $(-2)$ curves on $X$ to conclude that it has no singularities at all. Note that the parameters $\beta$ and $\mu$ reappear in $P(y)$ and $Q(x)$, respectively, precisely to eliminate any potential singularities outside the set $\{p_1, \dots, p_4$\}.

The parameters involved are $\alpha, \beta, a, b, \lambda, \mu$, and since the foliation is defined by the $1$-form $\omega$ up to a scalar multiple, this yields a five parameters family of foliations.

We now verify that the corresponding family of regular foliations on $X$ satisfies the conclusion of the theorem.  By a change of coordinates $(x,y)\mapsto (sx, ty)$ with 
\begin{eqnarray}\label{eq:st}
\mu s^3 t^2= \beta s^2 t^3=1,
\end{eqnarray}
one can normalize $\beta=\mu=1$. To see that there exist $s,t$ satisfying \eqref{eq:st}, consider the curves $\mu s^3t^2-u^5=0$ and $\beta s^2t^3-u^5=0$, in homogeneous coordinates $(s:t:u)$. These curves intersect each other along the line $u=0$ at exactly two points, each with intersection multiplicity 10. Since the product of their degree is 25,  B\'ezout's theorem implies the existence of further intersection points in the affine chart $u=1$, then satisfying \eqref{eq:st}.  With this normalization, we have
\[
\omega =(xyP(y)+\lambda y^2 +  x^2y^2)dx +(xy Q(x) +\alpha x^2+ x^2 y^2)dy,
\]
with 
\[
P(y)=\alpha +ay + y^2,\quad Q(x)=\lambda + bx +  x^2,\quad a\neq b\in \field.
\]
By considering the rational vector field defining $\bar\F$
\[
v=(xy Q(x)+\alpha x^2+ x^2 y^2)\partial_x - (xyP(y)+\lambda y^2 +  x^2y^2)\partial_y
\]
one can readily check that $v\wedge v^{(2)}\not=0$. This concludes the proof.
\end{proof}

The following result is implicit in the proof of Theorem~\ref{thm:5parameter}. 

\begin{thm}
    Assume the characteristic is $2$. Let $X$ be a weak del Pezzo surface of degree $4$ having the incidence graph of Figure~\ref{graph:4A1} and let $\F$ be a regular foliation on $X$. Then:
    \begin{itemize}
        \item either $\F$ is $2$-closed and can be obtained from a regular foliation on a weak del Pezzo surface of degree 6 after two blow-ups; or
        \item $\F$ is not $2$-closed and lies in the $5$-parameters family of foliations of Theorem~\ref{thm:5parameter}. Furthermore, $\F$ can be obtained from a (singular) foliation on $\P^1\times \P^1$ after $4$ blow-ups.  
    \end{itemize}
\end{thm}

We finish this section with an  example that lies in the above $5$-parameter family.  

\begin{ex}\rm \label{ex:regularnon2closed}
The following provides an explicit example of a regular foliation on the degree 4 weak del Pezzo $X$ as above, that is not $2$-closed. We begin with a foliation  on $\P^1\times \P^1$ given in affine coordinates $(x,y)$ by the $1$-form
    \[
    \omega =\big(xy(1+y^2) + y^2 + x^2y^2\big)dx +\big(xy(1+x+x^2) + x^2+ x^2 y^2\big)dy
    \]
(note that this foliation is invariant under the involutions $(x,y)\mapsto (\frac{1}{x}, y)$ and $(x,y)\mapsto (x, \frac{1}{y})$). 
Next, blow-up the four points $(0,0), (0, \infty), (\infty, 0), (\infty, \infty)$. The resulting surface $X$ is a weak del Pezzo surface of degree 4. The induced foliation on $X$, is a regular foliation that leaves the $(-2)$-curves $C_i$ of the configuration of Figure~\ref{graph:4A1} invariant and has tangency of order 1 along each $(-1)$-curve $L_i$.  
\terminou
\end{ex}

\section{Foliations on projective spaces}\label{section:regprojspaces}

We assume $n>1$.   Let $\mathcal F$ be a codimension $q$, $0 < q < n$, foliation on $\mathbb P_{\field}^n$. The $q$-th wedge product given by the inclusion $N_{\mathcal F}^* \to \Omega^1_{\mathbb P_{\field}^n}$ gives rise to a nonzero global section $\omega\in \h^0(\mathbb P_{\field}^n, \Omega^q_{\mathbb P_{\field}^n}\otimes \det(N_{\mathcal F}))$.  The \textbf{degree} of $\mathcal F$ is defined as the integer $d$ such that $\det (N_{\mathcal F}) = \mathcal O_{\mathbb P_{\field}^n}(q+d+1)$. In this section, we examine some conditions under which a foliation on $\mathbb P_{\field}^n$ is $p$-closed.

\subsection{Regular foliations on projective spaces} 

The characterization of regular foliations on projective spaces is a consequence of \cite[Theorem 1.5]{MR2148539}:  

\begin{thm}\label{thm:regfolproj}
Let $\mathcal F$ be a foliation on $\mathbb P_{\field}^n$ of codimension $q$, $0 < q < n$, and degree $d$. If $\mathcal F$ is regular, then $q=1$, $\mathrm{char}(\field)=2$, $d=0$ and $n$ is odd. Furthermore,
$\mathcal{F}$ is given by a global exact $1$-form $\Omega=dF$, where $F$ is a homogeneous polynomial of degree $2$.  In particular,  $\mathcal F$ is $2$-closed.
\end{thm}

\proof
Following the terminology of \cite{MR2148539}, the subbundle $T_{\mathcal F}\subset T_{\mathbb P_{\field}^n}$ defines a nesting map $f\colon \mathbb P_{\field}^n\to \mathbb G(n-q,n)$, $x\mapsto T_x\mathcal F$, where $\mathbb G(n-q,n)$ denotes the Grassmannian of $n-q$ dimensional linear subspaces of $\mathbb P_{\field}^n$.  According to \cite[Theorem 1.5]{MR2148539}:  $q=1$ and $n$ is odd. Furthermore, there is a non-degenerate alternating bilinear form  on $\field^{n+1}$ such that $f$ is given by $f(x) = x^{\perp}$. Where  $x^{\perp}$ is the orthogonal complement with respect to the bilinear form. 

We may assume, up to change of coordinates, that the non-degenerate alternating bilinear form is given by a matrix consisting of  blocks $\small{\begin{pmatrix} 0 & 1 \\ -1 & 0 \end{pmatrix}}$ along the main diagonal and zero elsewhere. This means that $\F$ is defined, in homogeneous coordinates by the standard contact form
\begin{eqnarray}\label{eq:contactform}
\Omega = x_0dx_1 - x_1dx_0 + x_2dx_3-x_3dx_2+ \cdots + x_{n-1}dx_n - x_ndx_{n-1}.
\end{eqnarray}
The integrability implies that $p=2$ and we get 
\[
\Omega = d(x_0x_1+\cdots+ x_{n-1}x_n).
\] 
\endproof

\begin{remark}\rm
    During the proof of Theorem~\ref{thm:regfolproj}, we established that for $0 < q < n$,  a codimension $q$ distribution on $\mathbb P_{\field}^n$  is regular only if $q=1$ and $n$ is odd. Moreover, up to change of coordinates it can be defined by the contact form \eqref{eq:contactform}. This aligns with Proposition 2.1 in page 85  of \cite{MR537038}. 
    \terminou 
\end{remark}

The next result is an immediate consequence of Theorem~\ref{thm:regfolproj}; however  we present an alternative proof.

\begin{prop}
Let $\mathcal F$ be a codimension one foliation on $\mathbb{P}_{\field}^n$ of degree $d$. If $\mathcal F$ is regular, then $\mathrm{char}(\field)=2$, $d=0$ and $n$ is odd. Furthermore,
$\mathcal{F}$ is given by $\Omega=dF$, where $F$ is a homogeneous polynomial of degree $2$.  In particular,  $\mathcal F$ is $2$-closed.
\end{prop}

\proof
The normal bundle of $\F$ can be written as $N_\mathcal F = \mathcal O_{\mathbb{P}_{\field}^n}(d+2)$. Considering the exact sequence of vector bundles 
\[
0\to T_\mathcal F \to T_{\mathbb{P}_{\field}^n} \to N_{\mathcal F} \to 0
\]
and taking the total Chern class we get the identity
\[
c(T_{\mathbb{P}_{\field}^n}) = c(T_\mathcal F)\cdot c(N_\mathcal F) \quad \in \mathrm{CH}^*(\mathbb{P}_{\field}^n)
\]
in the Chow ring of $\mathbb{P}_{\field}^n$.  Then 
\[
(1+h)^{n+1} - h^{n+1} = P(h)\cdot (1+(d+2)h)
\]
where $h=c_1(\mathcal O(1))$ and $P(h)$ is a polynomial of degree $n-1$ representing $c(T_\mathcal F)$. Therefore, $-(d+2)^{-1}$ is a root of $(1+h)^{n+1} - h^{n+1}$ and consequently 
\begin{eqnarray*}
(d+1)^{n+1} = 
\left\{
\begin{array}{l}
1 \quad \;\;\text{if} \; n \; \text{if odd}\\
-1\quad \text{if} \; n \; \text{if even}
\end{array}
\right.
\end{eqnarray*}
 Hence, since $d\ge 0$, one obtains that $d=0$ and $n$ is odd.  
 

Since $\mathcal{F}$ is regular and $\omega\wedge d\omega=0$,  the  de Rham-Saito's division lemma ensures the existence of a polynomial 1-form $\eta$ such that $d\omega = \eta \wedge \omega$. Comparing the degrees of $\omega$ and $d\omega$  we get $d\omega= 0$. Then, Euler's relation gives 
\[
0 = i_R d\omega = (d+2)\omega
\]
which implies that  $p$ is positive and divides $d+2$. Hence,  $p=2$ and $\mathcal{F}$ is given by a  closed projective $1$-form  $\omega$  having homogeneous polynomials of degree one as coefficients.  In particular, there is a Cartier decomposition \cite[Theorem 1.3.4]{MR2107324}
 \[
    \omega = dF+\sum_{i=0}^{n}u_{i}x_{i}dx_{i}
 \]
 and the contraction with the radial vector field $R=\sum_{i=0}^n x_i\frac{\partial}{\partial x_i}$ implies that $u_{i} = 0$ for every $i$. This concludes the proof of the proposition.
 
\endproof

\subsection{Invariant curves and singular points}\label{SS:singnonsing}

Let $C$ be a reduced affine plane curve which is invariant by a foliation $\F$ on $\mathbb A^2_{\field}$. 
Let us assume that there exists $q\in \textrm{Sing}(C)$ such that $q\notin  \textrm{Sing}(\F)$. This can only occur if the characteristic of the ground field $\field$ is positive. Actually, in characteristic zero, it is always possible to rectify (formally) a non vanishing vector field at $q$ and this prevents for this pathological phenomena, using that the kernel of the derivation $\partial_x$ in $\field \llbracket  x,y\rrbracket$ is precisely $\field \llbracket y\rrbracket$.

In positive characteristic $p$, a typical example that illustrates this phenomenon is provided by $f=y^2-x^p$, which is invariant by the foliation  $\F$ defined by $dy$.  At first glance, it  is reasonable  to think that this kind of phenomenon can only occur if $\deg(C)\equiv0\mod p$, or at least $\deg(C)\geq p$. This turns out to be false:



\begin{prop}\label{P:counterexample}
Assume $\mathrm{char}(\field)=p\geq 5$. Let $s=\lceil \frac{ p}{3} \rceil$ and $r=p-s$. Consider the affine plane curve $C=\mathcal Z(f)\subset { \mathbb A}^2_{\field}$, of degree $2s<p$, defined  by 
	\[
	f(x,y) = y^2 + 2(x^s + x^r)y + x^{2s}.
	\]
Then the polynomial  1-form

\begin{eqnarray}\label{eq:poly1form}
\omega = 2Ud(y+x^s+x^r) - (y+x^s+x^r)dU
\end{eqnarray}
where $U= 2+ x^{2r-p}$,  leaves $C$ invariant and  $\omega(0)\not=0$.

\end{prop}  

\begin{proof}
	The  curve given by $g(x,y)= y^2 -x^p$ is invariant by the foliation defined  by  $dy$. Now, since $f$ can be written as  
	$$f(x,y)={(y+x^r + x^s)}^2 -2x^p-x^{2r}$$
we see that
	$$
	f(x,y)= U\cdot g(\varphi(x,y))
	$$ 
	where $U = 2+ x^{2r-p}$ and $\varphi (x,y)= (x, U^{-\frac{1}{2}}\cdot (y+x^s+x^r))$ is a local transformation. Note that the condition on $s$ ensures that $2r>p$.  The curve $C$ is invariant by the  polynomial 1-form 
	\[
	\omega= U^{\frac{3}{2}}\varphi^* (dy)
	\]
	which is nonsingular at the origin. A straightforward computation reveals that $\omega$ coincides, up to a constant, with the one defined by \eqref{eq:poly1form}.
\end{proof}

\begin{remark}
	It is easy to see that this phenomenon does not occur in characteristic $p=2,3$. 
\end{remark}

Building on the same idea, one can refine the argument of Proposition~\ref{P:counterexample}. Set  $p\geq 5$.  
By descending induction on the degree, it is easy to construct  polynomials $P,Q\in \field [x]$, $P(0)=0$, $\textrm{deg}(P)=\lfloor \frac{p}{2}\rfloor -1$, $\textrm{deg}(Q)\leq \lceil \frac{p}{2}\rceil $  such that 
$$
{( P(x) +x^{\lfloor \frac{p}{2}\rfloor} + x^{\lceil \frac{p}{2}\rceil})}^2= Q(x) +2x^p +x^{p+1}.
$$
For instance, for $p=5$, one can take $P(x)=-\frac{x}{2}$, so that $Q(x)=\dfrac{x^2}{4}-x^3$. Now, set $$f(x,y)=y^2 +2( P(x) +x^{\lfloor \frac{p}{2}\rfloor} + x^{\lceil \frac{p}{2}\rceil})y +Q(x).$$ It is a degree $\lceil \frac{p}{2}\rceil +1$ polynomial which can be also rewritten as 
\begin{eqnarray}\label{eq:newf}
f(x,y)={ ( y+  P(x) +x^{\lfloor \frac{p}{2}\rfloor} + x^{\lceil \frac{p}{2}\rceil})}^2-x^p(2+x).
\end{eqnarray}
In particular, it is contact equivalent at the origin to $y^2- x^p$. We do the details below. 

\begin{prop}\label{P:counterexamplebis}
Assume $\mathrm{char}( \field )=p\geq 5$. Then there exists a polynomial $1$-form $\omega$ in ${ \mathbb A}^2_{\field}$, with  $\omega (0)\not=0$,   that leaves invariant a reduced affine curve $C=\mathcal Z(f)\subset { \mathbb A}^2_{\field}$, of degree $\lceil \frac{p}{2}\rceil +1$,  and that is singular at the origin. Furthermore, $C$ is defined by a polynomial as \eqref{eq:newf}.  
\end{prop}

\proof 
Let $s=\lfloor \frac{p}{2}\rfloor$ and $r=\lceil \frac{p}{2}\rceil$.  As explained above, it remains to construct  polynomials $P,Q\in \field [x]$, $P(0)=0$, $\textrm{deg}(P)=s-1$, $\textrm{deg}(Q)\leq r $  such that 
$$
{( P(x) +x^s + x^r)}^2= Q(x) +2x^p +x^{p+1}.
$$
Note that $s=r-1$, $r+s=p$ and $2s=p-1$. The identity above is equivalent to 
\begin{eqnarray}\label{eq:PQx}
P(x)^2+2P(x)x^s+2P(x)x^{s+1} + x^{2s}= Q(x). 
\end{eqnarray}
 
 Let us write 
 \[
 P(x) = a_{s-1}x^{s-1}+ a_{s-2}x^{s-2}+\cdots + a_1x
 \]
 with coefficients $a_i\in \field$, $i=1,\dots, s-1$. The coefficients must be chosen in such way that $\deg (Q) \le r$. If $p=5$, then $P(x)=-\frac{1}{2}x$   works. If $p\ge 7$  (then $s\ge 3$) grouping the coefficients of terms of the same degree on the left-hand side of \eqref{eq:PQx},  we see that must have
 \begin{eqnarray*}
\left\{
\begin{array}{l}
 a_{s-1} = -\frac{1}{2} \\
  a_{s-2} = \frac{1}{2}\\
2a_{s-2} + 2a_{s-3} + a_{s-1}^2 = 0\\
\vdots\\
2a_{s-t} + 2a_{s-(t-1)} + \displaystyle{\sum_{i+j=2s-t}} a_{i}a_{j} = 0\\
\vdots\\
\end{array}
\right.
\end{eqnarray*}
Hence, for $p\ge 7$, the system above has a unique solution $a_1,\dots, a_{s-1}$ and therefore there exists a unique  polynomial $P(x)$ satisfying the constraints. 

Similar to the proof of Proposition~\ref{P:counterexample}, by taking  $U=2+x$ and 
\[
\varphi(x,y)=(x, U^{-\frac{1}{2}}\cdot (y+P(x)+x^s+x^{s+1}))
\]
we see that $f = U\cdot g(\varphi(x,y))$ and $\omega =  U^{\frac{3}{2}}\varphi^* (dy)$ leaves $f=0$ invariant.   This concludes the proof of the proposition. 
\endproof

We end this section by proving that the singular locus of an invariant reduced curve $C$ is necessarily contained in the singular locus of $\F$ provided that $d_C=\deg C$ is ``small" with respect to $p$. That is, $d_C$ is asymptotically ``controlled'' by $\sqrt{p}$, see Theorem~\ref{thm:invsmalldegree} below. It is very likely that this bound is not sharp. However, observe that it cannot exceed $\lceil \frac{p}{2} \rceil +1$ according to Proposition \ref{P:counterexamplebis}.

We need to introduce some notation. Let $\field \llbracket x,y\rrbracket $ be the ring of formal series in two indeterminates. For $f,g\in \field\llbracket x,y\rrbracket$, we denote by

$$i(f,g)=\textrm{dim}_{\field} \field \llbracket x,y\rrbracket/ (f,g) \in \mathbb N \cup{ \{ \infty\}}$$
the intersection multiplicity between $f$ and $g$. The {\bf Kappa invariant} of $f\in \field \llbracket x,y\rrbracket$ is defined as 

$$\kappa(f)= i(f, \alpha f_x +\beta f_y)$$for general $(\alpha:\beta) \in { \mathbb P}^1_{\field}$. Here and hereafter, $f_x$ and $f_y$ stand for the partial derivative with respect to $x$ and $y$. It can be readily seen that, if $f$ is irreducible, then  

$$\kappa (f) = \textrm{min}( i(f,  f_x),i(f,f_y))$$
The following properties are well known:
\begin{prop}
	The intersection multiplicity and the kappa invariant are invariant under contact equivalence, that is, for any formal change of variables $\varphi(x,y)$ and every units $u,v\in \field \llbracket x,y\rrbracket $, then
	\begin{enumerate}
		\item $i(f,g)=i(uf\circ\varphi, vg\circ\varphi)$,
		\item $\kappa (f)=\kappa (uf\circ\varphi)$.
	\end{enumerate}
\end{prop}
\begin{cor}\label{C:boundkappa}
	Let $C \subset { \mathbb P}^2_{\field}$ be a reduced algebraic curve of degree $d$. Let $q$ be a singular point of $C$ and let $f$ be a local  equation for $C$ at $q$. Then 
	\[
	\kappa_q(C):=\kappa(f)\leq d(d-1).
	\]
\end{cor}

We will need the following result.  


\begin{prop}\label{P:totalkappa}
Let $f\in \field \llbracket x,y\rrbracket -\{ 0\}$, $f(0)=0$ and let $f=f_1 f_2....f_r$ be its decomposition into irreducible factors. Then
\[
\kappa(f)=\sum_{k=1}^r\big(\kappa (f_k) +\sum_{j\not=k} i(f_k, f_j)\big).
\]
\end{prop}

Assuming that $f\in \field \llbracket x,y\rrbracket$ is irreducible, it always admits a parametrization generalising  the Puiseux expansion in characteristic $0$, called Hamburger-Noether  expansion (see \cite{Camp80}), which is a pair $(x(t), y(t))$, with $x(t), y(t)$  in $t\field \llbracket t\rrbracket$, satisfying 
\[
f(x(t), y(t))=0
\]
 and possessing the following  properties:
%
%
%

\begin{enumerate}
\item  For every pair $(u(t),v(t))$,  with $u(t), v(t)\in t\field \llbracket t\rrbracket $ and $f(u(t), v(t))=0$, there exists a unique $\varphi (t)\in t\field \llbracket t\rrbracket $ such that 
\[
u(t)=x(\varphi (t)) \quad \text{and} \quad v(t)=y(\varphi (t)).
\]
\item For every $g\in \field \llbracket x,y\rrbracket $, we have
\begin{eqnarray}\label{E:paraideal}
g(x(t),y(t))=0 \Longleftrightarrow g\in (f)
\end{eqnarray}
and 
\begin{eqnarray}\label{E:paraintersect}
i(f,g)=\textrm{ord} ( g(x(t),y(t) ).
\end{eqnarray}
\end{enumerate}


Now, we can prove the main result of this section. 

\begin{thm}\label{thm:invsmalldegree}
Let $C\subset \mathbb A^2_{\field}$ be a reduced curve which is invariant by a foliation $\F$ on $\mathbb A^2_{\field}$ and  assume that 
\[
\deg(C)(\deg(C) -1)<p. 
\]
Then $\sing(C)\subset \sing(\F)$.
\end{thm}
\begin{proof}
	We proceed by contradiction by assuming that there exists $q\in  \sing(C)$ and $q\notin\sing(\F)$. According to \cite{Mc17}, in suitable formal coordinates centered at $q$ the foliation is defined by a vector field having the form 
	$$v = \partial_x + x^{p-1}g(x^p,y)\partial_y,\ g\in \field \llbracket x,y\rrbracket ,$$ or equivalently by the form 
	
	$$\omega=x^{p-1}g(x^p,y) dx -dy.$$
This implies  that $h:=g(x^p,y)$ is an equation of the $p$-divisor of $\F$ in the formal neighborhood of $p$ (we do not exclude that $g$ is identically zero, corresponding to $\F$ being $p$-closed). 

Let $f(x,y)=0$ be a local equation of $C$ and $f=f_1 \cdots f_r$ its (local) decomposition into irreducible factors. Let $\varphi_l (t)=(x_l (t), y_l(t))$ be a parametrization of $f_l$. From the vanishing of $\varphi_l^* \omega$ and from $\ord_q (C)\leq \deg(C) < p$, we can infer that 
\[
x_l(t)=t^{n_l} U(t) \quad \text{and} \quad y_l(t)= \alpha_l (t^p), 
\]
where $1\leq n_l<p$ and $\alpha_l\in t\field \llbracket t\rrbracket $ and $U$ is a unit in $\field \llbracket t\rrbracket $. Indeed, since $f_l$ is invariant and $h$ is an equation for the $p$-divisor, then $f_l$ divides $h$ and we get 
\begin{eqnarray*}
y_l'(t)=x_l^{p-1}h(\varphi_l(t))x_l'(t)=0.
\end{eqnarray*}
By performing the change of variables $t\to tU^{\frac{1}{n_l}}(t)$ (which makes sense as $n_l<p$), one can actually suppose that
	\begin{equation}\label{E:param}
	 x_l(t)=t^{n_l}\ \textrm{ and}\ y(t)= \alpha_l(t^p).
	 \end{equation} 
	 
	 This allows us to describe $f_l$ (up to multiplication by a formal unit $u_l$) by using fractional power series like in the classical Puiseux parametrization:  $f_l=u_l g_l$ with
	\[
	g_l= \prod_{i=1}^{n_l}(y-\alpha_l (\xi_{n_l,i}^p x^\frac{p}{n_l}))
	\]
where the $\xi_{n_l,i}$'s are the  (pairwise distinct) $n_l\textrm{th}$ roots of unity in $\field$. Actually, one can justify this equality by observing that $g_l$ is well defined as a power series (after expansion), $g_l(\varphi_l(t))=0$ and $\textrm{ord}(g_l)=\textrm{ord}(f_l)=n_l$, and the  conclusion follows by \eqref{E:paraideal}. Note that the power series expansion of $g_l$ takes the form 

\begin{equation}\label{E:expansion}
g_l(x,y)=y^{n_l}+ \sum_{i=1}^{n_l}a_{l,i}(x^p)y^{n_l-i},\ a_{l,i}(x)\in x\field \llbracket x\rrbracket 
\end{equation}

Now, since $q$ is a singular point of $C$,  two possible (not mutually exclusive) cases arise:
\begin{enumerate}
	\item $C$ admits two local smooth branches; or
	\item  $C$ admits a singular branch.
\end{enumerate}
Up to renumbering, one can suppose that in the first occurrence  (resp. the second one) this corresponds to $g_1$ and $g_2$ (resp. to $g_1$). In the first case, $n_1=n_2=1$ and 
\[
i(f_1,f_2)=i(g_1,g_2)=0 \mod p
\] 
(and of course positive) thanks to \eqref{E:paraintersect}, \eqref{E:param} and \eqref{E:expansion}. In the second case, $n_1>1$, one can conclude similarly that 
\[
\kappa (f_1)=\kappa(g_1)=i(g_1, { (g_1)}_y)=0 \mod p.
\]
But, the combination of Corollary \ref{C:boundkappa},  Proposition \ref{P:totalkappa} and our hypotheses give 
\[i(f_1,f_2)\le \kappa (f) \le \deg(C)(\deg(C)-1) < p\] in the first case and 
\[\kappa (f_1)\le \kappa (f)\le\deg(C)(\deg(C)-1) < p\] in the second one, yielding to a contradiction. 
  
\end{proof}

\begin{remark}\rm 
    It would be interesting to determine the sharp bound ensuring that any singular point of a reduced invariant curve is necessarily a singular point of the foliation. Let $C_d\subset \mathbb A^2_{\field}$ be a reduced singular curve of degree $d$ invariant under a foliation $\F$. Below, we write $(\F, C_d)$ to denote such a pair. In view of Lemma~\ref{L:normalcrossing}, we may assume $d\ge 3$. Let us fix a prime number $p\ge 5$ and set
\begin{equation}
    \frak l(p):=\max\{ m\in \mathbb N \quad : \quad \text{if } d < m, \text{ then } \sing(C_d)\subset \sing (\F) \text{ for any pair } (\F, C_d)\}.  
\end{equation}
Proposition~\ref{P:counterexamplebis} and Theorem~\ref{thm:invsmalldegree} yield 
\[
\sqrt{p+3} \le \frak l(p) \le \left\lceil \frac{p}{2} \right\rceil +1.
\]
\terminou
\end{remark}

\begin{remark}\rm
 The bound $\deg (C)(\deg(C)-1)<p$ already appears   in the paper \cite{Nguy16}. One also refers to the results presented in  the  introductory section, where it is shown to provide a sufficient condition for the  nonexistence of the so-called wild vanishing cycles.
 \terminou
\end{remark}

\subsection{Codimension one foliations and invariant hypersurfaces I}

In this section, we address  the case where a codimension one foliation $\F$ on $\mathbb{P}_{\field}^n$ admits a smooth invariant hypersurface that does not intersect the singular set of $\F$. This is a typical phenomenon that occurs exclusively in positive characteristic.

When the degree of a homogeneous polynomial $F$ is a multiple of the characteristic $p$, Euler's identity implies that $i_R dF=0$,  where $R$ is the radial vector field. Consequently,  the differential $dF$ induces a foliation on  $\mathbb{P}_{\field}^n$. The following example illustrates this phenomenon. 

\begin{ex}\rm
 For each prime number $p > 2$ consider the polynomial in $\overline{\mathbb{F}}_p[x_0,\ldots,x_n]$ given by
$$
    F = x_0x_1^{p-1}+x_1x_2^{p-1}+\cdots+x_{n-1}x_{n}^{p-1}+x_{0}^{p}
$$    
and let $ Y = \mathcal{Z}(F) \subset \mathbb{P}_{\overline{\mathbb{F}}_p}^{n}$ be the corresponding hypersurface. 
The $1$-form $\Omega = dF$ defines a codimension one foliation in $\mathbb{P}_{\overline{\mathbb{F}}_p}^{n}$ having a single singular point $q=(1:0:\cdots: 0)$. We note that $\Omega$ leaves the smooth hypersurface $Y$  invariant, and the  singular point $q$ of  $\mathcal F$  does not lie in  $Y$. In Corollary~\ref{cor:HnotSing} below, we will see that this example represents, to some extent,  the general case.  
\terminou
\end{ex}

In the next result,   Proposition \ref{decomposition} is applied to obtain a global decomposition of the homogeneous $1$-form that defines the  foliation on affine coordinates of $\mathbb{P}^n_{\field}$. 

\begin{prop}\label{decP} 
Let $\mathcal{F}$ be a codimension one foliation on $\mathbb{P}_{\field}^{n}$ given by a projective $1$-form $\Omega = \sum_{i=0}^{n}A_{i}dx_{i}$. Let $Y = \mathcal{Z}(F) \subset \mathbb{P}_{\field}^{n}$ be a smooth hypersurface invariant by $\mathcal{F}$. Then there exist
a homogeneous polynomial $H \in \field[x_0,\ldots,x_n]$ and a homogeneous $1$-form $\eta$ in $\mathbb{A}_{\field}^{n+1}$   such that 
\[
    \Omega = HdF+F\eta.
\]
\end{prop}

\begin{proof} The argument  is similar to  \cite[Proposition 1]{MR1150571}. Consider the open set 
\[
U = \mathbb{A}_{\field}^{n+1}\setminus \{(0,\ldots,0)\}. 
\]
We will use the following facts (see \cite[Theorem 1.14]{MR1917232}):

%
%

\begin{enumerate}
    \item $\h^{1}(U, \mathcal{O}_U) =0$; and 
    
    \item  the restriction map $\mathcal{O} (\mathbb{A}_{\field}^{n+1}) \to \mathcal{O}(U)$ is surjective. 
    
\end{enumerate}

Denote by $f$ the local equation of $Y$ in $U$ and $\omega$ the $1$-form defining $\mathcal{F}$ in $U$ via restriction, that is, $\omega = \Omega|_{U}$. Since, $Y$ is smooth, by Proposition \ref{decomposition}, we know that there exists an open covering  $\{U_i\}$ of $U$, elements $h_i \in \mathcal{O}_{X}(U_i)$ and $\sigma_i \in \Omega_X^{1}(U_{i})$ such that
\[
\omega|_{U_i} = h_idf+f\sigma_i .
\]
If $U_{ij} = U_{i}\cap U_{j}$ is nonempty, we have 
$h_{i}df+f\sigma_{i} = h_{j}df+f\sigma_{j}$, which gives $(h_{i}-h_{j})df = f(\sigma_{j}-\sigma_{i})$, and then     $h_{i}-h_{j} = fg_{ij}$  
for certain $g_{ij} \in \mathcal{O}(U_{ij})$. Note that $g_{ik}= g_{ij}+g_{jk}$ for every $i,j,k$ and since 
$\h^{1}(U, \mathcal{O}_U) =0$
we have $g_{ij}=g_{j}-g_{i}$ for some $g_{i} \in \mathcal{O}(U_i)$ and  $g_{j} \in \mathcal{O}(U_j)$. Therefore  
\[
h_{i}+fg_{i} = h_{j}+fg_{j} \quad \text{in} \quad U_{ij}
\] 
and then we can take  $H\in \mathcal{O}_{\mathbb{A}^{n+1}}(U)$ by
\[
H|_{U_i} := h_i+fg_i.
\]
Since the restriction map  $\mathcal{O}(\mathbb{A}^{n+1}) \to \mathcal{O}(U)$ is surjective, one can assume that 
\[
H \in \field[x_0,\ldots,x_n]. 
\]

In order to define the $1$-form $\eta$ of the statement, we note that
\[
 \sigma_{i}-g_{i}df = \sigma_{j}-g_{j}df \quad \text{in} \quad U_{ij}. 
 \] 
Using the same extension principle above, there exists 
$\eta \in \Omega_{\mathbb A_{\field}^{n+1}}^{1}(\mathbb A_{\field}^{n+1})$ such that
\[
\eta|_{U_i} = \sigma_i - g_idf.
\] 
Hence, 
we have the identity
    $$
        \Omega = HdF+F\eta.
    $$
Finally,  using that $F$ and $\Omega$ are homogeneous, we can assume  that $H$ and $\eta$ are homogeneous. This finishes the proof.

\end{proof}

As a first and easy application of Proposition~\ref{decP}, we provide a bound for the degree of a smooth invariant hypersurface in terms of the degree of the foliation. We refer to \cite{EstevesKleiman} for a more complete treatment of the problem of bounding the degree of invariant subvarieties.

\begin{cor} Let $\mathcal{F}$ be a codimension one foliation on $\mathbb{P}_{\field}^{n}$, $n\ge 2$, and $Y = \mathcal{Z}(F)$ be a smooth hypersurface invariant by $\mathcal{F}$. Then 
\[
\deg(Y)\leq \deg(\mathcal{F})+2.
\]
\end{cor}

\begin{proof} 
It  follows from Proposition \ref{decP} that there is a global decomposition: $\Omega = HdF+F\eta$, where $\Omega$ is a $1$-form defining $\mathcal{F}$ in homogeneous coordinates and $H$ is a homogeneous polynomial. In particular, we have 
\[
\deg(\mathcal{F})+1 = \deg(H)+\deg(Y)-1
\] 
and the result follows. 

\end{proof}

Now, we deal with a smooth invariant hypersurface that does not intersect the singular locus of $\F$. 

\begin{cor}\label{cor:HnotSing} Let $\mathcal{F}$ be a codimension one foliation on $\mathbb{P}_{\field}^{n}$, $n\ge 2$, and let  $Y=\mathcal Z (F)$ be a smooth hypersurface invariant by $\mathcal{F}$. If 
$Y$ does not intersect $\sing (\mathcal F)$, then $\field$ has  characteristic $p>0$,  $\deg F$ is a multiple of $p$ and $\mathcal{F}$ is given by the kernel of the global exact $1$-form $\Omega=dF$. In particular, $\F$ is $p$-closed.
\end{cor}

\begin{proof} 
Let us consider the decomposition given by  Proposition \ref{decP}:
$$
    \Omega = HdF+F\sigma.
$$

We will show that $\sigma = 0$.  Aiming to a contradiction, let us assume that $\sigma \neq 0$. In this case,  $H$ has degree $\ge 1$, because $F\sigma$ has coefficients of degree at least $\deg F$ and $dF$ has degree $\deg F-1$.
Hence, the algebraic set $\mathcal{Z}(H, F)$  is  nonempty:
\[
\dim \mathcal{Z}(H,F)\geq n-2 \ge 0. 
\] 
See \cite[Chap. I Theorem 7.2]{AGHarthorne}.
Since $\mathcal{Z}(H, F)$ lies in the singular set of $\mathcal F$, we get a contradiction, because this implies that  $\sing(\mathcal{F})\cap Y$ is nonempty. 

Therefore,  $\sigma = 0$,   $H \in \field^{*}$ (since $\codim \sing(\mathcal{F})\geq2$), and $\mathcal{F}$ is defined by $dF$. In addition,  the characteristic of $\field$ is $p>0$, with $p|\deg F $ and $\deg F = \deg(\mathcal{F})+2$. 
\end{proof}

\subsection{Codimension one foliations and invariant hypersurfaces II}

%

In this section, we  extend Corollary~\ref{cor:HnotSing} to distributions on $\mathbb{P}^n_{\field}$, while also relaxing the condition that the hypersurface does not intersect the singular locus of the distribution.

We  need the following result. 

\begin{lemma}\label{lemma:d=0modp}
Let $\F$ be a foliation on $\mathbb P^2_{\field}$ and let $C=\mathcal Z(F)$ be an invariant curve, possibly singular, that does not intersect $\sing (\F)$. Then the characteristic $p$ of $\field$ is positive and $\deg C = 0 \mod p$.   
\end{lemma}

\proof 
This is an immediate consequence of Theorem~\ref{thm:CSformula}. 

%

\endproof

Now we can prove the main result of this section. 

\begin{thm}\label{thm:HcapSing=empty}
Let $\D$ be a codimension one distribution on $\mathbb{P}^n_{\field}$, $n\ge 2$. Let $Y=\mathcal Z(F)$ be a reduced  hypersurface invariant by $\D$. If $Y\cap \textrm{Sing}(\D)$ has codimension at least $3$,  then $\field$ has characteristic $p>0$, $\deg F$  is a multiple of $p$ and $\D$ is given by the kernel of $dF$. In particular, $\D$ is integrable and $p$-closed.
\end{thm}

\begin{proof}
Let us first assume that $n=2$. In this situation, $\mathcal D=\F$ is automatically integrable and $Y=C$ is a (maybe singular) curve which does not contain any singular point of $\F$.

By Lemma~\ref{lemma:d=0modp} we may assume $d_C:=\deg C = 0 \mod p$, and thus $p>0$. By assumption, we have that 
\begin{equation}\label{E:omegawedgedf}
\Omega\wedge dF=F\eta
\end{equation}
where $\Omega$ is a section of $\Omega_{ {\mathbb P}^2}^1 (l)$ defining $\F$. We see $F$  and $dF$ as sections of ${ \mathcal O}_{ {\mathbb P}^2} (d_C)$ and    $\Omega_{ {\mathbb P}^2}^1 (d_C)$, respectively. Then $\eta\in \h^0({\mathbb  P}^2,\Omega_{ {\mathbb P}^2}^2 (l) )$. 

For every $q\in {\mathbb P}^2$, there exists a neighborhood of $q$ and a one form $\xi_q$ such that $\eta=\Omega\wedge\xi_q$. Indeed, when $q\in C$, this is a consequence of $\Omega (q)\not =0$ and otherwise, one can take $\xi_q=\frac{dF}{F}$. In particular, one can  find an open cover $\mathcal U=\{U_i\}$ of ${\mathbb P}^2$ together with a collection of one forms $\xi_i\in \h^0(U_i,\Omega_{ {U_i}}^1 )$ such that
\begin{equation}\label{E:divisionprop} 
\eta=\Omega\wedge \xi_i .
\end{equation}
It follows that the $\xi_i-\xi_j$ define a cocycle in $Z^1(\mathcal U, N_\F^*)$ where $N_\F^*\simeq { \mathcal O}_{ {\mathbb P}^2} (-l)$ is the conormal sheaf of the foliation. By the triviality of degree one cohomology and up to refining $\mathcal U$, this implies that there exists a collection of one forms $\alpha_i\in 	\h^0(U_i,N_\F^* )$ such that the $\xi_i-\alpha_i$ glue together to a global one form, which is then necessarily trivial. Consequently, $\eta=0$, as wanted.

%

	
Now, we take $n\ge 2$ arbitrary. Let $\Omega\in \h^0(\mathbb{P}^n_{\field}, \Omega_{ \mathbb{P}^n_{\field}}^1 (l))$ defining $\D$.  We consider the restriction $\D|_{\proj2}$ of $\D$ to a general plane section $\proj2\subset\mathbb{P}^n_{\field}$ that avoids $Y\cap \sing \mathcal D$. 

We will show that $\deg C = 0 \mod p$, and therefore $\deg Y = 0 \mod p$.  We cannot guarantee that $C$ does not intersect the singular locus of  $\D|_{\proj2}$. Indeed, there may have tangencies between $\proj2$ and  $\D$ lying in $C$. However, one can show that the Camacho-Sad indices are vanishing along $C$.  For this, given $q\in C$, there exists a neighborhood of $q$ in $\mathbb{P}^n_{\field}$ where the  decomposition $g\Omega = df + f \sigma$ holds, because $q\notin \sing (\D)$. Here, $f$ is a local equation for $Y$,  see Proposition~\ref{decomposition}. Therefore, there is a similar decomposition  for the restriction  $\tilde{\Omega} = \Omega|_{\proj2}$:
\begin{eqnarray}\label{eq:restric}
\tilde{g}\tilde{\Omega} = d\tilde{f}  + \tilde{f} \tilde{\sigma}. 
\end{eqnarray}
where the symbol $\sim$ means restriction to $\proj2$. 
 This implies that 
 \[
 CS(\D|_{\proj2}, C, q)  = 0.  
 \]
 It follows from Theorem~\ref{thm:CSformula} that $\deg C = 0 \mod p$. 

Since $Y$ is invariant by $\D$, we can write
\[
\Omega\wedge dF=F\eta. 
\] 
Let us show that the restriction $\tilde{\eta}$ of $\eta$ to $\proj2$ is identically vanishing, and so is $\eta$, because the plane is general. First note that even $C$ may intersect the singular locus of $\D|_{\proj2}$, there is still a covering $\{U_i\}$ of $\proj2$ and a family of 1-forms $\xi_i$ satisfying 
\begin{eqnarray}\label{eq:familyxi}
\tilde{\eta} = \tilde{\Omega} \wedge \xi_i.
\end{eqnarray}
Indeed, from \eqref{eq:restric}, near a point $q$ of $C$, we obtain 
\[
\tilde{\Omega}\wedge d\tilde{f} = \tilde{f}\tilde{\sigma}\wedge \tilde{\Omega}
\]
and therefore $\tilde{\eta} = \tilde{\sigma}\wedge \tilde{\Omega}$. If $q\notin C$ we can take $\displaystyle\xi = \frac{d\tilde{f}}{\tilde{f}}$.  Finally, using the family of $1$-forms $\xi_i$ that satisfy \eqref{eq:familyxi} and following the same procedure  as in the case $n=2$, we conclude that $\tilde{\eta} = 0$. So $\eta = 0$ and then $\Omega\wedge dF = 0$. 
%
This finishes the proof of the theorem.

\end{proof}

We observe that Theorem~\ref{thm:HcapSing=empty} is consistent with \cite[Corollary 4.5]{EstevesKleiman}. Indeed, under the hypothesis of Theorem~\ref{thm:HcapSing=empty}, we could apply \cite[Corollary 4.5]{EstevesKleiman} to conclude that $p$ divides $\deg(Y)$.

The next result is an immediate consequence of Theorem~\ref{thm:HcapSing=empty}.

\begin{cor}\label{cor:nonintegrabledis}
	Let $\D$ be a codimension one  non integrable distribution on  $\mathbb{P}^n_{\field}$ such that $ \sing (\D) $ has codimension at least $3$, then $\D$ has no invariant hypersurface. 
\end{cor}

\begin{OBS}\rm
Theorem~\ref{thm:HcapSing=empty} and Corollary~\ref{cor:nonintegrabledis}  are basically illustrated by the standard contact distribution $\D$,  given on $\mathbb P^3_{\field}$  in homogeneous coordinates by the form  
\[
\Omega=x_0 dx_1 -x_1dx_0 + x_2 dx_3 -x_3 dx_2. 
\]
The singular locus of $\D$ is empty. Furthermore, $\D$ is never integrable and does not admit any invariant hypersurface, except when $p=2$, in which case $\Omega= d(x_0 x_1 +x_2 x_3)$.
\terminou
\end{OBS}

In strong contrast with the non integrable case,  the degree of the normal bundle $N_\F$ of a codimension one  foliation on $\mathbb{P}^n_{\field}$,  whose singular locus has codimension at least $3$, is necessarily a multiple of $p=\mathrm{char} (\field)$,  in particular $p>0$. This is the  content of the next result. 

\begin{cor}\label{cor:singcod3}
Let $\F$ be a codimension one foliation on $\mathbb{P}^n_{\field}$ of degree $d$ and assume that $\sing (\F)$ has codimension at least $3$.  Then $p=\mathrm{char} (\field)$ is positive,  $d+2$ is a multiple of $p$ and $\F$ is  given by the kernel of a closed projective $1$-form $\Omega$ on $\mathbb A^{n+1}_{\field}$. Moreover, the following conditions are equivalent: 
\begin{enumerate}
\item   $\F$ admits a reduced invariant hypersurface; 
\item   $\Omega$ is exact, i.e there exists a homogeneous polynomial $G$ of degree $d+2$ such that $\Omega=dG$. 
\end{enumerate}
\end{cor}

\proof
Let $\Omega = \sum_{i=0}^n A_i dx_i$ be a projective $1$-form defining $\F$, with $\sing (\Omega)$ of codimension $\ge 2$ and whose coefficients are homogeneous polynomials of degree $d+1$. If $\sing (\Omega)$ has codimension $\ge 3$, since $\Omega$ satisfies $\Omega\wedge d\Omega = 0$, we can apply de Rham-Saito's lemma (see \cite{MR413155}) to guarantee the existence of a polynomial 1-form $\eta$ such that $d\Omega = \eta\wedge \Omega$. Comparing the degrees in both sides of the last identity, it follows that  $\eta$ as well as $ d\Omega$ must be zero. Hence, $\Omega$ is closed. However, Euler's relation gives 
\[
0 = i_R d\Omega = (d+2)\Omega
\]
then $p$ is positive and divides $d+2$.  

The last assertion of the statement will follow from Theorem~\ref{thm:HcapSing=empty}.  According to \cite [Theorem I.3.4]{ MR2107324}, one can write 

$$\Omega=dG +\sum_{i=0}^n a_i^p x_i^{p-1} dx_i$$
and $\Omega$ is exact iff the terms in the summand are trivial. It is clear that (2) implies (1). Suppose from now on that $\F$ admits an invariant hypersurface.  From Theorem~\ref{thm:HcapSing=empty}, there exists a homogeneous polynomial $F$ such that $deg(F)=0\ \mathrm{mod}\ p$ and such that $\F$ is given by the kernel of $dF$. Let us choose $F$ of minimal degree with these properties. Then, there exists a polynomial $H$ such that $dF=H\Omega$. Moreover $dH\wedge \Omega=0$ by closedness. On the other hand, $deg(H)<deg(F)$, thus implying that $dH=0$, i.e, $H=h^p$ is a $p-$power. From the exactness of $dF=H\Omega= d(HG)+\sum_{i=0}^n { ( h a_i)}^p x_i^{p-1} dx_i$, one concludes that $\Omega=dG$.
\endproof

Note that  if a codimension one foliation $\F$ on $\mathbb{P}^n_{\field}$ has degree $d$, and  $p=\mathrm{char} (\field)$ does not divide $d+2$, then $\sing (\F)$ has an irreducible component of codimension two. This follows from Corollary~\ref{cor:singcod3}. 





\begin{ex}\label{ex:noinvhy}\rm 
For all prime $p>0$ and $n\ge 3$ there exist codimension one foliations on $\mathbb{P}_{\field}^{n}$ with finite singular set and no invariant hypersurface.   
Consider the projective $1$-form 
$$
\omega=d\big( \sum_{i=0}^{n-1} x_i x_{i+1}^{2p-1}\big) +x_n^p x_{n-1}^{p-1}dx_{n-1}-x_ { n-1}^p x_{n}^{p-1}dx_{n}.
$$
Note that $\omega$ defines a codimension one foliation on $\mathbb{P}_{\field}^{n}$ of degree $2p-2$ with finite singular set: $\sing(\mathcal{F}) = \{(1\!:\!0\!:\cdots :\!0)\}$. Also note that $\mathcal{F}$ has no algebraic invariant hypersurfaces. Indeed, on the contrary, Corollary \ref{cor:singcod3} would imply that $\omega$ is exact.  In particular, $x_n^px_2^{p-1}dx_{n-1}-x_{n-1}^{p}x_{n}^{p-1}dx_n$ should be also exact, but this contradicts \cite [Theorem I.3.4]{MR2107324}.
\terminou
\end{ex}

\section{Bott vanishing in Hodge cohomology}\label{sec:BottCohom}

We assume $X$ is a smooth   variety over an algebraically closed field $\field$. 

\subsection{The Bott-connection}\label{sec:BottCon}

Let $\mathcal{F}$ be a regular foliation on $X$ and  consider the  map
$$
\nabla^{\mathrm B}\colon T_{\mathcal{F}} \longrightarrow \End(N_{\mathcal{F}})
$$
which sends a derivation $v$ of $T_{\mathcal F}$ to $\nabla^{\mathrm B}_v\in \End(N_{\mathcal{F}})$, given by 
\[
\nabla^{\mathrm B}_v(w):= [v,w]
\] 
for each $w\in N_{\mathcal F}=T_X/T_{\mathcal F}$. It follows from the fact that $T_{\mathcal F}$ is involutive that $\nabla^{\mathrm B}_v$ is well defined, i.e., $[v,w]$ is independent of the choice of a lift in $T_X$ representing the class  $w$ in $N_{\mathcal F}$.  We note that $\nabla^{\mathrm B}$ satisfies 
\begin{itemize}
\item  $\nabla^{\mathrm B}_{a\cdot v} = a\cdot \nabla^{\mathrm B}_v$,  for all $a\in \mathcal O_X$ and $v\in T_{\mathcal F}$;
\item $\nabla^{\mathrm B}_v(a\cdot w) = a\cdot \nabla^{\mathrm B}_v(w) + v(a)\cdot w$,  for all $a\in \mathcal O_X$, $v\in T_{\mathcal F}$ and $w\in N_{\mathcal F}$. 
\end{itemize}
From this, one obtains a \textit{$\mathcal F$-connection on $N_{\mathcal F}$}, i.e., a $\field$-linear map
\[
\nabla^{\mathrm B}: N_{\mathcal F} \longrightarrow N_{\mathcal F}\otimes \Omega^1_{\mathcal F}
\]
satisfying  Leibniz' rule:
\[
\nabla^{\mathrm B}(aw) = a\nabla^{\mathrm B}(w) + d_{\mathcal F}(a)\cdot w, \quad \forall\; a\in \mathcal O_X, w\in N_{\mathcal F}
\]
where $d_{\mathcal F}: \mathcal O_X\to \Omega^1_{\mathcal F}$ is defined by $d_{\mathcal F}(a)(v)=v(a)$, for all $v\in T_{\mathcal F}$. 
We call $\nabla^{\mathrm B}: N_{\mathcal F} \to N_{\mathcal F}\otimes \Omega^1_{\mathcal F}$ the {\bf Bott connection}  associated to $\mathcal F$.

\subsection{The $\mathcal F$-Atiyah class}

Let $\mathcal{F}$ be a regular foliation on $X$. Let us endow the sheaf $\mathcal E := N_{\mathcal F}\oplus (N_{\mathcal F}\otimes \Omega^1_{\mathcal F})$ with a structure of 
 $\mathcal{O}_X$-module: given $a\in \mathcal{O}_X$ and $(w,\sigma) \in \mathcal E$ we define:
$$
    a\cdot (w,\sigma) = (a w, d_{\mathcal{F}}(a) \cdot w+a\sigma). 
$$
Consider the  exact sequence of $\mathcal O_X$-modules
\begin{center}
    \begin{tikzcd}[row sep=tiny]
        0  \arrow{r} & N_{\mathcal F}\otimes \Omega_{\mathcal{F}}^{1} \arrow[r,"i"] & \mathcal E \arrow[r,"\pi"] & N_{\mathcal F} \arrow{r} & 0
    \end{tikzcd}
\end{center}
where $i\colon \sigma \mapsto (0,\sigma)$ and $\pi\colon (s,\sigma) \mapsto s$. The exact sequence splits if there is an $\mathcal{O}_{X}$-morphism $\psi\colon N_{\mathcal F} \longrightarrow \mathcal E$ such that $\pi\circ\psi = id_{N_{\mathcal F}}$. On the one hand, the  sequence above splits if and only if $N_{\mathcal F}$ admits a $\mathcal{F}$-connection. The connection is given by the second factor of $\psi$.  On the other hand, the  sequence above determines an element $\at_{\mathcal{F}}(N_{\mathcal F})$ in the group 
\[
\Ext(N_{\mathcal F},\Omega_{\mathcal{F}}^{1}\otimes N_{\mathcal F}) \cong \Ext(\mathcal{O}_{X},\Omega_{\mathcal{F}}^{1}\otimes N_{\mathcal F} \otimes N_{\mathcal F}^{*}) \cong \hh^1(X, \text{End}_{\mathcal O_X} (N_{\mathcal F})\otimes \Omega_{\mathcal{F}}^{1} ).
\] 
Since a regular foliation admits the Bott-connection then the sequence splits and the $\mathcal{F}$-Atiyah class of $N_{\mathcal F}$ is vanishing, i.e.  
\begin{eqnarray}\label{eq:fatiyah}
\at_{\mathcal{F}}(N_{\mathcal F}) = 0.
\end{eqnarray}

\subsection{Bott vanishing theorem}


In this section, we study the vanishing of powers of Chern classes in $\h^*(X, \Omega_X^*)$, a phenomenon often referred to as Bott's vanishing theorem.  Before proceeding, we recall a fundamental result:  Chern classes can be expressed in terms of the Atiyah class. 

Let $E$ be a vector bundle over a smooth variety $X$,  let $c_s(E)\in \h^s(X, \Omega_X^s)$ be its $s$th Chern class and let
\[
c(E) = 1 + c_1(E) t + \cdots  + c_s(E)t^s + \cdots
\]
be its total Chern class; see \cite[p. 143]{Gro1958}. These classes satisfy, and are characterized, by  the following properties
\begin{itemize}
\item If $f: X\to Y$ is a morphism between smooth varieties over $\field$, and  $F$ is a vector bundle over $Y$,  then 
\[
c_s(f^*(F)) = f^*(c_s(F)). 
\]
\item If $0\to E'\to E\to E''\to 0$ is an exact sequence of vector bundles over $X$, then 
\[
c(E) = c(E')\cdot c(E'').
\]
\item If $E$ is a line bundle over $X$, then $c_1(E)$ is the image in $\h^1(X, \Omega_X^1)$ of the class of $E$ in $\h^1(X, \mathcal O_X^*)$,  via the homomorphism $\mathcal O_X^* \to \Omega_X^1$  which sends $f$ to $\frac{df}{f}$. Also, $c_s(E) = 0$ for all $s>1$. 
\end{itemize}
See \cite[Th\'eor\`eme 1]{Gro1958}.

Using polarizations of the invariant polynomials together with the Atiyah class, we can construct additional classes in Hodge cohomology.  For this purpose,  
let $\sigma_s$ be  the $s$th  invariant polynomial in the space of $n\times n$ matrices, defined by the relation
\[
\det (A+t\cdot I) = \sum_{s=0}^n \sigma_{n-s}(A)\cdot t^s.
\] 
We can write 
\begin{eqnarray*}
\sigma_s(A) &=& \sum_{\# J=s} \det A_{J,J}\\
		&=& {\rm trace}(\bigwedge^s A)
\end{eqnarray*}
where $J\subset \{1, \dots, n\}$ and $A_{J,J}$ denotes the $(J,J)$th minor $(a_{i,j})_{i,j\in J}$ of $A$. To polarize $\sigma_s$, for a $s$-tuple of matrices $(A^1, \dots, A^s)$, $\tau$ a permutation of $\{1, \dots, n\}$, and  $J\subset \{1, \dots, n\}$ of cardinality $s$, let $A^{\tau}_J$ be the $s\times s$ matrix whose $i$th column is the $i$th column of $A^{\tau(i)}_{J,J}$, then we set 
\begin{eqnarray}\label{eq:polarization}
\tilde{\sigma_s}(A^1, \dots, A^s) = \sum_{\tau}\sum_{\# J=s} \det (A^{\tau}_J)
\end{eqnarray}
where $\tau$ runs over the set of all the permutations $\{1, \dots, s\}\to \{1, \dots, s\}$; see \cite[p. 403]{GH}. 
The polarization satisfies 
\[
\tilde{\sigma_s}(A, \dots, A) = s!\sigma_s(A). 
\]
 Now, using the Atiyah class $\at(E) \in \text{H}^1(X, \text{End}_{\mathcal O_X}E\otimes\Omega^1_X)$, we can consider the class
\[
\tilde{\sigma_s}(\at(E), \dots, \at(E))\in \h^s(X, \Omega_X^s).
\]
It turns out that these classes coincide with the Chern classes, up to a constant:
\begin{eqnarray}\label{eq:chernpol}
\tilde{\sigma_s}(\at(E), \dots, \at(E)) = s!c_s(E) . 
\end{eqnarray}
See \cite[Theorem 6]{MR86359} and \cite{kumar2005hodge}.

For example, when $s=2$, we have 
\[
\tilde{\sigma_2}(A^1, A^2) =\sigma_2(A^1+A^2)-\sigma_2(A^1) - \sigma_2(A^2) 
\]
and if $\at(E) = \{ A_{ij} \}_{ij}\in\text{H}^1(X, \text{End}_{\mathcal O_X}(E)\otimes\Omega^1_X)$, then 
\[
2c_2(E) = \{ \tilde{\sigma_2}(A_{ij}, A_{jk}) \}_{i,j,k} \in \h^2(X,\Omega_X^2).  
\]

The following is a classical result, proved in \cite[Section~3]{BaumBott70} for the complex case. 

\begin{thm}\label{thm:bottvanishing}(Bott vanishing) 
Let $X$ be a smooth algebraic variety over an algebraically closed field $\field$ of characteristic $p$. Let $\mathcal F$ be a regular foliation on $X$, of codimension $q$. Assume that either $p=0$ or $p>q$. Then, for every sequence of nonnegative integers $\alpha_1,\dots, \alpha_r, s_1,\dots, s_r$  such that $l:= \alpha_1s_1+\cdots+\alpha_rs_r>q$ we have 
\[
c_{s_1}(N_{\mathcal{F}})^{\alpha_1}\cdots c_{s_r}(N_{\mathcal{F}})^{\alpha_r} = 0 \quad \mbox{in}
\quad \hh^{l}(X,\Omega^{l}_{X}).
\]
\end{thm}

\proof
Let us consider the exact sequence 
\[
0 \longrightarrow N_{\mathcal{F}}^*  \longrightarrow \Omega^1_X \longrightarrow \Omega^1_{\mathcal F} \longrightarrow 0.
\]
The Atiyah class $\at(N_{\mathcal F})$  lies in $\text{H}^1(X, \text{End}_{\mathcal O_X}(N_{\mathcal F})\otimes\Omega^1_X)$.  Twisting the sequence above by $\text{End}_{\mathcal O_X}(N_{\mathcal F})$ we get 
\[
0 \longrightarrow \text{End}_{\mathcal O_X}(N_{\mathcal F})\otimes N_{\mathcal{F}}^*  \longrightarrow \text{End}_{\mathcal O_X}(N_{\mathcal F})\otimes\Omega^1_X \longrightarrow \text{End}_{\mathcal O_X}(N_{\mathcal F})\otimes\Omega^1_{\mathcal F} \longrightarrow 0
\]
and by taking the exact long sequence in cohomology we see that $\at(N_{\mathcal F})$ is sent to the $\mathcal F$-Atiyah class $\at_{\mathcal{F}}(N_{\mathcal F})$. Since  $\at_{\mathcal{F}}(N_{\mathcal F})$ is vanishing, see \eqref{eq:fatiyah},  the long exact sequence yields to 
\[
\at(N_{\mathcal F}) \in \h^1(X, \text{End}_{\mathcal O_X}(N_{\mathcal F})\otimes N_{\mathcal F}^*).
\]
This gives 
\[
c_{s}(N_{\mathcal F}) = \frac{1}{s!}\tilde{\sigma_s}(\at(N_{\mathcal F}), \dots, \at(N_{\mathcal F})) \in {\rm Im}[ \h^s(X, \wedge^s N_{\mathcal F}^*) \longrightarrow \h^s(X, \Omega_X^s)]
\]
for any $0<s\le q$.  Then the result follows from the fact that $N_{\mathcal F}$ has rank $q$. 
\endproof


\begin{remark}\rm In \cite[Corollaire 4.3.2]{Toen}, the author shows a Bott vanishing theorem for Chern classes in crystalline cohomology. In addition, as B. To\"en kindly explained to us, it is possible to prove, using the methods of \cite{Toen}, that a result similar to \cite[Corollaire 4.3.2]{Toen} also holds in the context of Hodge cohomology, which then removes the restriction on the characteristic appearing in our Theorem~\ref{thm:bottvanishing}.
\terminou
\end{remark}

For the first Chern class, the result holds over any ground field, regardless of its characteristic:

\begin{prop}\label{thm:bottvanishingC1}
Let $X$ be a smooth algebraic variety over an algebraically closed field $\field$ of characteristic $p\ge 0$. Let $\mathcal F$ be a regular foliation on $X$ of codimension $q$. Then $c_{1}(N_{\mathcal F}) \in \h^1(X,  N_{\mathcal F}^*)$. In particular, we have
\[
c_1(N_{\mathcal{F}})^{q+1} = 0 \quad \mbox{in}
\quad \hh^{q+1}(X,\Omega^{q+1}_{X}).
\]
\end{prop}

\proof
The same argument in the proof of Theorem~\ref{thm:bottvanishing} shows that 
\[
\at(N_{\mathcal F}) \in \h^1(X, \text{End}_{\mathcal O_X}(N_{\mathcal F})\otimes N_{\mathcal F}^*).
\]
Therefore, we get
\[
c_{1}(N_{\mathcal F}) = {\rm trace}(\at(N_{\mathcal F})) \in {\rm Im}[ \h^1(X,  N_{\mathcal F}^*) \longrightarrow \h^1(X, \Omega_X^1)]
\]
and the result follows. 
\endproof

Now, we consider the case where $\codim(\sing(\F))\ge q+2$:

\begin{thm}\label{prop:SingBB}
Let $X$ be a smooth algebraic variety of dimension $n>1$, over an algebraically closed field $\field$ of characteristic $p\ge 0$. If $\mathcal F$ is a codimension $q$ foliation on $X$ with singular set of codimension at least $q+2$, then 
\[
c_1(N_\mathcal F)^{q+1}=0 \quad \text{in}\quad \h^{q+1}(X, \Omega_X^{q+1}).
\] 
\end{thm}

\proof
Let $U:=X\setminus \sing (\F)$. By Proposition~\ref{thm:bottvanishingC1}, $c_1(N_{\F}|_U)$ lies in $\h^1(U, N_{\F}^*|_U)$ and $c_1(N_{\F}|_U)^{q+1}=0$ in $\h^{q+1}(U, \Omega_U^{q+1})$. Then the conclusion follows from the fact that the natural map
\[
\h^{q+1}(X, \Omega_X^{q+1}) \to \h^{q+1}(U, \Omega_U^{q+1})
\]
is injective. See \cite[Lemma 3.5]{DruelNumtrivial} and \cite[Lemma 6.4]{AraDru}.
\endproof

For surfaces, it follows from Bott vanishing theorem that $N_\F^2 = 0$ in $\field$. However,  it may happen $N_\F^2\neq 0$ in $\mathbb Z$,  see Remark~\ref{rmk:existencechr2}.  Note also that there are examples of codimension one regular foliations  on $\mathbb P^n_{\field}$, $n$ odd,  with $N_\F = \mathcal O_{\mathbb P^n_{\field}}(2)$, see Theorem~\ref{thm:regfolproj}.

\section{Bott ``vanishing'' in the Chow ring}\label{bott_chow}

We now want to profit from the particularities of  Geometry in positive characteristic and in this section we assume that  $p>0$. 

Let $X$ be a smooth variety over $\field$. We then let $X'$ be its Frobenius twist, i.e. as a scheme $X'$ is just $X$, but its structure of $\field$-scheme is obtained by means of the the composition 
\[
X \to\mathrm{Spec}\, \field\to\mathrm{Spec}\, \field,
\]
where the rightmost map is  induced by the {\it inverse} of the Frobenius. With these conventions, the identity map $\mathrm {id}:|X|\to |X'|$, between underlying topological spaces, together with the Frobenius morphism of rings $(-)^p:\mathcal O_{X'}\to \mathcal O_X$ give rise to the {\it geometric Frobenius morphism $F:X\to X'$}. This is $\field$-linear. In what follows, we shall       identify $\mathcal O_{X'}$ with a subring of $\mathcal O_X$. For an open subset $U$, we have $\mathcal O_{X'}(U)=\mathcal O_X(U)^p$.

\subsection{Basic material on $p$-closed foliations}

Let $\F$ be a regular foliation on $X$ and write $m=\mathrm{rank}\,T_\F$, and $q=\mathrm{rank}\,N_\F$. Let us in addition suppose that $\F$ is $p$-closed. 

For each open subset $U\subset X$, we define the  $\mathcal O_{X'}(U)$-module 
\[
\mathcal B (U):=\left\{a\in \mathcal O_X(U)\,:\,\begin{array}{c}\text{for all $V\subset U$ and all} \\ \text{   $D\in T_\F(V)$, we have $D(a|_V)=0$}\end{array}\right\}.\]
It is a simple matter to note that $\mathcal B$ is a sheaf of rings containing     $\mathcal O_{X'}$.
Let \[Y=(|X|,\mathcal B) \]  and let $\varphi:X\to Y$ be the morphism of ringed spaces naturally obtained from the inclusion $\mathcal B\subset \mathcal O_X$.  
The proof of the following   can easily be extracted from \cite[\S3]{seshadri60}. 

\begin{prop}\label{frobenius_theorem0} 
$\empty$
\begin{enumerate}
\item The ringed space $Y$ is a scheme and $\varphi$ is a morphism of schemes. 
\item The morphism   $\varphi:X\to Y$ is finite and  flat of rank   $p^m$. In particular, $\varphi$ has degree $p^m$. 
\item Each $P\in X$ possesses an affine and  open neighbourhood $U$ on which there is  a $p$-basis $\{x_1,\ldots,x_m\}$ of  $\mathcal O_X(U)$ over $\mathcal O_Y(U)$, i.e. 
\[\mathcal O_X(U)=\bigoplus_{0\le i_1,\ldots,i_m<p}\mathcal O_Y(U)\cdot x_1^{i_1}\cdots x_m^{i_m}.\]
\item  
The $\mathcal O_X$-module $\Omega_{X/Y}^1$, on $U$, is free and generated by $\mathrm dx_1,\ldots,\mathrm dx_m$.\end{enumerate}
\end{prop}
\begin{proof}[Elements of the proof.] Let us assume that $X=\mathrm{Spec}\,A$,   that $T_\F$ is free on $v_1,\ldots,v_m$,   
$T_X$ is free on $\{\partial /\partial x_i\}_{i=1}^n$, where $x_i\in A$, and $T_X/T_\F$ is also free.  
 Then    $B:=\mathcal B(X)$ is simply $\{a\in A\,:\,v_i(a)=0,\,\forall i\}$.   Write $v_j=\sum_{i=1}^na_{ij}\partial/\partial x_i$. We can assume that  the minor $\|a_{ij}\|_{1\le i,j\le m}$ is invertible and hence we can find a basis $\{D_i\}_{i=1}^m$ of $T_\F$ such that $D_i(x_j)=\delta_{ij}$ for $i,j\in\{1,\ldots,m\}$.   Let $x^{\boldsymbol i}:=x_1^{i_1}\cdots x_m^{i_m}$ for each ${\boldsymbol i}=(i_1,\ldots,i_m)$. Then, arguing  as in \cite[Lemma 3]{seshadri60}, we have  
\[A=\bigoplus_{\boldsymbol i\in\mathbb N_p^m}Bx^{\boldsymbol i},\]
where $\mathbb N_p=\{0,\ldots,p-1\}$.
 \end{proof}

Basic Commutative Algebra (cf. Theorem 23.7 and    Theorem 28.8 in \cite{Mat89}) now  implies: 

\begin{cor} The scheme  $Y$ is regular and hence  $\field$-smooth. \qed\end{cor}

Another pleasant way of presenting the situation at hand is thorough the 
\begin{cor}[The Frobenius theorem] 
\label{frobenius_theorem} Let $P\in X$ be given. 
There exists an open neighbourhood $U$ of $P$ and functions $x_1,\ldots,x_m\in\mathcal{O}_X(U)$, $y_1,\ldots,y_q\in\mathcal{O}_Y(U)$ such that 
\[
\mathrm dx_1,\ldots,\mathrm dx_m,\varphi^*(\mathrm d y_1),\ldots,\mathrm \varphi^*(\mathrm{d} y_q)
\]
is a basis of $\Omega^1_{X}|_U$. In addition, $\varphi^*(\mathrm d   y_1),\ldots,\varphi^*(\mathrm d y_q)$ is a basis of $N_\F^*|_U$ over $U$. If we  interpret $y_1,\ldots,y_q$ as elements in $\mathcal O_X(U)$ and use these to define the basis 
\[
\{\partial/\partial {x_i},\partial/\partial{y_j}\,:\,1\le i\le m,\,1\le j\le q\}
\]
of $T_X|_U$, then  the vector fields $\{\partial/\partial{x_i}\}_{i=1}^m$ form a basis of  $T_{\mathcal F}|_U$.
\end{cor}
\begin{proof}We employ the notations specified in the natural exact sequence 
\[
\varphi^*\Omega_{Y}^1\stackrel{\varphi^*}\longrightarrow \Omega^1_{X}\stackrel\rho\longrightarrow\Omega^1_{X/Y}\longrightarrow0.
\]
Let    $Q=\varphi(P)$. The following  are easily verified.
\begin{enumerate} 
\item[($\star$)] The space $\Omega_{X}^1(P)$ is spanned by elements of the form $\mathrm df(P)$, where  $f$ is defined in a neighbourhood of $P$ and vanishing on it. (The same claim is true if one replaces $X$ and  $P$ by $Y$ and $Q$. )
\end{enumerate}

Using ($\star$) and that $\Omega_{X/Y}^1$ is locally free of rank $m$,  we can find   an open neighbourhood  $U$ of $P$ and   $x_1,\ldots,x_m\in\mathcal{O}_X(U)$    such that 
\[\{\rho(dx_1)(P),\ldots,\rho(dx_m)(P)\}
\] 
is a basis of  $\Omega^1_{X/Y}(P)$.  (In particular $\{ dx_i(P)\}$ are linearly independent.) 
By the same argument, we can find   an open neighbourhood $V$ of $Q$ and  $y_1,\ldots,y_q\in\mathcal O_Y(V)$   such that 
\[
\{\varphi^*(dy_1)(P),\ldots,\varphi^*(dy_q)(P)\}
\]
is a basis for       $\mathrm{Im}\,\,\Omega^1_{Y}(Q)\otimes\field (P)\to\Omega^1_{X}(P)$. 
Since  $N_\F^* (P)$ is a $\field  (P)$-subspace of $\Omega_{X}^1(P)$ of dimension $q$ and contains each $d\varphi^\#(y_i)=\varphi^*(  dy_i)$, 
we conclude that 
\[
\{\varphi^*( dy_1)(P),\ldots,\varphi^*( dy_q)(P)\}
\]
is a basis of $N_\F^*(P)$. 
Consequently, $N_\F^*$ is free on  $\{\varphi^*( dy_i)\}_{i=1}^q$ near $P$. 
\end{proof}



We  insert  $\varphi$   into a commutative diagram 
\begin{equation}\label{diagram1}
\xymatrix{X\ar[r]^\varphi \ar[rd]_F& Y\ar[d]^\psi\\&X'.}
\end{equation}
Since $F$ and $\varphi$ are faithfully flat we have by \cite[p. 46]{Mat89}: 

\begin{cor}\label{cor:faithfullyflat}
The morphism  $\psi$ is also faithfully flat and finite, and has degree $p^q$.
\end{cor}

Let us end this section by reviewing some further local properties  $\varphi:X\to Y$ in order to draw an important parallel between foliations and ``local'' actions of the additive groups. 

Let $U=\mathrm{Spec}\,A$ be an affine open subset of $X$  where we can 
find 
\[
x_1,\ldots,x_m,y_1,\ldots,y_q\in  A
\]  
such that: 
\begin{itemize}
\item If $B=\mathcal B(U)$, then $A$ is a free $B$-module on $x_{1}^{i_1}\cdots x_m^{i_m}$, where $(i_1,\ldots,i_m)$ ranges over $\{0,\ldots,p-1\}^m$; 
\item The $\mathcal O_U$-module $T_\F|_{U}$ is free on $\{\partial/\partial x_i\}_{i=1}^m$;  
\item
The $\mathcal O_U$-module $T_U$ is free on $\{\partial/\partial x_i\}_{i=1}^m\cup\{\partial/\partial y_i\}_{i=1}^q$. 
\end{itemize}

To proceed, we require some   notations. The set $\{0,\ldots,p-1\}$ is denoted by $\mathbb N_p$, the derivations $\partial/\partial x_i$ are abbreviated to $\partial_i$. Given $\boldsymbol i=(i_1,\ldots,i_m)\in\mathbb N^m_p$, we put 
\[
x^{\boldsymbol i}=x_1^{i_1}\cdots x_m^{i_m},  \quad\partial^{\boldsymbol i}=\partial_{1}^{i_1}\cdots\partial_m^{i_m}\quad\text{and}\quad \partial^{[{\boldsymbol i}]}=\frac{1}{{\boldsymbol i}!}\partial^{\boldsymbol i}
\] 
where ${\boldsymbol i} = i_1!\cdots i_m!$.

Let $A':=A\otimes_BA$ be considered as an $A$-algebra via the unnamed morphism $a\mapsto a\otimes1$. As we learn from Grothendieck's theory of differential operators, it is important to consider   also the morphism \[\tau:A\longrightarrow A',\quad a\longmapsto1\otimes a,\] and to interpret it as a ``Taylor series''. This is done by writing
\[
\delta_i= \tau (x_i)-x_i= 1\otimes x_i -x_i\otimes 1
\] 
and by noting that the binomial formula  gives, for $j\in\mathbb N_p$,  
\[
\tau(x^j_i)=\sum_{\ell=0}^j\partial^{[\ell]}_i(x^j_i)\delta_i^{\ell}. 
\]
Hence, if we agree to abbreviate $\delta^{\boldsymbol \ell}=\delta_1^{\ell_1}\cdots\delta_m^{ \ell_m}$, we obtain
\[
\tau(x^{\boldsymbol j}) = \sum_{\boldsymbol\ell\in\mathbb N_p^m} \partial^{[\boldsymbol\ell]}(x^{\boldsymbol j})\cdot\delta^{\boldsymbol\ell},
\]
which extends, by $B$-linearity, to 
\begin{equation}\label{taucompatibility}
\tau(a) =  \sum_{\boldsymbol i\in\mathbb N_p^m} \partial^{[\boldsymbol i]}(a)\cdot\delta^{\boldsymbol i}.
\end{equation}

\subsection{A Bott ``vanishing theorem'' for $p$-closed foliations}
Let us maintain the notations and conventions used in diagram \eqref{diagram1}. Also, let us follow \cite{Fulton13} and write $Z^\ell(-)$ for the group of cycles of codimension $\ell$, $[-]$ for the cycle associated to a subscheme, and     $\varphi^*:Z^\ell(Y)\to Z^\ell(X)$ for the   flat pull-back. See   \cite[Ch. 1]{Fulton13}. It should be noted that, since $\varphi$ is a homeomorphism, then for each integral subscheme $W\subset Y$, we have 
\[\varphi^*[W]=\|W\|_\varphi\cdot[\varphi^{-1}(W)_{\text{red}}],\]
where $\|W\|_\varphi$ is the multiplicity of the irreducible closed subscheme $\varphi^{-1}(W)$, i.e. the length of the local ring of the generic point. This in particular says that the map  
$[W]\mapsto [\varphi^{-1}(W)_{\mathrm{red}}]$ defines a bijection $Z^\ell(Y)\to Z^\ell(X)$.

\begin{prop}
Let $\ell > q$ be an integer. Then
\[\varphi^*Z^\ell(Y)\subset  p^{\ell-q}Z^\ell(X).
\]
\end{prop}  

\begin{proof} Recall that $X'$  is  $(|X|, \mathcal{O}_X^p)$ with its structure morphism to $\mathrm{Spec}\, \field$ modified  and   $F:X\to X'$ is  the identity on the underlying topological spaces while $F^\#:\mathcal O_{X}\to \mathcal O_X$ is the Frobenius morphism.
It follows that  $[V']\mapsto [F^{-1}(V')_{\mathrm{red}}]$ defines an isomorphism $Z^\ell(X')\to Z^\ell(X)$. In addition, $F^*:Z^\ell (X')\to Z^\ell (X)$, is given by 
\[
[V']\longmapsto\|V'\|_F\cdot[F^{-1}(V')_{\mathrm{red}}], 
\]
 where we agree to denote by $\|V'\|_F$  the multiplicity of the irreducible closed subscheme $F^{-1}(V')$. These observations and accompanying notations are employed below.

The following facts shall lead   to a proof of the proposition. 
\begin{enumerate}
\item\label{1propproof}  If $V'\subset X'$ is a subvariety of codimension $\ell$, then $\|V'\|_F=p^\ell$, i.e. $F^*[V']=p^\ell\cdot [F^{-1}(V')_\mathrm{red}]$.
\item\label{2propproof}  If $W\subset Y$ is a subvariety  of codimension $\ell$, then $\|W\|_\varphi$  is a power of $p$. 
\item\label{3propproof}  Assume that $\ell>q$. Then, $\|W\|_\varphi$ is divisible by    $p^{ \ell-q}$.
\end{enumerate}

Let us prove \eqref{1propproof}. This is a consequence of a remark of Kunz and although  well-known, is  almost never patiently explained.  Let $V:=F^{-1}(V')$ and denote its generic point by $\xi$ so that 
$\|V'\|_F$ is, by definition, just $\mathrm{length}_{\mathcal O_{V,\xi}}(\mathcal O_{V,\xi})$. 
If $\mathcal I'\subset\mathcal O_X$ is the ideal of $V'$, then $\mathcal I'^{[p]}\subset\mathcal O_X$ is the ideal of $V$. Let   
$\xi$ be the generic point of  the space $|V'|=|V|$ so that $I':=\mathcal I_{\xi}'\subset\mathcal O_{\xi}$ is the maximal ideal of this $\ell$  dimensional regular local ring. 
Since we have 
\[
\mathrm{length}_{\mathcal O_\xi/I'^{[p]}}(\mathcal O_\xi/I'^{[p]})=\mathrm{length}_{\mathcal O_\xi}(\mathcal O_\xi/I'^{[p]}), 
\]
we now only need to employ Kunz's observation: for a regular local ring $R$ with maximal ideal $\mathfrak m$, the equality 
\[
\mathrm{length}_R(R/\mathfrak m^{[p]})=p^{\dim R}
\]
holds \cite{kunz69}.

The statement \eqref{2propproof}  is a direct consequence of \eqref{1propproof} and \cite[Lemma 1.7.1]{Fulton13}.

In order to prove \eqref{3propproof},   we observe that for any  $V'\subset X'$,   we have $\psi_*\psi^*[V']=p^q\cdot[V']$ according to \cite[1.7.4]{Fulton13}. (The proof of this employs \cite[A.1.3]{Fulton13}.) As a consequence, $\|V'\|_\psi\le p^q$.

Let  $[V']\in Z^\ell(X')$ be a prime cycle and let  $W=\psi^{-1}(V')_{\mathrm{red}}$ so that $\psi^*[V']=p^t\cdot[W]$ with $t\le q$. If we write $\varphi^*[W]=p^s\cdot[\varphi^{-1}(W)_{\mathrm{red}}]$, then  $\varphi^*\psi^*[V'] =p^{s+t}\cdot[\varphi^{-1}(W)_{\mathrm{red}}]$. On the other hand, $\varphi^*\psi^*[V']=F^*[V']=p^\ell\cdot[F^{-1}(V')_{\mathrm{red}}]$. Then $s=\ell-t\ge\ell- q$. \end{proof}

 {The previous proposition motivates the following definition.}

\begin{definition}
Let $\mathcal M$ be a quasi-coherent $\mathcal{O}_X$-module. An  $\mathcal F$-quotient     of  $\mathcal{M}$ is a pair $(\mathcal N,\sigma)$ consisting of a quasi-coherent $\mathcal{O}_{Y}$-module $\mathcal N$ and an isomorphism \[\sigma:\varphi^*\mathcal N\stackrel\sim\longrightarrow \mathcal M.\]

An $\mathcal O_X$-module having an $\mathcal F$-quotient will be said to {\bf descend} to $Y$. 
\end{definition}

 {Below, we let  $\mathrm{CH}^*(X)$ denote the Chow ring of $X$ and write $\mathrm{Chern}^*(\mathcal M)$ for the graded subring of $\mathrm{CH}^*(X)$ generated by the Chern classes of $\mathcal M$.}

\begin{cor}[Bott  vanishing modulo $p$]\label{21.08.2019--1} 
Let $\mathcal  M$ be a coherent $\mathcal O_X$-module on $X$  {having an $\F$-quotient}.  Then, for each $\ell>q$ we have 
\[\mathrm{Chern}^\ell(\mathcal M)\subset p^{\ell-q}\mathrm{CH}^\ell(X). \]
\end{cor}

The prime example of a module having an  $\mathcal F$-quotient  is the normal bundle of a regular $p$-closed foliation:

\begin{ex}[The Bott $\mathcal F$-quotient] \label{rmk:Bquotient}  \rm
We contend that  $N_{\mathcal F}^*$ (the annihilator of $T_\F$) descends to $Y$. In particular, its dual $N_{\mathcal F}$, descends also.

Let $U$ be an open subset of $X$ and let $\{y_i\}_{i=1}^q$ and $\{x_i\}_{i=1}^m$  be as in  Corollary \ref{frobenius_theorem}:   $N_\F^*|_U$ is free on $\{\mathrm dy_i\}_{i=1}^q$ and $T_{\mathcal F}|_U$ is free on $\{\partial/\partial x_i\}_{i=1}^m$. Let now $U'$ be an open subset overlapping $U$ where we have a similar situation with $N_\F^*|_{U'}$
being free on $\{\mathrm{d} y'_i\}_{i=1}^q$  and $T_{\mathcal F}|_{U'}$ on $\{\partial/\partial x_i'\}_{i=1}^m$.  We now write over $U\cap U'$: 
\[\begin{split}
\mathrm{d} y_j'&=\sum_{i=1}^m\fr{\partial y'_j}{\partial x_i}\mathrm{d} x_i+\sum_{i=1}^q\fr{\partial y'_j}{\partial y_i}\mathrm{d} y_i
\\&=\sum_{i=1}^q\fr{\partial y'_j}{\partial y_i}\mathrm{d} y_i;
\end{split}
\]
consequently, the $\GL_q(\mathcal{O}(U\cap U'))$-cocycle giving rise to $N_\F^*$ is \[\left(\fr{\partial y'_j}{\partial y_i}\right)_{1,i,j\le q}.\]
But, for every $h=1,\ldots,m$, we have 
\[
\fr\partial{\partial x_h}\fr{\partial y'_j}{\partial y_i}=\fr\partial{\partial y_i}\fr{\partial y'_j}{\partial x_h}=0. 
\]
This means that a 1-cocycle for $N_\F^*$ can be found with coefficients in $\mathcal{O}_Y$, and this implies that $N_\F^*$ descends to $Y$. Similarly, a cocycle for $N_\F$ is $\displaystyle\left(\fr{\partial y_i}{\partial y'_j}\right)_{1,i,j\le q}$;   we obtain a locally free $\mathcal B$-module   $\mathcal L$ by
\[\mathcal L|_U=\bigoplus_{i=1}^q\mathcal B|_U e_i^U\]
and an isomorphism $\varphi^*\mathcal L\simeq N_\F$ defined by $e_i^U\mapsto\partial /\partial y_i$.
\terminou
\end{ex}

 {It follows from  Example~\ref{rmk:Bquotient} that $N_{\F}$ has an $\F$-quotient, this yields to the following result, which is Theorem~\ref{thmintro:Bootva} of the introduction.  It  should be compared with \cite[Theorem $(*)''$, p.35]{Bott06}.} 

\begin{cor}[Bott vanishing]\label{thm:Bootva}
 {Let $\F$ be a regular $p$-closed foliation on a smooth variety $X$}. If $\mathrm{Chern}^*(N_{\mathcal F})$ stands for the graded subring of $\mathrm{CH}^*(X)$ generated by the Chern classes of $N_{\mathcal F}$, it follows that 
\[
\mathrm{Chern}^\ell(N_{\mathcal F})\subset p^{\ell-q}\cdot\mathrm{CH}^\ell(X)
\]
for each $\ell>q$. 
\end{cor}

In view of Theorem~\ref{thm:bottvanishing}, which applies for foliations that are not necessarily $p$-closed, the following question arises naturally.  

\begin{question}
Does Corollary~\ref{thm:Bootva} also hold  for foliations that are not  $p$-closed?
\end{question}

The reader acquainted with the theory of Cartier descent \cite[Theorem 5.1]{katz70} must have remarked that the notion of $\mathcal F$-quotient  
needs to  be   related to the notion of ``partial'' connection which appeared in a particular setting in   Section~\ref{sec:BottCon}.

\begin{definition}Let $\mathcal M$ be a quasi-coherent $\mathcal{O}_X$-module. A partial integrable $\mathcal F$-connection is an $\mathcal O_X$-linear morphism of sheaves  
\[
\nabla:T_\mathcal{F}\longrightarrow \mathcal {\rm End}_{\field}(\mathcal M),\quad v\longmapsto\nabla_v
\]
such that 
\begin{enumerate}\item for all sections $v$ of $T_\mathcal{F}$, $a$ of $\mathcal{O}_X$ and $s$ of $\mathcal M$, over an open subset, we have 
\[
\nabla_v(as)=v(a)\cdot s+a\nabla_v(s),
\] 
\item for sections $v$ and $w$ of $T_\mathcal{F}$, we have $\nabla_{[v,w]}=[\nabla_v,\nabla_w]$.\end{enumerate}
If, in addition,   $\nabla_{v^p}=(\nabla_v)^p$ for all sections $v$ of $T_\F$ over some open subset,  then we say that $\nabla$ is   $p$-closed.\end{definition}

Note that if $\nabla$ is an integrable $\mathcal F$-partial connection on $\mathcal M$, then $\nabla$ actually takes values in $\mathcal {\rm End}_{\mathcal O_Y}(\mathcal M)$. As with usual connections, a partial connection can also be defined by means of a $\field$-linear morphism of sheaves $\nabla:\mathcal M\to\mathcal M\otimes T_{\mathcal F}^*$ which satisfies Leibniz' rule. 

Define 
\[
\mathcal M^\nabla (U):=\left\{s\in \mathcal M (U)\,:\,\begin{array}{c}\text{for all $V\subset U$ and all} \\ \text{   $v\in T_\F(V)$, we have $\nabla_v(s|_V)=0$}\end{array}\right\}.\]
In this way, we obtain an $\mathcal O_Y$-module $\mathcal M^\nabla$; it is clear that $\mathcal M^\nabla$ can also be expressed as the kernel of $\nabla:\mathcal M\to \mathcal M\otimes T_{\mathcal F}^*$.

\begin{prop}[``Partial Cartier descent''] \label{partial_cartier_descent}
Let $\mathcal F$ be a regular $p$-closed foliation on a smooth algebraic variety $X$ and let $\varphi: X \to Y$ be the $\F$-quotient map.
\begin{enumerate}
\item\label{1thmpartial}  Let $\mathcal N$ be a quasi-coherent 
$\mathcal O_Y$-module. Then $\varphi^*\mathcal N=\mathcal O_X\otimes_{\mathcal O_Y}\mathcal N$ carries an integrable $p$-closed  partial $\mathcal F$-connection $\nabla^{\mathrm {can}}$ defined by letting each $v\in T_\mathcal{F}$ act as 
\[
v\otimes\id:\mathcal O_X\otimes_{\mathcal O_Y}\mathcal N\longrightarrow \mathcal O_X\otimes_{\mathcal O_Y}\mathcal N
\]
\item\label{2thmpartial} The preceding construction gives rise to a functor 
\[\varphi^*:\left\{\begin{array}{c}\text{quasi-coherent} \\ \text{ $\mathcal O_Y$-modules}\end{array}\right\}\longrightarrow\left\{\begin{array}{c}\text{partial integrable $p$-closed} \\ \text{ $\mathcal F$-connections}\end{array}\right\},\]
which is an equivalence. Associating  to each integrable  $p$-closed  partial $\mathcal F$-connection $(\mathcal M,\nabla)$  the quasi-coherent $\mathcal O_Y$-module $\mathcal M^\nabla$ gives rise to an inverse equivalence.
\item\label{3thmpartial} If $\nabla^{\rm B}$ is the Bott connection on $N_\F$ constructed in Section~\ref{sec:BottCon}, then it coincides with the canonical connection on $N_\F$ obtained through the construction in \eqref{1thmpartial} applied to the descended module of   Example~\ref{rmk:Bquotient}. 
\end{enumerate}
\end{prop}

\begin{proof} Item \eqref{1thmpartial} is quite simple. The proof of  \eqref{2thmpartial} is a  direct consequence of faithfully flat descent for modules \cite[Theorem II.6.4, p.51]{knus-ojanguren}. For the convenience of the reader, we render explicit some details. 
  
Let $(\mathcal M,\nabla)$ be a partial $\mathcal F$-connection. Let $U=\mathrm{Spec}\,A$ be an open and affine subset of $X$ where we can 
find 
\[x_1,\ldots,x_m,y_1,\ldots,y_q\in\mathcal O(U)\]  such that: 
\begin{itemize}\item If $B=\mathcal B(U)$, then $A$ is a free $B$-module on $x_{1}^{i_1}\cdots x_m^{i_m}$, where $(i_1,\ldots,i_m)$ ranges over $\{0,\ldots,p-1\}^m$; 
\item The $\mathcal O_U$-module $T_\F|_{U}$ is free on $\{\partial/\partial x_i\}_{i=1}^m$;  \item
The $\mathcal O_U$-module $T_U$ is free on $\{\partial/\partial x_i\}_{i=1}^m\cup\{\partial/\partial y_i\}_{i=1}^q$. 
\end{itemize}
That we can set up such a situation is a consequence of Proposition \ref{frobenius_theorem0} and Corollary~\ref{frobenius_theorem}.
Let
  $M=\mathcal M(U)$. 
We want to show that $M$ descends  to a $B$-module, i.e., that $M\simeq A\otimes_BN$ for a certain $B$-module $N$; that  $N=\mathcal M^\nabla(U)$ will then follow from the general theory \cite[II, Theorem 3.2, p.36]{knus-ojanguren}. 

Following \cite{knus-ojanguren} (see also \cite[Remark 2.21]{milne80}), we need to produce a {\it descent data} for the $A$-module $M$. For this,  let us now adopt the notations employed at the discussion made after Corollary~\ref{cor:faithfullyflat} and add a few more. Put  $D_i=\nabla_{\partial_i}$; these are commuting  $B$-linear endomorphisms of $M$. For $\boldsymbol  i=(i_1,\ldots,i_m)\in\mathbb N_p^m$,   write    
\[
D^{\boldsymbol  i}=D_1^{i_1}\cdots D_m^{i_m},\quad  D^{[\boldsymbol  i]}=  \frac{1}{\boldsymbol  i!}D^{\boldsymbol  i}, \quad\text{and}\quad (\!(\boldsymbol i,\boldsymbol j)\!)=\binom{i_1+j_1}{i_1}\cdots\binom{i_m+j_m}{i_m}.  
\]
This last element is taken to be in $\field$.
Given   $s\in M$, let  
\[
\Phi(s)=\sum_{\boldsymbol  i\in\mathbb N^m_p} D^{[\boldsymbol  i]}(s)\otimes\delta^{\boldsymbol  i}\in M\otimes_AA'.
\]  
Because of eq. \eqref{taucompatibility}, we have $\tau(a)\Phi(s)=\Phi(as)$, so that we obtain an $A'$-linear map  
\[\Phi':
A'\otimes_{\tau,A}M\longrightarrow M\otimes_AA',\quad a'\otimes s\longmapsto \sum_{\boldsymbol i\in\mathbb N^m_p} D^{[\boldsymbol  i]}(s)\otimes a'\delta^{\boldsymbol  i}.
\] 
If we for a moment regard $A$ as an $A'$-algebra via  the multiplication    morphism $a'\otimes b'\mapsto a'b'$, then   $\mathrm{id}_A\otimes_{A'}  \Phi'=\mathrm{id}_M$.   
 
 Let now $A''=A\otimes_BA\otimes_BA$, so that we have morphisms of rings 
\[
d_{12}:A'\longrightarrow A'',\quad a'\otimes b'\longmapsto a'\otimes b'\otimes1,
\]  
\[d_{23}:A'\longrightarrow A'',\quad a'\otimes b'\longmapsto 1\otimes a'\otimes b' 
\]
and 
\[d_{13}:A'\longrightarrow A'',\quad a'\otimes b'\longmapsto a'\otimes1\otimes b'.
\] Let $d_1:A\to A'$ be  the map $a\mapsto a\otimes1$. Then 
   $d_{12}d_1 =d_{13}d_1$, $d_{12}\tau=d_{23}d_1$ and   $d_{13}\tau =d_{23}\tau$. Let us also adopt the convention that $\tau^*M$, $d_1^*M$, etc, denote the tensor products $M\otimes_{\tau,A}A'$, $M\otimes_{d_1,A}A'$, etc.
   The cocycle condition for the descent is now equivalent to the commutativity of the following   hexagon of maps of $A''$-modules 
\begin{equation}\label{cocycle0}
\xymatrix{
&d_{23}^*\tau^*M\ar[dr]^{d_{23}^*\Phi'}\ar@{=}[dl]&
\\
d_{13}^*\tau^*M\ar[d]_{d_{13}^*\Phi'}&&d_{23}^*d_1^*M\ar@{=}[d]
\\
d_{13}^*d_1^*M\ar@{=}[rd]&&d_{12}^*\tau^*M\ar[dl]^{d_{12}^*\Phi'}
\\
&d_{12}^*d_1^*M.&
}
\end{equation}
 Now,   $d_{13}(\delta_i)=1\otimes\delta_i+\delta_i\otimes1$, and writing 
\[
(1\otimes\delta+\delta\otimes1)^{\boldsymbol h}=\prod_{i=1}^m(1\otimes\delta_i+\delta_i\otimes1)^{h_i},
\]
commutativity of \eqref{cocycle0}  amounts to showing that  
\begin{equation}\label{cocycle1}
\sum_{\boldsymbol h\in\mathbb N_p^m} (1\otimes\delta+\delta\otimes1)^{\boldsymbol h}\underset{d_{12}}\otimes \left( D^{[\boldsymbol h]}(s)\underset{d_{1}}\otimes 1\right) = \sum_{\boldsymbol i,\boldsymbol j} (\delta\otimes1)^{\boldsymbol j}\cdot (1\otimes\delta)^{\boldsymbol i}\underset{d_{12}}\otimes \left(D^{[\boldsymbol j]}  D^{[\boldsymbol i]} (s)\underset{d_{1}}\otimes 1\right)
\end{equation}
for each $s\in M$.
The binomial formula assures that for $\boldsymbol h\in\mathbb N_p^m$, we have  \[
(1\otimes\delta+\delta\otimes1)^{\boldsymbol h}=\sum_{\boldsymbol i+\boldsymbol j=\boldsymbol h} (\!(\boldsymbol i,\boldsymbol j)\!)(\delta\otimes1)^{\boldsymbol j}(1\otimes\delta)^{\boldsymbol i}. 
\]
From the fact that $D_i^h=0$ for $h\ge p$, we have 
\[
D^{[\boldsymbol j]}D^{[\boldsymbol i]}=(\!(\boldsymbol i,\boldsymbol j)\!)D^{[\boldsymbol i+\boldsymbol j]}, \]
which proves   that  eq. \eqref{cocycle1} holds true and hence diagram \eqref{cocycle0} commutes. It is therefore possible to apply \cite[II, Theorem 3.2]{knus-ojanguren}: Let 
\[\begin{split}
N&=\{s\in M\,:\, \Phi(s)=s\otimes1\} 
\\& =\bigcap_i\{s\in M\,:\,D_i(s)=0\};
\end{split}\]
then $N\otimes_BA\to M$ is an isomorphism. This concludes the proof of \eqref{2thmpartial} of the statement.

The proof of \eqref{3thmpartial} is straightforward. 
\end{proof}

\begin{remark}\rm A result along the lines of Proposition  \ref{partial_cartier_descent} appeared in \cite{borelli-moreira-salomao25}. 
\terminou
\end{remark}

\bibliographystyle{amsplain}
\bibliography{annot}

\end{document}